\documentclass[11pt]{amsart}

\usepackage{amsfonts,amsmath,amscd,amssymb,amsbsy,amsthm,amstext,amsopn}
\usepackage{fullpage,mathrsfs,subfigure}
\usepackage{pstricks,pst-node,pst-text,pst-tree,pst-plot,pstricks-add}
\usepackage{tikz} 
\usepackage{stmaryrd}  
\usepackage{extarrows}
\usepackage[hmargin=3cm,vmargin=3.5cm]{geometry}
\usepackage{graphicx}
\usepackage{mathdots}
\usepackage{color}
\usepackage{url}
\usepackage[capitalize]{cleveref}
\usepackage{longtable}
\usepackage{pdflscape}
\usepackage{rotating}
\usepackage{array}

\numberwithin{equation}{section}
\newtheorem{theorem}{Theorem}[section]
\newtheorem{proposition}[theorem]{Proposition}
\newtheorem{lemma}[theorem]{Lemma}

\newtheorem{corollary}[theorem]{Corollary}
\newtheorem{conjecture}[theorem]{Conjecture}
\newtheorem{acknowledgements}[theorem]{Acknowledgments}

\theoremstyle{definition}
\newtheorem{definition}[theorem]{Definition}
\newtheorem{example}[theorem]{Example}

\theoremstyle{remark}
\newtheorem{remark}[theorem]{Remark}
\newtheorem{remarks}[theorem]{Remarks}

\newcommand{\CC}{\ensuremath{\mathbb{C}}}
\newcommand{\NN}{\ensuremath{\mathbb{N}}} 
\newcommand{\PP}{\ensuremath{\mathbb{P}}} 
 
\newcommand{\RR}{\ensuremath{\mathbb{R}}} 
\newcommand{\ZZ}{\ensuremath{\mathbb{Z}}}

\renewcommand{\div}{\operatorname{div}}

\newcommand{\git}{\ensuremath{/\!\!/\!}}

\newcommand{\inc}{\operatorname{inc}}

\newcommand{\supp}{\operatorname{supp}}

\newcommand{\tot}{\operatorname{tot}}

 \newcommand{\Cl}{\operatorname{Cl}}  

\newcommand{\Cone}{\operatorname{Cone}}

\newcommand{\End}{\operatorname{End}}
 
\newcommand{\GL}{\operatorname{GL}}  
  
\newcommand{\Hom}{\operatorname{Hom}} 

\newcommand{\im}{\operatorname{im}}

\newcommand{\Pic}{\operatorname{Pic}}
\newcommand{\pic}{\operatorname{pic}}
\newcommand{\Proj}{\operatorname{Proj}}

\newcommand{\Wt}{\operatorname{Wt}}

\newcommand{\Bcal}{\mathcal{B}}
\newcommand{\Ccal}{\mathcal{C}}
\newcommand{\Dcal}{\mathcal{D}}
\newcommand{\Ecal}{\mathcal{E}}
\newcommand{\Fcal}{\mathcal{F}}

\newcommand{\Mcal}{\mathcal{M}}
\newcommand{\Ocal}{\mathcal{O}}

\newcommand{\Tcal}{\mathcal{T}}
\newcommand{\Lcal}{\mathcal{L}}
\newcommand{\Scal}{\mathcal{S}}
\newcommand{\Dx}{\mathcal{D}^b(X)}
\newcommand{\mkD}{\mathfrak{D}}

\newcommand\Qcoh{\operatorname{Qcoh}}

\begin{document}
\title{Tilting Bundles on Toric Fano Fourfolds}
\author{Nathan Prabhu-Naik}
\address{Nathan Prabhu-Naik\\Department of Mathematical Sciences\\University of Bath\\Claverton Down\\Bath\\Somerset\\BA2 7AY\\United Kingdom}
\email{N.Prabhu-Naik@bath.ac.uk}
\date{\today}
\begin{abstract}
This paper constructs tilting bundles obtained from full strong exceptional collections of line bundles on all smooth $4$-dimensional toric Fano varieties. The tilting bundles lead to a large class of explicit Calabi-Yau-$5$ algebras, obtained as the corresponding rolled-up helix algebra. A database of the full strong exceptional collections can be found in the package \emph{QuiversToricVarieties} for the computer algebra system \emph{Macaulay2}.
\end{abstract}

\maketitle

\section{Introduction}
\noindent Let $X$ be a smooth variety over $\CC$ and let $\Dx$ be the bounded derived category of coherent sheaves on $X$. A tilting object $\Tcal \in \Dx$ is an object such that $\Hom^i(\Tcal,\Tcal) = 0$ for $i \neq 0$ and $\Tcal$ generates $\Dx$. If such a $\Tcal$ exists, then tilting theory provides an equivalence of triangulated categories between $\Dx$ and the bounded derived category $\Dcal^b(A)$ of finitely generated right modules over the algebra $A = \End (\Tcal)$ via the adjoint functors

\begin{center}
\[
\begin{tikzpicture} [bend angle=20, looseness=1]
\node (C1) at (0,0)  {$\Dx$};
\node (C2) at (4,0)  {$\Dcal^b(A)$};
\draw [->,bend left] (C1) to node[above]  {\scriptsize{$\mathbf{R}\Hom_{X}(\Tcal,-)$}} (C2);
\draw [->,bend left] (C2) to node[below]  {\scriptsize{$ (-)\stackrel{\mathbb{L}}{\otimes}_{A}\Tcal$}} (C1);
\end{tikzpicture}
\]
\end{center}

\noindent If $X$ is also projective then one can use a full strong exceptional collection to obtain a tilting object; a full strong exceptional collection of sheaves $\{E_i \}_{i \in I}$ defines a tilting sheaf $\Tcal := \bigoplus_{i \in I} E_i$ and conversely, the non-isomorphic summands in a tilting sheaf determine a full strong exceptional collection. The classical example of a tilting sheaf was provided by Be\u\i linson \cite{Beil}, who showed that $\Ocal \oplus \Ocal(1) \oplus \ldots \oplus \Ocal(n)$ is a tilting bundle for $\mathbb{P}^n$. 

The combinatorial nature of toric varieties makes it feasible to check whether a collection of line bundles on a smooth projective toric variety is full strong exceptional, in which case one can construct the resulting endomorphism algebra explicitly. Smooth toric Fano varieties are of particular interest; there are a finite number of these varieties in each dimension and they have been classified in dimension $3$ by Watanabe--Watanabe and Batyrev \cite{WaWa,Bat2}, dimension $4$ by Batyrev and Sato \cite{Bat1,Sato}, dimension $5$ by Kreuzer--Nill \cite{KrNi}, whilst \O bro \cite{Obro} provided a general classification algorithm. King \cite{King} has exhibited full strong exceptional collections of line bundles for the $5$ smooth toric Fano surfaces, and by building on work by Bondal \cite{Bond1}, Costa--Mir\'{o}-Roig \cite{CoMR1} and Bernardi--Tirabassi \cite{BeTi}, Uehara \cite{Ueha} provided full strong exceptional collections of line bundles for the $18$ smooth toric Fano threefolds. The main theorem of this paper is as follows:
\begin{theorem}
Let $X$ be one of the $124$ smooth toric Fano fourfolds. Then one can explicitly construct a full strong exceptional collection of line bundles on $X$, a database of which is contained in the computer package \emph{QuiversToricVarieties} \cite{Prna1} for \emph{Macaulay2} \cite{M2}.
\end{theorem}

\noindent In addition to low-dimensional smooth toric Fano varieties, other classes of toric varieties have been shown to have full strong exceptional collections of line bundles -- for example, see \cite{CoMR1,DLMi,LaMi}. Kawamata \cite{Kawa} showed that every smooth toric Deligne-Mumford stack has a full exceptional collection of sheaves, but we note that these collections are not shown to be strong, nor do they consist of bundles. It is important to note that the existence of full strong exceptional collections of line bundles is rare; Hille--Perling  \cite{HiPe} constructed smooth toric surfaces that do not have such collections. Even when only considering smooth toric Fano varieties, there exist examples in dimensions $\geq 419$ that do not have full strong exceptional collections of line bundles, as demonstrated by Efimov \cite{Efim}.

Using the tilting bundles constructed on smooth toric Fano varieties as above, we construct a tilting bundle on the total space of its canonical bundle $\omega_X$:   

\begin{theorem}
Let $X$ be an $n$-dimensional smooth toric Fano variety for $n \leq 4$, $\Lcal = \{L_0, \ldots , L_r\}$ be the full strong exceptional collection on $X$ from the database and $\pi \colon Y:= \tot(\omega_X) \rightarrow X$ be the bundle map. Then $Y$ has a tilting bundle that decomposes as a sum of line bundles, given by $\bigoplus_{L_i \in \Lcal} \pi^*(L_i)$.
\end{theorem} 

To show that a given collection on a toric Fano variety $X$ is strong exceptional, we utilise the construction of the non-vanishing cohomology cones in the Picard lattice for $X$ as introduced by Eisenbud--Musta{\c{t}}{\v{a}}--Stillman \cite{EiMuSt}. The strong exceptional condition then becomes a computational exercise, which has been implemented into \emph{QuiversToricVarieties} \cite{Prna1}. 

The procedure to check whether a given strong exceptional collection $\Lcal$ on $X$ generates $\Dx$ is less straightforward. We use one of two methods to show that $\Lcal$ is full, the first of which is similar to the method used by Uehara on the toric Fano threefolds \cite{Ueha}. This approach uses the Frobenius pushforward to obtain a set of line bundles that are known to generate $\Dx$, and then we show that $\Lcal$ generates this set by using exact sequences of line bundles. The second method uses the line bundles to obtain a resolution of  $\Ocal_\Delta$, the structure sheaf of the diagonal. We achieve this by utilising the idea of a \emph{toric cell complex}, introduced by Craw--Quintero-V\'{e}lez \cite{CrQV} to guess what a minimal projective $A,A$-bimodule resolution of the endomorphism algebra $A = \End(\bigoplus_i L_i)$ is. In particular, by considering the pullback of $\Lcal$ on $Y$, we obtain a CY$5$ algebra for which we know the $0^{th}$, $1^{st}$ and $2^{nd}$ terms of its minimal projective bimodule resolution. The natural duality inherent in the CY$5$ algebra then gives clues as to what the $3^{rd}$, $4^{th}$ and $5^{th}$ terms are. We sheafify the result, restrict to $X$ and then check that the resulting exact sequence of sheaves $S^\bullet$ is indeed a resolution of $\Ocal_\Delta$ by using quiver moduli. Our calculations lead us to the following conjecture: 

\begin{conjecture}
Let $X$ be a smooth toric Fano threefold or one of the $88$ smooth toric Fano fourfolds such that the given full strong exceptional collection $\Lcal$ in the database \cite{Prna1} has a corresponding exact sequence of sheaves $S^\bullet \in \Dcal^b(X \times X)$. Let $B$ denote the rolled up helix algebra of $A$. Then the toric cell complex of $B$ exists and is supported on a real five-dimensional torus. Moreover,
\begin{itemize}
\item the cellular resolution exists in the sense of \cite{CrQV}, thereby producing the minimal projective bimodule resolution of $B$;
\item the object $S^\bullet$ is quasi-isomorphic to $\Tcal \stackrel{\mathbf{L}}{\boxtimes}_A \Tcal^\vee \in \Dcal^b(X \times X)$, where $\Tcal := \bigoplus_{L \in \Lcal} L$ and $\Tcal \stackrel{\mathbf{L}}{\boxtimes}_A \Tcal^\vee$ is the exterior tensor product over $A$ of $\Tcal$ and $\Tcal^\vee$. 
\end{itemize}
\end{conjecture}  

By considering the birational geometry of the toric Fano fourfolds and choosing collections $\Lcal$ from a special set of line bundles as Uehara did on the toric Fano threefolds, the pushforward of $\Lcal$ onto a torus-invariant divisorial contraction is automatically full (see Proposition \ref{prop:picMapFull}). Using this result, we obtain full strong exceptional collections on many of the toric Fano fourfolds from the the pushforward of collections on the birationally maximal examples. A database of the full strong exceptional collections on $n$-dimensional smooth toric Fano varieties, $1 \leq n \leq 4$, as well as many of the computational tools used in the proofs of the theorems above, is contained in the \emph{QuiversToricVarieties} package.  

\begin{acknowledgements}
The author wishes to thank his Ph.D. supervisor Alastair Craw for suggesting the problem and for guidance during the research. He also wishes to thank Greg Smith who developed the foundations of the \emph{Macaulay2} package \emph{QuiversToricVarieties}, as well as Alastair King, Jesus Tapia Amador, Ziyu Zhang and Matthew Pressland for useful conversations. The author's Ph.D. research is supported by an EPSRC studentship.
\end{acknowledgements}

\subsection*{Notation}
We work throughout over the field $\CC$ of complex numbers.
\section{Background}\label{1stSect}
\subsection{Toric Geometry}\label{subsect:ToricGeometry}
For $n \geq 0$, let $M$ be a rank $n$ lattice and define $N := \Hom_\ZZ (M, \ZZ)$ to be its dual lattice. The realification $M_\RR := M \otimes_\ZZ \RR$ and $N_\RR := N \otimes_\ZZ \RR$ are real vector spaces which contain the underlying lattices and there exists a natural pairing $\langle \ , \ \rangle \colon M_\RR \times N_\RR \rightarrow \RR$. The convex hull of a finite set of lattice points in $M$ defines a \emph{lattice polytope} $P \subset M_\RR$ and its \emph{facets} are the codimension $1$ faces of $P$. We will assume that the dimension of $P$ is equal to the rank of $M$. The theory of polytopes (see for example Cox--Little--Schenck \cite{CoLiSc}) states that every facet $F$ in $P$ has an inward-pointing normal $n_F$ that defines a one-dimensional cone $\{\lambda n_F \mid \lambda \in \RR_{\geq 0}\}$ in $N_\RR$. The cone is rational as $P$ is a lattice polytope, so it has a unique generator $u_F \in N$. Given $a \in \RR$ and a non-zero vector $u \in N_\RR$ we have the \emph{affine hyperplane} $H_{u,a} := \{ m \in M_\RR \ \vert \ \langle m,u \rangle = a \}$ and the \emph{closed half-space} $H^+_{u,a} := \{ m \in M_\RR \ \vert \ \langle m,u \rangle \geq a \}$. As $P$ is full-dimensional, each facet $F$ defines a unique number $a_F \in \RR$ such that $F = H_{u_F , a_F} \cap P$ and $P \subset H^+_{u_F , a_F}$. We can therefore use the generators to completely describe $P$, using its unique \emph{facet presentation}  
\begin{center}
$P = \{ m \in M_\RR \ \vert \ \langle m , u_F \rangle \geq -a_F \text{ for all facets } F \text{ in } P \}$.
\end{center}
If the origin of $M_\RR$ is an interior lattice point of $P$, then $P$ has a dual polytope $P^\circ$ which is defined to be the convex hull of the generators for the inward-pointing normal rays of $P$:
\begin{center}
$P^\circ = \text{Conv}(u_F \ | \ F \text{ is a facet of } P) \subset N_\RR$
\end{center}
The dual polytope determines a \emph{fan} in $N_\RR$:
\begin{definition}
For $1 \leq m \leq n$, let $F$ be an ($m-1$)-dimensional face of $P^ \circ$ with vertices $\lbrace u_{i_1}, \ldots , u_{i_m} \rbrace$. The $m$-dimensional cone $\sigma(F)$ is given by
\begin{equation}
\sigma(F) := \{ \lambda_1 u_{i_1} + \ldots + \lambda_m u_{i_m} \in N \ | \ \lambda_j \geq 0, 1 \leq j \leq m \}. 
\end{equation}
The fan $\Sigma(P^\circ) \subset N$ associated to $P^\circ$ is given by the collection of cones
\begin{equation}
\Sigma = \Sigma(P^\circ) := \{0 \} \cup \{ \sigma (F) \}_{F \subsetneq P^\circ} 
\end{equation}
where $F$ runs over all proper faces of $P^\circ$.  
\end{definition}

Let $\Sigma(k)$ denote the set of $k$-dimensional cones in a fan $\Sigma$. The \emph{rays} of $\Sigma$ are the one-dimensional cones which, by construction, are generated by the vectors $u_F$ for each facet $F \subset P$. We can use the ray generators to define \emph{primitive collections} and \emph{primitive relations} which describe $\Sigma$ combinatorially.    
\begin{definition}
A subset $\mathcal{P} = \{u_{i_1}, \ldots , u_{i_k} \}$ of the set of ray generators $\mathcal{V} = \{u_F \in N \mid F \text{ is a facet of } P\}$ for $\Sigma$ is a \emph{primitive collection} if

\begin{enumerate}
\item[(i)] there does not exist a cone in $\Sigma$ that contains every element of $\mathcal{P}$ and
\item[(ii)] any proper subset of $\mathcal{P}$ is contained in some cone of $\Sigma$.
\end{enumerate}
The integral element $s(\mathcal{P}) = u_{i_1} + \ldots + u_{i_k}$ is contained in some cone $\sigma \subset \Sigma$ with ray generators $\{ u_{j_1}, \ldots, u_{j_m} \}$ and so can be uniquely written as a sum of the generators:
\begin{equation*}
s(\mathcal{P}) = c_1 u_{j_1} + \ldots + c_m u_{j_m}, c_i > 0, c_i \in \ZZ.
\end{equation*}   
The linear relation
\begin{equation*}
u_{i_1} + \ldots + u_{i_m} - (c_1 u_{j_1} + \ldots + c_m u_{j_m}) = 0
\end{equation*} 
between the ray generators of $\Sigma$ is the \emph{primitive relation} associated to the primitive collection $\mathcal{P}$.
\end{definition}
\noindent Note that primitive relations and collections can alternatively be defined for polytopes.

Toric geometry (see e.g.\ Fulton \cite{Fult} or Cox--Little--Schenck \cite{CoLiSc}) associates to each fan $\Sigma$ a toric variety $X_\Sigma$ such that $M$ is the character lattice of the dense torus $T \cong \Hom_\ZZ (M,\CC^*)$ in $X_\Sigma$. If the fan $\Sigma$ is constructed from a polytope $P \subset M_\RR$ as above, then we use $X_P$ to denote the corresponding toric variety. For a cone $\sigma \subset \Sigma$ define $\sigma^\perp := \{ m \in M \ \vert \ \langle m,u \rangle = 0 \text{ for all } u \in \sigma \}$. The Orbit-Cone Correspondence implies that for each ray $\rho \in \Sigma(1)$, the closure of the $T$-orbit $\Hom_\ZZ (\rho^\perp \cap M, \CC^* )$ is a torus-invariant divisor $D_\rho$ in $X_\Sigma$. The lattice of torus-invariant divisors in $X_\Sigma$ will therefore be denoted $\ZZ^{\Sigma(1)}$ and the class group will be denoted $\Cl(X_\Sigma)$. We now have an exact sequence    
\begin{equation}\label{exseq}
 \begin{CD}   
    0@>>> M @>>> \ZZ^{\Sigma(1)} @>{\deg}>> \Cl(X_\Sigma)@>>> 0,
\end{CD}
\end{equation}
where the injective map is $m \mapsto \Sigma_{\rho \in \Sigma(1)} \langle m, u_\rho \rangle D_\rho$ and the map \emph{deg} sends the divisor $D$ to the rank one reflexive sheaf $\Ocal_{X_\Sigma} (D)$. If $X_\Sigma$ is smooth then every rank one reflexive sheaf is invertible, so the class group $\Cl (X_{\Sigma})$ is isomorphic to the Picard group $\Pic(X_\Sigma)$. Note that $X_\Sigma$ is smooth if and only if for every cone $\sigma \subset \Sigma$, the minimal generators for $\sigma$ form part of a $\ZZ$-basis for $N$.

The \emph{Cox ring} for $X_\Sigma$ is the semigroup ring $S_X := \CC [x_\rho \mid \rho \in \Sigma(1)]$ of $\NN^{\Sigma(1)} \subset \ZZ^{\Sigma(1)}$. The map \emph{deg} induces a $\Cl(X_\Sigma)$-grading on $S_X$, where the degree of a monomial $\prod_{\rho \in \Sigma(1)} x_\rho^{a_\rho} \in S_X$ is $\Ocal_{X_\Sigma} (\sum_{\rho \in \Sigma(1)} a_\rho D_\rho )$. For a divisor $D$ and integer $l \in \ZZ$, we let $\CC [x_\rho \mid \rho \in \Sigma(1)]_{lD}$ denote the $\Ocal_X(lD)-$graded piece. Cox \cite[Proposition 3.1]{Cox} defines an exact functor from the category of $\Cl(X_\Sigma)$-graded $S_X$-modules to the category of quasi-coherent sheaves on $X_\Sigma$:
\begin{equation}\label{eq:FunctorGradedSModulesQCohSheaves}
\{ \Cl(X_\Sigma) \text{-graded } S_X \text{-modules} \} \longrightarrow \Qcoh (X_\Sigma) \ \colon M \mapsto \widetilde{M}. 
\end{equation}   
When $X_\Sigma$ is smooth, every coherent sheaf on $X_\Sigma$ is isomorphic to $\widetilde{M}$ for some finitely generated $\Pic(X_\Sigma)$-graded $S_X$-module $M$ and two finitely generated $\Pic (X_\Sigma )$-graded $S_X$-modules determine isomorphic coherent sheaves if and only if they agree up to saturation by the \emph{irrelevant ideal} $B :=  \left( \prod_{\rho \nsubseteq \sigma} x_\rho \ \vert \ \sigma \in \Sigma \right)$ \cite[Propositions 3.3, 3.5]{Cox}. 

Morphisms between two toric varieties can be described by maps between their associated fans that preserve the cone structure. For example, consider the blowup of a torus-invariant subvariety. By the Orbit-Cone Correspondence, a $k$-codimensional torus-invariant subvariety of a toric variety $X_\Sigma$ corresponds to a cone $\sigma \in \Sigma(k)$, and the blowup of this subvariety is the toric variety whose fan is the \emph{star subdivision} of $\sigma$. The star subdivision is a combinatorial process that introduces a new ray $x$ with generator $u_\sigma = \sum_{\rho \in \sigma(1)} u_\rho$ and replaces $\Sigma$ with 
\begin{equation}
\Sigma^*_{\sigma,x} := \{ \tau \subset \Sigma \ \vert \ \sigma \nsubseteq \tau \} \cup \bigcup_{\sigma \subseteq \tau} \Sigma^*_\tau (\sigma) 
\end{equation}     
where $\Sigma^*_\tau (\sigma) := \{ \text{Cone}(A) \ \vert \ A \subseteq \{u_\sigma \} \cup \tau(1), \sigma(1) \nsubseteq A \}$. The map between fans $\Sigma^*_{\sigma,x} \rightarrow \Sigma$ determines the blowup 
\begin{equation}\label{blowup}
\varphi \colon X_1 := X_{\Sigma^*_{\sigma,x}} \longrightarrow X_2:= X_\Sigma
\end{equation}
and induces a commutative diagram between the corresponding exact sequences \eqref{exseq} for the varieties:
\begin{equation}\label{gamma}
 \begin{CD}   
    0@>>> M @>>> \ZZ^{\Sigma^*_{\sigma,x}(1)} @>{\deg_{X_1}}>> \Cl(X_1)@>>> 0 \\
     @.   @|            @V{\beta}VV   @V{\gamma}VV      @.          \\
0 @>>> M @>>> \ZZ^{\Sigma(1)}  @>{\deg_{X_2}}>> \Cl(X_2)   @>>> 0 \\
\end{CD}
\end{equation}
where $\beta$ projects away from the coordinate corresponding to the exceptional divisor and $\gamma$ is such that $\gamma \circ \deg_{X_1} = \deg_{X_2} \circ \beta$.

We now restrict our attention to the class of \emph{smooth reflexive} lattice polytopes in $M_\RR$. A lattice polytope $P$ is \emph{reflexive} if its facet presentation is
\begin{center}
$P = \{ m \in M_\RR \ \vert \ \langle m , u_F \rangle \geq -1 \text{ for all facets } F \text{ in } P\}$.
\end{center}If $P$ is reflexive then the origin of $M_\RR$ is the only interior lattice point of $P$ and its dual polytope
is also a reflexive polytope. A polytope if \emph{smooth} if its dual polytope determines a smooth fan, and two reflexive polytopes $P_1,P_2 \subset M_\RR$  are \emph{lattice equivalent} if $P_1$ is the image of $P_2$ under an invertible linear map of $M_\RR$ induced by an isomorphism of $M$. Batyrev \cite{Bat1} uses smooth reflexive polytopes to classify smooth toric Fano varieties:

\begin{theorem}{\cite[Theorem 2.2.4]{Bat1}}
If $P$ is an $n$-dimensional smooth reflexive polytope then $X_P$ is an $n$-dimensional smooth toric Fano variety. Conversely, If $X$ is an $n$-dimensional smooth toric Fano variety then there exists an $n$-dimensional smooth reflexive polytope $P$ such that $X_P \cong X$. Moreover, if $P_1$ and $P_2$ are two smooth reflexive polytopes then $X_{P_1} \cong X_{P_2}$ if and only if $P_1$ and $P_2$ are lattice equivalent.
\end{theorem}
\noindent We therefore refer to smooth reflexive lattice polytopes as \emph{Fano polytopes} and as Batyrev observed, there are finitely many Fano polytopes up to lattice equivalence in each dimension \cite{Bat3}. There are five corresponding smooth toric Fano varieties in dimension 2 that were known classically, whilst Watanabe-Watanabe \cite{WaWa} and Batyrev \cite{Bat2} classified the 18 smooth toric Fano varieties in dimension 3. In dimension 4, Batyrev \cite{Bat1} used primitive collections and relations to classify the Fano polytopes and Sato \cite{Sato} completed the classification using toric blowups, bringing the total number of 4-dimensional smooth toric Fano varieties to 124. Kreuzer and Nill \cite{KrNi} calculated that there are 866 5-dimensional Fano polytopes up to lattice equivalence, while {\O}bro \cite{Obro} presented an algorithm that has classified Fano polytopes in dimensions up to 9. 

Sato \cite{Sato} records the birational geometry between the smooth toric Fano fourfolds by computing toric divisorial contractions in terms of the primitive relations for each variety. Figure \ref{fig:poset} in Appendix A is a diagram of the divisorial contractions between the smooth toric Fano fourfolds. There are 29 maximal toric Fano fourfolds with regard to these divisorial contractions, and we call these varieties \emph{birationally maximal}. A diagram showing the divisorial contractions between the smooth toric Fano threefolds can be found in \cite[page 92]{Oda} \cite{WaWa}.

\begin{remark}
The contraction from the variety $K_2$ to the variety $H_{10}$ stated in \cite{Sato} should be a contraction from $K_3$ to $H_{10}$.
\end{remark}

\subsection{Full Strong Exceptional Collections and Tilting Objects}

For a set of objects $\Scal = \{ \Scal_i \}$ in a triangulated category $\Dcal$, define $\langle \Scal \rangle$ to be the smallest triangulated subcategory of $\Dcal$ containing $\Scal$, closed under isomorphisms, taking cones of morphisms and direct summands, and $\langle \Scal \rangle^\perp$ to be the full triangulated subcategory of $\Dcal$ containing objects $\Fcal$ such that $\Hom (S, \Fcal) = 0$ for all $S \in \Scal$.  

\begin{definition}
For a set of objects $\Scal = \{ \Scal_i \}$ in $\Dcal$,
\begin{itemize}
\item[(i)] $\Scal$ \emph{classically generates} $\Dcal$ if $\langle \Scal \rangle = \Dcal$,
\item[(ii)] $\Scal$ \emph{generates} $\Dcal$ if $\langle \Scal \rangle^\perp = 0$. 
\end{itemize}  
\end{definition}

\noindent Let $\Dx$ be the bounded derived category of coherent sheaves on a variety $X$.

\begin{definition}
\noindent \begin{itemize}
\item[(i)] An ordered set of objects $( \Ecal_0 , \ldots , \Ecal_r )$ in $\Dx$ is a \emph{strong exceptional collection} if $\Hom(\Ecal_k, \Ecal_k) = \CC$ for all $k \in \{ 0, \ldots , r \}$ and
\begin{equation*}
 \Hom^i (\Ecal_k , \Ecal_j ) = 0 \text{ when } \begin{cases}
k>j, & i = 0, \\
\forall \ k, j, & i \neq 0.
\end{cases}
\end{equation*}

\item[(ii)] A strong exceptional collection $(\Ecal_0, \ldots , \Ecal_r)$ in $\Dx$ is \emph{full} if $\langle \Ecal_0, \ldots , \Ecal_r \rangle = \Dx$.
\end{itemize} 
\end{definition}

\begin{remark}
The distinction between classical generation and generation becomes irrelevant when using strong exceptional collections. To show that a strong exceptional collection $(\Ecal_0, \ldots , \Ecal_r )$ is full, it is enough to show that $\langle \Ecal_0, \ldots, \Ecal_r \rangle^\perp = 0$ as observed by Bridgeland--Stern \cite[Lemma C.1]{BrSt}.
\end{remark}

\noindent Given a full strong exceptional collection $(\Ecal_1, \ldots , \Ecal_r)$ of non-isomorphic objects in $\Dx$, its sum $\bigoplus \Ecal_i$ is a \emph{tilting object}:
\begin{definition}
An object $\Tcal$ in $\Dx$ is a \emph{tilting object} if $\Hom^i (\Tcal, \Tcal) = 0$ for $i \neq 0$ and $\langle \Tcal \rangle = \Dx$. If additionally $\Tcal$ is a sheaf or vector bundle, then it is called a \emph{tilting sheaf} or \emph{tilting bundle} respectively.
\end{definition}
\noindent For a tilting object $\Tcal$, let $A = \End(\Tcal)$ and $\Dcal^b(A)$ be the bounded derived category of finitely generated right $A$-modules. It was shown by Baer \cite{Baer} and Bondal \cite{Bond} that in the case when $X$ is a smooth projective variety, if the tilting object $\Tcal$ exists then we obtain an equivalence of categories
\begin{equation}
 \mathbf{R}\Hom_X (\Tcal, - ) \colon \Dx \longrightarrow \Dcal^b(A).
\end{equation}   

\noindent Note that when $\Tcal = \bigoplus_0^r \Ecal_i$ is the direct sum of a full strong exceptional collection, the Grothendieck group of $X$ is isomorphic to a rank $r+1$ lattice.

For two smooth projective varieties $Y$ and $Z$, let $\Ecal \in \Dcal^b(Y)$ and $\Fcal \in \Dcal^b(Z)$. Define
\begin{equation*}
\Ecal \stackrel{\mathbf{L}}{\boxtimes} \Fcal := \mathbf{L} p_1^* (\Ecal) \stackrel{\mathbf{L}}{\otimes} \mathbf{L}p_2^* (\Fcal) \in \Dcal^b(Y \times Z)
\end{equation*}
where $p_1$ and $p_2$ are the natural projections of $Y \times Z$ onto its components.
\begin{lemma}\label{lem:FSECproducts}\cite[Lemma 5.2]{Ueha}
Let $Y$ and $Z$ be as above. If $\Ecal \in \Dcal^b (Y)$ and $\Fcal \in \Dcal^b (Z)$ are tilting objects, then $\Ecal \stackrel{\mathbf{L}}{\boxtimes} \Fcal$ is a tilting object for $\Dcal^b (Y \times Z)$.   
\end{lemma} 

\section{Strong Exceptional Collections on Smooth Toric Varieties}

\noindent The combinatorics of a toric variety $X$ allow us to computationally determine whether a collection of line bundles on $X$ is strong exceptional. We can then consider how these strong exceptional collections behave under torus-invariant divisorial contractions. 

\subsection{Non-vanishing Cohomology Cones}\label{subsect:EMS-StrExc}

To check that a collection of effective line bundles $\{ L_0 := \Ocal_X, L_1 , \ldots , L_r \}$ on a smooth toric variety $X$ is strong exceptional, one needs to check that $H^i(X, L^{-1}_s \otimes L_t) = \Hom^i (L_s, L_t) = 0$ for $i > 0,\  0 \leq s,t \leq r$. Eisenbud, Musta{\c{t}}{\u{a}} and Stillman \cite{EiMuSt} introduced a method to determine when the cohomology of a line bundle on $X$ vanishes by considering if the line bundle avoids certain affine cones constructed in $\Pic(X)_\RR$. We recall the construction of these non-vanishing cohomology cones below.

For a $n$-dimensional toric variety $X$ with fan $\Sigma$, fix an enumeration $I_\Sigma$ of $\Sigma(1)$ and for $I \subset I_\Sigma$, let $Y_I$ be the union of the cones in $\Sigma$ having all edges in the complement of $I$. Using reduced cohomology with coefficients in $\CC$ we have
\begin{equation}
H^i_{Y_I} (\vert \Sigma \vert ) := H^i ( \vert \Sigma \vert, \vert \Sigma \vert \backslash Y_I) = H^{i-1} (\vert \Sigma \vert \backslash Y_I), 
\end{equation}
where the last equality holds for $i > 0$ as $\vert \Sigma \vert$ is contractible.

An element of
\begin{equation}
H_\Sigma := \lbrace I \subseteq I_\Sigma \phantom{.}\vert\phantom{.} H^i_{Y_I} (\vert \Sigma \vert ) \neq \emptyset \text{ for some } i>0 \rbrace  
\end{equation}
is called a \emph{forbidden set}. Define
\begin{equation}
\mathbf{p}_I \in \ZZ^{\Sigma(1)}, \text{ where } (\mathbf{p}_I)_\rho =
\begin{cases}
-1 \text{ if } \rho \in I \\
\phantom{-}0 \text{ if } \rho \notin I
\end{cases}
\end{equation}
and
\begin{equation}
C_I = \left\lbrace \mathbf{x} = (x_\rho) \in \ZZ^{\Sigma(1)} \phantom{.}\vert\phantom{.} x_\rho \leq 0 \text{ if } \rho \in I, x_\rho \geq 0 \text{ if } \rho \notin I \right\rbrace.   
\end{equation}

Setting $L_I := C_I + \mathbf{p}_I \subseteq \ZZ^{\Sigma(1)}$ we see that $L_I \subset C_I$ and $L_I = \lbrace \mathbf{x} \in \ZZ^{\Sigma(1)} \phantom{.}\vert\phantom{.} neg(\mathbf{x}) = I \rbrace$, where 
\begin{equation}
neg \colon \ZZ^{\Sigma(1)} \rightarrow I_\Sigma, \phantom{---} \mathbf{x} = (x_\rho) \mapsto \begin{cases}
\rho \text{  if } x_\rho < 0, \\
\emptyset \text{ else.} 
\end{cases}   
\end{equation}  
For each $I \in H_\Sigma$, the image of $L_I$ under the map $deg$ is an affine cone $\beta \subset \Pic(X)_\RR$. If a rank one reflexive sheaf $\mathcal{O}_{X_\Sigma} (D)$ lies in $\beta$ then \cite[Theorem 2.7]{EiMuSt} shows that
\begin{equation}
H^i (X , \mathcal{O}_{X} (D) ) \neq \emptyset, \text{ for some } i > 0. 
\end{equation}

\subsection{Cones Affected by Blow Ups}\label{subsect:CohConesBlowUps}
Assume that we have a chain of torus-invariant divisorial contractions $ X := X_0 \rightarrow X_1 \rightarrow \cdots \rightarrow  X_t$ between smooth toric varieties and let $\Lcal = \{ L_0, L_1, \ldots,  L_r \}$ be a collection of non-isomorphic line bundles on $X$ with corresponding vectors $\{ v_0, \ldots, v_r \}$ in $\Pic(X)_\RR$. By \eqref{gamma} we have maps between the Picard lattices
\begin{equation}\label{eq:ChainsOfPicMaps}
\Pic(X_0) \stackrel{\gamma_1}{\longrightarrow} \Pic (X_1) \stackrel{\gamma_2}{\longrightarrow} \cdots \stackrel{\gamma_{t}}{\longrightarrow} \Pic(X_t).
\end{equation}

For ease of notation we set:
\begin{itemize}
\item $\gamma_{(i \rightarrow j)}$ to be the composition of maps $\gamma_{j}\circ \gamma_{j-1}\circ \cdots \circ \gamma_{i+1}$ for $0 \leq i < j \leq t$;
\item $\Lcal_{X_k}$ to be the set of non-isomorphic line bundles on $X_k$ in the image of $\gamma_{(0 \rightarrow k)}(\Lcal)$, for $1 \leq k \leq t$;
\item $\mathfrak{C}_k \subset \Pic(X_0)_\RR$ to be the preimage of all non-vanishing cohomology cones for $X_k$ under the map $\gamma_{(0 \rightarrow k)}$ for $1 \leq k \leq t$, and $\mathfrak{C}_0 \subset \Pic(X_0)_\RR$ to be the non-vanishing cohomology cones for $X_0$.     
\end{itemize} 
By the construction of the sets $\mathfrak{C}_k$, we have the following result:
\begin{lemma}
If \[v_i - v_j \notin \bigcup_{k = 0}^t \mathfrak{C}_k\] for all $0 \leq i,j \leq r$ then $\Lcal$ is strong exceptional on $X$ and $\Lcal_{X_k}$ is strong exceptional on $X_k$, for $1 \leq k \leq t$.
\end{lemma}  It will be shown in this section that the preimage of the non-vanishing cohomology cones for $X_k$ under $\gamma_{(0 \rightarrow k)}$ is closely related to the non-vanishing cohomology cones for $X$.

\begin{lemma}[Forbidden sets duality]\label{lem:ForbiddenSetsDuality}
Let $I \subsetneq I_\Sigma$ and set $I^\vee  = I_\Sigma \backslash I$. If $I \in H_\Sigma$, then $I^\vee \in H_\Sigma$.  
\end{lemma}

\begin{proof}
It is enough to show that the line bundle corresponding to $\mathbf{p}_{I^\vee}$ has non-vanishing higher cohomology. Let $D := \sum_{\rho \in \Sigma(1)} a_\rho D_\rho$ be the torus-invariant Weil divisor corresponding to $\mathbf{p}_I = (a_\rho)$. By assumption, $H^i ( X ,\Ocal ( D ) ) \neq \emptyset$ for some $0<i<n$. By Serre duality and the fact that on a toric variety the canonical divisor is $K_X = -\sum_{\rho \in \Sigma(1)} D_\rho$,
\begin{equation}
\emptyset \neq H^i (X, \mathcal{O} ( D ))^\vee \cong H^{n-i} \left( X ,\Ocal ( K_X - D ) \right)
= H^j (X ,\mathcal{O} (\sum_{\rho \in \Sigma(1)} b_\rho D_\rho ))
\end{equation}
where $b_\rho = -1-a_\rho$. But
\begin{equation}
a_\rho = \begin{cases}
-1 \text{ if } \rho \in I \\
\ \ 0 \text{ if } \rho \notin I
\end{cases} \Rightarrow \ b_\rho = \begin{cases}
\ \ 0 \text{ if } \rho \in I \\
-1 \text{ if } \rho \notin I.
\end{cases} 
\end{equation}
Therefore, $(b_\rho) = \mathbf{p}_{I^\vee}$ and as $H^j (X ,\mathcal{O}(\sum_{\rho \in \Sigma(1)} b_\rho D_\rho )) \neq \emptyset$ for some $0<j<n$, we have $I^\vee \in H_\Sigma$. 
\end{proof}

\begin{remarks}
\item[(i)] The cone in $\Pic(X)_\RR$ corresponding to $I = I_\Sigma \in H_\Sigma$ is the non-vanishing $n^{th}$-cohomology cone. Its dual $I^\vee = \emptyset$ gives the non-vanishing $0^{th}$-cohomology cone. The chosen line bundles for a strong exceptional collection need to be effective but must have no higher cohomology, so for our purposes $\lbrace \emptyset \rbrace \notin H_\Sigma$ and the duality statement above does not hold for $I_\Sigma$. 

\item[(ii)] By Lemma \ref{lem:ForbiddenSetsDuality}, $Y_I$ can be redefined as the union of the cones in $\Sigma$ having all edges in $I$. 
   
\end{remarks}

Continuing with the notation in (\ref{blowup}), extend the enumeration $I_\Sigma$ of $\Sigma(1)$ to $I_{\Sigma^*} := I \cup \{x \}$, an enumeration of the rays of $\Sigma^* := \Sigma^*_{\sigma,x}$. It is clear that for
\begin{equation}
C_\sigma := \bigcup_{\sigma \subseteq \tau \subseteq \Sigma} \tau
\end{equation}
we have 
\begin{equation}
\Sigma \backslash \vert C_\sigma \vert = \Sigma^* \backslash \vert \bigcup_{\sigma \subseteq \tau} \Sigma^*_\tau (\sigma) \vert
\end{equation}
and so we only need to consider $C_\sigma$ when determining how the cones of $\Sigma$ change after the blow up of $\sigma \subset \Sigma$.

\begin{lemma}\label{lem:BlowupForbiddenSet}
Let $\emptyset \neq I \subseteq I_\Sigma \backslash C_\sigma (1)$. Then $I \cup \lbrace x\rbrace \in H_{\Sigma^*}$.
\end{lemma}

\begin{proof}
Firstly, assume for some $I \subset I_\Sigma$ that there exists a ray $\tau \subsetneq Y_I$ such that there is no cone $\sigma \subset Y_I$ with $\tau \subset \sigma$, and $\tau \cap \varsigma = \{ 0 \}$ for all cones $\tau \neq \varsigma \subset Y_I$. By considering $Y_I \subset \vert \Sigma \vert \cong \RR^n$, we can construct a loop around $\tau$ that is not contractible in $\vert \Sigma \vert \backslash Y_I$. Thus $I \in H_\Sigma$.

Now assume $\emptyset \neq I \subseteq I_\Sigma \backslash C_\sigma (1)$. By the construction of $\Sigma^*$ we have $x \cap \varsigma = \{ 0 \}$ for any cone $x \neq \varsigma \subset Y^*_{I \cup \{x \}}$ and there is not a cone $x \neq \sigma \subset Y^*_{I \cup \{x \}}$ that contains $x$, so $I \cup \{x\} \in H_{\Sigma^*}$ by the observation above.   
\end{proof}

\noindent By Lemma \ref{lem:ForbiddenSetsDuality}, $(I \cup \lbrace x \rbrace )^\vee \in H_{\Sigma^*}$ for $\emptyset \neq I \subseteq I_\Sigma \backslash C_\sigma (1)$. But $(I \cup \lbrace x \rbrace )^\vee = J \cup C_\sigma (1)$ for some $J \subsetneq I_\Sigma \backslash C_\sigma(1)$, so we have the corollary: 

\begin{corollary}
If $I \subsetneq I_\Sigma \backslash C_\sigma (1)$, then $I \cup C_\sigma (1) \in H_{\Sigma^*}$.
\end{corollary}  
\begin{lemma}
If $I \in H_\Sigma$ and $I \cap C_\sigma (1) = \emptyset$ then $I, I\cup \lbrace x \rbrace \in H_{\Sigma^*}$.
\end{lemma}

\begin{proof}
Let $I \in H_\Sigma$ such that $I \cap C_\sigma (1) = \emptyset$. By Lemma \ref{lem:BlowupForbiddenSet}, $I \cup \lbrace x \rbrace \in H_{\Sigma^*}$. As $I \cap C_\sigma (1) = \emptyset$, then $Y^{\Sigma^*}_I = Y^{\Sigma}_I$ and $\vert \Sigma^* \vert = \vert \Sigma \vert$, so $H^i_{Y^{\Sigma^*}_I} ( \vert \Sigma^* \vert ) = H^i_{Y^{\Sigma}_I} ( \vert \Sigma \vert )$ for all $i$. Thus $I \in H_{\Sigma^*}$.  
\end{proof}    
\noindent Again by duality, we have the corollary:
\begin{corollary}
If $I \in H_\Sigma$ is such that $C_\sigma (1) \subseteq I$, then $I, I \cup \lbrace x \rbrace \in H_{\Sigma^*}$.
\end{corollary}

\begin{lemma}\label{lem:BlowupForbiddenSet2}
If $I \in H_\Sigma$ then either $I \in H_{\Sigma^*}$ or $I\cup \lbrace x \rbrace \in H_{\Sigma^*}$.   
\end{lemma}
\begin{proof}
We have shown that the statement holds if $I \cap C_\sigma (1) = \emptyset$ and dually if $C_\sigma (1) \subseteq I$. There are two other cases to consider: in both cases, let $I = I_1 \cup I_2$ for $I_1 \subseteq I_\Sigma \backslash C_\sigma (1)$ and $I_2 \subsetneq C_\sigma (1)$.
\begin{itemize}
\item [Case 1:] $(\sigma(1) \nsubseteq I)$. Any subset $S \subseteq \tau(1)$ of any cone $\tau \subset \Sigma$ forms a cone in $\Sigma$ as $\Sigma$ is a smooth fan. From this and the fact that $\sigma(1),\lbrace x \rbrace \nsubseteq I$ we see that $Y^\Sigma_I = Y^{\Sigma^*}_I$ by the construction of $\Sigma^*$. Therefore $I \in H_\Sigma \Rightarrow I \in H_{\Sigma^*}$.
\item [Case 2:] $(\sigma(1) \subseteq I)$. By duality $I^\vee \in H_\Sigma$ and $I^\vee \cap \sigma(1) = \emptyset$, so $I^\vee \in H_{\Sigma^*}$ by Case 1. Applying duality again we have $I \cup \lbrace x \rbrace = (I^\vee)^\vee \in H_{\Sigma^*}$. 
\end{itemize}
\end{proof}
\begin{remark}
It is not always the case that $I \in H_\Sigma \Rightarrow I,I\cup \lbrace x \rbrace \in H_{\Sigma^*}$ (see Example \ref{ex:E1Fourfold}). 
\end{remark}

\noindent Recalling the chain of linear maps \eqref{eq:ChainsOfPicMaps}, we have a simple description of the preimage in $\Pic(X)_\RR$ of non-vanishing cohomology cones for the variety $X_t$ using non-vanishing cohomology cones for $X$. Let $\Lambda \subset \Pic(X)_\RR$ be a non-vanishing cohomology cone for $X$ and $\{x_1 , \ldots , x_t \}$ be the rays in the fan $\Sigma_X$ for $X$ that label the exceptional divisors $\{E_1, \ldots , E_t \}$ from the blow ups in \eqref{eq:ChainsOfPicMaps}. The list of exceptional divisors can be extended to give a basis $\{E_1, \ldots , E_t, y_{1}, \ldots , y_{s} \}$ of $\Pic (X)_\RR$. Let $I_E \subset I_{ \Sigma_{X} }$ label the rays $\{ x_1, \ldots , x_t \}$. Using the non-vanishing cohomology cone construction from \cite{EiMuSt}, $\Lambda$ is the image under $deg$ of the cone $L_I$ for some $I \subset I_{ \Sigma_{X} }$, which we can decompose as $I = I^\Lambda_E \sqcup I^\Lambda$ for some $I^\Lambda_E \subset I_E, I^\Lambda \subset I_{ \Sigma_{X} } \backslash I_E$. We write $(I^\Lambda_E)^* := I_E \backslash I^\Lambda_E$.     
\begin{proposition}\label{prop:PreimageCohomologyCones}
Let $\Lambda_t$ be a non-vanishing cohomology cone for the variety $X_t$ in \eqref{eq:ChainsOfPicMaps} and $\tilde{\Lambda_t}$ be its preimage under $\gamma_{(0 \rightarrow t)}$. Then there exists a non-vanishing cohomology cone $\Lambda$ for $X$ with forbidden set $I = I^\Lambda_E \sqcup I^\Lambda$ such that \[\tilde{\Lambda_t} \cap \Pic(X) = \left(\deg_X(L_{I^\Lambda_E \sqcup I^\Lambda}) \cup \deg_X(L_{(I^\Lambda_E)^* \sqcup I^\Lambda})\right) \cap \Pic(X)\]. 

\noindent Let $\Lambda$ be defined by the intersection of closed half-spaces $H_i$ in $\Pic (X)_\RR$ given by equations $a^i_1 E_1 +  \ldots + a^i_t E_t + a^i_{t+1} y_{1} + \ldots + a^i_{t+s} y_s \leq a^i$ where $a^i_1, \ldots , a^i_{t+s}, a^i \in \RR$ are fixed and $i$ is in an indexing set $S$. Then $\tilde{\Lambda_t}$ is the intersection of the closed half spaces $a^i_{t+1} y_{1} + \ldots + a^i_{t+s} y_s \leq a^i$, $i \in S$. 
\end{proposition}
\begin{proof}
We first show the statement for the blowup $\varphi \colon X_{ \Sigma^{*}_{\sigma,x} } \longrightarrow X_\Sigma$ from \eqref{blowup}. Let $\Lambda_0$ be a non-vanishing cone for $X_\Sigma$ determined by the forbidden set $I \subset I_\Sigma$ and the cone $L_I$. By Lemma \ref{lem:BlowupForbiddenSet2} there exists a non-vanishing cohomology cone $\Lambda$ for $X_{\Sigma^*_{\sigma,x}}$ such that its defining forbidden set is either $I \cup {x} \subset I_{ \Sigma^{*}_{\sigma,x} }$ or $I \subset I_{ \Sigma^{*}_{\sigma,x} }$. The map $\beta$ in \eqref{gamma} is just projection away from the coordinate corresponding to the exceptional divisor $x$ and so the preimage of $(L_I \cap \NN^{\Sigma(1)}) \subset \RR^{\Sigma(1)}$ under $\beta$ is $(L_I \cup L_{I \cup {x}}) \cap \NN^{\Sigma^*_{\sigma,x}(1)} \subset \RR^{\Sigma^*_{\sigma,x}(1)}$. As \eqref{gamma} is a commutative digram, the preimage of $\Lambda_0 \cap \Pic(X)$ under $\gamma$ is $(\deg (L_I \cup L_{I \cup {x}})) \cap \Pic(\Sigma^*_{\sigma,x})$. By repeated application of Lemma \ref{lem:BlowupForbiddenSet2}, we obtain the required result for a chain of blowups \eqref{eq:ChainsOfPicMaps}.
As $\{ E_1 , \ldots , E_t \}$ can be extended to give a basis of $\Pic (X)_\RR$ in \eqref{eq:ChainsOfPicMaps}, the second statement of the proposition holds.
\end{proof} 

The simplicity of the preimage of non-vanishing cohomology cones under blowups can help explain why the following proposition holds. Recall that a smooth toric Fano fourfold $X$ is called \emph{birationally maximal} if there does not exist a smooth toric Fano fourfold $X'$ with blowup $X' \rightarrow X$.  
\begin{proposition}\label{prop:SECflowDownFourfolds}
Let $X$ be a birationally maximal smooth toric Fano fourfold and $r+1 =$ rank$(K_0(X))$. There exists a strong exceptional collection of line bundles $\Lcal = \{ L_0, \ldots , L_{r} \}$ on $X$ such that for every chain of torus-invariant divisorial contractions $X \rightarrow X_1 \rightarrow \cdots \rightarrow X_t$ from Figure \ref{fig:poset}, the set of line bundles $\Lcal_{X_i}$ on $X_i$ is strong exceptional, for $1 \leq i \leq t$. A database of these collections can be found in \cite{Prna2}.   
\end{proposition}      

\begin{proof}
Given a maximal smooth toric Fano fourfold $X$ and a chain of divisorial contractions between $\{ X_0 := X, X_1, \ldots, X_t \}$, we construct the preimage $\mathfrak{C}_i$ in $\Pic(X)_\RR$ of the non-vanishing cohomology cones for each contraction $X_i$ using the \emph{QuiversToricVarieties} package \cite{Prna1}. A computer search then finds line bundles $\{ L_0, L_1, \ldots,  L_{r} \}$ on $X$ with corresponding vectors $\{ v_0, \ldots, v_{r} \}$ in $\Pic(X)_\RR$ such that $v_j - v_k$ avoids $\mathfrak{C}_i$ for all $0 \leq j,k \leq r$ and $0 \leq i \leq t$. \end{proof}

\begin{remarks}
\item[(i)] The collections given in Proposition \ref{prop:SECflowDownFourfolds} are not necessarily the same collections given by Theorem \ref{thm:FSECfourfolds}. In particular, not all of them have been show to be full. 
\item[(ii)] If two toric varieties $X_1$ and $X_2$ have the same primitive collections, then they have the same forbidden sets up to a suitable ordering of the rays of $\Sigma_{X_1}$ and $\Sigma_{X_2}$. It is therefore often the case that given a suitable basis of $\Pic(X_1)_\RR$ and $\Pic(X_2)_\RR$, if the line bundles corresponding to a list of integral points $\{v_i\}_{i \in I} \subset \ZZ^d \cong \Pic(X_1)_\RR$ is strong exceptional on $X_1$, then the collection of line bundles corresponding to the same list   $\{v_i\}_{i \in I} \subset \ZZ^d \cong \Pic(X_2)_\RR$ is strong exceptional on $X_2$.
\end{remarks}

\begin{example}\label{ex:E1Fourfold}
The $10^{th}$ smooth toric Fano variety $X := E_1$ has ray generators \[u_0 = \begin{tiny}
\begin{bmatrix}
1 \\ 
0 \\ 
0 \\ 
0
\end{bmatrix} 
\end{tiny}
, 
u_1 = \begin{tiny} \begin{bmatrix}
0 \\ 1 \\ 0 \\ 0
\end{bmatrix} \end{tiny},
u_2 = \begin{tiny} \begin{bmatrix}
0 \\ 0 \\ 1 \\ 0
\end{bmatrix} \end{tiny},
u_3 = \begin{tiny} \begin{bmatrix}
0 \\ 0 \\ 0 \\ 1
\end{bmatrix} \end{tiny},
u_4 = \begin{tiny} \begin{bmatrix}
-1 \\ 0 \\ 0 \\ 0
\end{bmatrix} \end{tiny},
u_5 = \begin{tiny} \begin{bmatrix}
3 \\ -1 \\ -1 \\ -1
\end{bmatrix} \end{tiny},
u_6 = \begin{tiny} \begin{bmatrix}
2 \\ -1 \\ -1 \\ -1
\end{bmatrix} \end{tiny}\] for its fan $\Sigma_X$. The blowup $\phi \colon E_1 \rightarrow B_1$ between the $10^{th}$ and the $1^{st}$ smooth toric Fano fourfolds (see Figure \ref{fig:poset}) has the exceptional divisor $E = D_6$ labelled by the ray generator $u_6$. Note that $Y := B_1$ has the fan $\Sigma_Y$ with ray generators $\{ u_0, \ldots , u_5 \}$. We take the corresponding divisors $D_0, D_1, E$ to be a basis for $\Pic(X)$, and the linear equivalences between the divisors for $X$ are $D_1 \sim D_2 \sim D_3, \ D_4 \sim D_0 + 3D_1 -E, \ D_5 \sim D_1 - E$. The linear equivalences between the divisors for $Y$ are $D'_1 \sim D'_2 \sim D'_3, \ D'_4 \sim D'_0 + 3D'_1, \ D'_5 \sim D'_1$. The forbidden sets for $X$ are
\begin{center}
\begin{tabular}{|c|c|}
\hline 
Non-Vanishing $i$-th & Forbidden Sets \\ Cohomology Cones &  \\ 
\hline 
1 & $\{0,4\},\{4,5\},\{0,4,5\},\{0,6\},\{0,4,6\}$ \\ 
\hline 
2 &  \\ 
\hline 
3 & $\{1,2,3,5\},\{1,2,3,4,5\},\{1,2,3,6\},\{0,1,2,3,6\},\{1,2,3,5,6\}$ \\ 
\hline 
4 & $\{0,1,2,3,4,5,6\}$ \\ 
\hline 
\end{tabular}
\end{center}
\phantom{.}
   
\noindent and the forbidden sets for $Y$ are

\phantom{.} 

\begin{center}
\begin{tabular}{|c|c|}
\hline 
Non-Vanishing $i$-th Cohomology Cones & Forbidden Sets \\ 
\hline 
1 & $\{0,4\}$ \\ 
\hline 
2 &  \\ 
\hline 
3 & $\{1,2,3,5\}$ \\ 
\hline 
4 & $\{0,1,2,3,4,5\}$ \\ 
\hline 
\end{tabular}
\end{center}
 
\noindent In this example we see that for the forbidden set $I \in \{ \{0,4\}, \{1,2,3,5\} \}$ for $Y$, both $I$ and $I \cup \{u_6 \}$ are forbidden sets for $X$, whilst for the forbidden set $I = \{0,1,2,3,4,5\}$ for $Y$, only $I \cup \{u_6 \}$ is a forbidden set for $X$. Let $\gamma \colon \Pic(X)_\RR \rightarrow \Pic(Y)_\RR$ be the linear map defined in \eqref{gamma} and denote $\gamma^{-1} (C)$ for the preimage of a cone $C \subset \Pic(Y)_\RR$ under $\gamma$. Then
\[ \gamma^{-1} (\deg_Y(L_{\{0,4\}} )) \cap \Pic(X) = (\deg_X(L_{\{0,4\}}) \cup \deg_X(L_{\{0,4,6\}})) \cap \Pic(X) 
\] and
\[ \gamma^{-1} (\deg_Y(L_{\{1,2,3,5\}} )) \cap \Pic(X) = (\deg_X(L_{\{1,2,3,5\}}) \cup \deg_X(L_{\{1,2,3,5,6\}})) \cap \Pic(X) 
\]
Thus for a strong exceptional collection of line bundles $\Lcal$ on $X$, only $\gamma^{-1} (\deg_Y (L_{\{0,1,2,3,4,5\}}))$ provides a restriction for the distinct line bundles in the image of $\gamma(\Lcal)$ to be strong exceptional on $Y$. The cone $\deg_X (L_{\{0,1,2,3,4,5,6\}})$ is given by the system of equations \[
\begin{cases} 
a_1 + 6a_2 + a_3 \leq -2\\
a_2 + a_3 \leq -7\\
a_3 \leq 1
\end{cases}, \begin{tiny} \begin{bmatrix}
a_1 \\ a_2 \\ a_3  
\end{bmatrix} \end{tiny} \in \Pic(X)_\RR \] in $\Pic(X)$, whilst $\gamma^{-1} (\deg_Y (L_{\{0,1,2,3,4,5\}}))$ is given by the system of equations
\[
\begin{cases} 
a_1 + 6a_2  \leq -2\\
a_2  \leq -7
\end{cases}, \begin{tiny} \begin{bmatrix}
a_1 \\ a_2 \\ a_3  
\end{bmatrix} \end{tiny} \in \Pic(X)_\RR \] as expected by Proposition \ref{prop:PreimageCohomologyCones}.   
\end{example}

\section{Generation of $\Dx$: The Frobenius Morphism (Method 1)}

\noindent Let $X$ be an $n$-dimensional smooth toric variety and $\Lcal$ a strong exceptional collection on $X$. We present two different methods to show that $\Lcal$ is full in this paper. The first method depends on the Frobenius morphism and follows Uehara's approach \cite{Ueha} to generation of the derived category by line bundles on the smooth toric Fano threefolds.

\subsection{The Frobenius Morphism}\label{frob}

Fix a positive integer $m$ and define the $m$-th Frobenius map $F_m \colon X \rightarrow X$ to be the morphism that is the identity on the underlying topological space but which takes a section $s \in \Ocal_X$ to $s^{\otimes m}$. Thomsen \cite{Thom} shows that for a line bundle $L$ on $X$, $(F_m)_*(L)$ splits into a finite direct sum of line bundles. He provides an algorithm to compute these line bundles, which is detailed below.

Let $\Sigma$ be the fan for $X$ and set $d := \vert \Sigma (1) \vert$. From (\ref{exseq}), a vector $\textbf{w} \in \ZZ^d$ determines the line bundle $L = \Ocal_X (\sum w_i D_i)$. To compute $(F_m)_* (L)$, fix a maximal cone $\sigma = \text{cone}(\rho_{i_1}, \ldots, \rho_{i_n}) \subset \Sigma$ and set
\begin{equation}
P^p_m := \{ \textbf{v} \in \ZZ^p \ \vert \ 0 \leq v_i < m \}.
\end{equation} 
Define $A = (\textbf{a}_1, \ldots, \textbf{a}_d)^t \in M(d,n)$ to be the matrix whose rows are the ray generators $\textbf{a}_i := u_{\rho_i}$ in $\Sigma$. As $\sigma$ is maximal and $\Sigma$ is smooth, we have $\sigma(1) = \{ \rho_{i_1}, \ldots , \rho_{i_n}\}$ and the corresponding matrix $A_\sigma := (\textbf{a}_{i_1}, \ldots , \textbf{a}_{i_n})^t  \in M(n,n)$ is invertible. Define the restriction $\textbf{w}$ to $\sigma$ as $\textbf {w}_\sigma := (w_{i_1}, \ldots , w_{i_n} )^t \in \ZZ^n$. For $\textbf{v} \in P^n_m$, the vectors $\textbf{q}^m (\textbf{v},\textbf{w},\sigma) \in \ZZ^d$ and $\textbf{r}^m (\textbf{v},\textbf{w},\sigma) \in P^d_m$ are uniquely determined by the equation
\begin{equation}
AA^{-1}_\sigma (\textbf{v} - \textbf{w}_\sigma) + \textbf{w} = m \textbf{q}^m(\textbf{v},\textbf{w},\sigma) + \textbf{r}^m(\textbf{v},\textbf{w},\sigma).
\end{equation}   
Finally, define the Weil divisor $D^m_{\textbf{v},\textbf{w},\sigma} := \sum q_i^m (\textbf{v},\textbf{w},\sigma) D_i$.

\begin{lemma}\cite[Theorem 1]{Thom}
The Frobenius push-forward of $L = \Ocal_X (\sum w_i D_i )$ is 
\begin{equation}
(F_m)_* (L) = \bigoplus_{\textbf{v} \in P^n_m} \Ocal_X (D^m_{\textbf{v},\textbf{w},\sigma}).
\end{equation}
\end{lemma}

\noindent Thomsen \cite{Thom} also observes that $(F_m)_* (L)$ does not depend on the choice of the maximal cone $\sigma$. We can assume the primitive ray generators of $\sigma$ forms the standard basis of $\ZZ^n$, in which case

\begin{equation}\label{frobeq}
q_i^m (\textbf{v}, \textbf{w}) = \lfloor \frac{u^t_{\rho_i} (\textbf{v}-\textbf{w}_{\sigma})- w_i}{m} \rfloor
\end{equation} 
where $\lfloor x \rfloor \in \ZZ$ is the round-down $x-1 < \lfloor x \rfloor \leq x$.
Set
\begin{equation}
\mkD (D)_m := \{ L \in \Pic (X) \ \vert \ L \text{ is a direct summand of } (F_m)_* (\Ocal_X (D)) \}.
\end{equation}
Thomsen \cite[Proposition 1]{Thom} shows that the set
\begin{equation}
\mkD(D) := \bigcup_{m > 0} \mkD (D)_m, 
\end{equation} 
is finite. For brevity, we denote $\mkD_m := \mkD(0)_m$, the set of line bundles in $(F_m)_* (\Ocal_X)$. Note that we can use $\mkD_m$ to find strong exceptional collections of line bundles on $X$:

\begin{lemma}\cite[Lemma 3.8(i)]{Ueha}\label{nefsec}
For any fixed positive integer $m$, the set of line bundles $\{ L \in \mkD_m \ \vert \ L^{-1} \text{ is nef}\} \subseteq \mkD_m$ is a strong exceptional collection on $X$.
\end{lemma}

\subsection{Method 1}\label{subsect:Method 1} 
We can use the Frobenius morphism to find sets of line bundles that generate $\Dx$. 

\begin{lemma}\label{pushgen}\cite[Lemma 5.1]{Ueha}
Let $f \colon X \rightarrow Y$ be a proper morphism between smooth varieties. Assume that $\Ecal$ generates $\Dx$ and $\Ocal_Y$ is a direct summand of $\RR f_* \Ocal_X$ Then $\RR f_* \Ecal$ generates $\Dcal^b (Y)$. 
\end{lemma}

\begin{proposition}\label{prop:Method1}
Let $X$ be a smooth toric Fano variety of dimension $n$ and $\Lcal$ be a strong exceptional collection of line bundles on $X$. If the set of line bundles
\begin{equation}
\mkD_m^{gen} := \bigcup_{0 \leq i \leq n} \mkD_m (\omega_X^{-i})   
\end{equation}
is contained in $\langle \Lcal \rangle$ for some positive integer $m$, then $\Lcal$ generates $\Dx$.
\end{proposition}
\begin{proof}
As $X$ is Fano, the anticanonical bundle $\omega_X^{-1}$ is ample and so a result by Van den Bergh \cite[Lemma 3.2.2]{VdBe} implies that $\bigoplus_{i = 0}^{n} \omega_X^{-i}$ is a generator for $\Dx$. The Frobenius morphism $F_m$ is proper so $\bigcup_{0 \leq i \leq n} \mkD_m (\omega^{-i})$ generates $\Dx$ by Lemma \ref{pushgen}.  
\end{proof}

To show that $\mkD_m^{gen} \subset \langle \Lcal \rangle$ for some $m > 0$, we use exact sequences of line bundles to generate objects in $\langle \Lcal \rangle$; for examples of these calculations on the toric Fano threefolds see \cite{Ueha} or \cite{BeTi}. This process is easier when the line bundles in $\mkD_m^{gen}$ are close together in $\Pic(X)$, which occurs when the value of $m$ is large. However, the larger the value of $m$, the longer it takes to compute $\mkD_m^{gen}$, so in practice $m$ is often chosen by trial and error.  

\begin{example} 
We use \emph{Method 1} to show that a given collection of line bundles generates $\Dx$ when $X$ is $I_1$, the $61^{st}$ smooth toric Fano fourfold. The variety $X$ has ray generators \[u_0 = \begin{tiny}
\begin{bmatrix}
1 \\ 
0 \\ 
0 \\ 
0
\end{bmatrix} \end{tiny} ,
u_1 = \begin{tiny}\begin{bmatrix}
0 \\ 1 \\ 0 \\ 0
\end{bmatrix} \end{tiny},
u_2 = \begin{tiny}\begin{bmatrix}
0 \\ 0 \\ 1 \\ 0
\end{bmatrix} \end{tiny},
u_3 = \begin{tiny}\begin{bmatrix}
0 \\ 0 \\ 0 \\ 1
\end{bmatrix} \end{tiny},
u_4 = \begin{tiny}\begin{bmatrix}
0 \\ -1 \\ 1 \\ 0
\end{bmatrix} \end{tiny},
u_5 = \begin{tiny}\begin{bmatrix}
2 \\ 0 \\ -1 \\ -1
\end{bmatrix} \end{tiny},
u_6 = \begin{tiny}\begin{bmatrix}
-1 \\ 0 \\ 0 \\ 0
\end{bmatrix} \end{tiny}
u_7 = \begin{tiny}\begin{bmatrix}
-1 \\ 0 \\ 1 \\ 0
\end{bmatrix} \end{tiny} \]
and the linear equivalences between the toric divisors are $D_0 \sim -2D_5 + D_6 +D_7, \ D_1 \sim D_4, \ D_2 \sim -D_4 +D_5 -D_7 , \ D_3 \sim D_5$. Set $m = 10$ and let $\mathbf{v} = \begin{tiny} \begin{bmatrix}
x \\ y \\ z \\ w
\end{bmatrix} \end{tiny} \in P^4_m$. The anticanonical divisor $-K_X$ is equal to $\sum_{i \in \Sigma(1)} D_i$, so in order to calculate $\mkD_m (\omega^{-1})$ we set $\mathbf{w} = \begin{tiny} \begin{bmatrix}
1 \\ \vdots \\ 1 
\end{bmatrix} \end{tiny} \in \ZZ^8$. By \eqref{frobeq}, the solution to \[ \mathbf{q}^m(\mathbf{v}, \mathbf{w}) =  \begin{bmatrix}
\lfloor \frac{(x-1)+1}{m} \rfloor \\ 
\lfloor \frac{(y-1)+1}{m} \rfloor \\
\lfloor \frac{(z-1)+1}{m} \rfloor \\
\lfloor \frac{(w-1)+1}{m} \rfloor \\
\lfloor \frac{(-y+z)+1}{m} \rfloor \\
\lfloor \frac{(2x-z-w)+1}{m} \rfloor \\
\lfloor \frac{(-x+1)+1}{m} \rfloor \\
\lfloor \frac{(-x+z)+1}{m} \rfloor \\
\end{bmatrix}  = 
\begin{bmatrix}
0 \\ 
0 \\
0 \\
0 \\
\lfloor \frac{-y+z+1}{m} \rfloor \\
\lfloor \frac{2x-z-w+1}{m} \rfloor \\
\lfloor \frac{-x+2}{m} \rfloor \\
\lfloor \frac{-x+z+1}{m} \rfloor \\
\end{bmatrix} \]   
for $\mathbf{v} \in P^4_m$ is an element of $\mkD(\omega^{-1})_m$ and similarly we can calculate $\mkD(\omega^{-i})_m$ by determining $\mathbf{q}^m(\mathbf{v}, i\mathbf{w})$, for $0 \leq i \leq 4$. It follows that $\vert \mkD_m \vert = 18, \ \vert \mkD (\omega^{-1})_m \vert = 18$ and $\vert \mkD^{gen}_m \vert = 46$. For each line bundle $L$ in the collection
\[ \Lcal = \left\lbrace
\begin{array}{c|}
\Ocal_X (-iD_4 - jD_5 - kD_6), \Ocal_X ( - D_6 - D_7),\ \Ocal_X (-D_4 - D_6 - D_7),  \\
\Ocal_X (-D_5 - D_6 - D_7), \ \Ocal_X (-D_4 - D_5 - D_6 - D_7) \\
\end{array} \ i,j,k = 0,1 \right\rbrace \subset \mkD_m, \] 
$L^{-1}$ is nef, so $\Lcal$ is a strong exceptional collection by Lemma \ref{nefsec}. A list of rays $\{ \rho_{i_1} , \ldots , \rho_{i_j} \}$ forms a cone in $\Sigma$ if and only if $D_{i_1} \cap \ldots \cap D_{i_j} \neq \emptyset$, so we can use the primitive collections of $X$ to determine which divisors do not intersect. For example, the primitive collection $\{ u_0,u_7 \}$ for $X$ implies that $D_0 \cap D_7 = \emptyset$ and so we obtain the exact sequence
\[ 0 \rightarrow \Ocal_X(-D_0 - D_7) \rightarrow \Ocal_X(-D_0) \oplus \Ocal_X(- D_7) \rightarrow \Ocal_X \rightarrow 0. \] Using the basis $\{D_4, D_5, D_6, D_7 \}$ for $\Pic(X)$, rewrite the exact sequence as \begin{equation}\label{lbexseq} 0 \rightarrow \Ocal_X(2D_5 - D_6 - 2D_7) \rightarrow \Ocal_X(2D_5 -D_6 - D_7) \oplus \Ocal_X(- D_7) \rightarrow \Ocal_X \rightarrow 0. \end{equation}
We can use the exact sequences determined by the primitive collections to show that $\mkD^{gen}_m \subset \langle \Lcal \rangle$. For example, the tensor of $\Ocal_X(-2D_5 + D_7) \in \mkD_m^{gen} \backslash \Lcal$ with \eqref{lbexseq} gives the exact sequence 
\begin{equation}\label{tlbexseq} 0 \rightarrow \Ocal_X( - D_6 - D_7) \rightarrow \Ocal_X( -D_6 ) \oplus \Ocal_X(- 2D_5) \rightarrow \Ocal_X(-2D_5 + D_7) \rightarrow 0. \end{equation}
All of the line bundles in \eqref{tlbexseq} except $\Ocal_X(-2D_5 + D_7)$ are in $\langle \Lcal \rangle$, hence so is $\Ocal_X(-2D_5 + D_7)$. By the same method and using the exact sequences of line bundles determined by the primitive collections for $X$, every line bundle in $\mkD_{gen}$ is contained in $\langle \Lcal \rangle$ and so $\Lcal$ is full by Proposition \ref{prop:Method1}. 
\end{example}

\section{Quiver Moduli and the Structure Sheaf of the Diagonal}\label{sect:d1Map}

\noindent Let $X$ be a smooth toric variety and $\Lcal = \{ L_0, \ldots , L_r \}$ be a collection of line bundles on $X$. In this section we use the \emph{quiver of sections} to encode the endomorphism algebra $A = \End\left(\bigoplus_{i = 0}^r L_i \right)$. We also introduce the map $d_1$ between certain vector bundles on $X \times X$ and give conditions as to when the cokernel of $d_1$ is the structure sheaf $\Ocal_\Delta$ of the diagonal embedding into $X \times X$.  

\subsection{Quivers of Sections and Moduli Spaces of Quiver Representations}

A quiver $Q$ consists of a vertex set $Q_0$, an arrow set $Q_1$ and maps $\textbf{h,t} \colon Q_1 \rightarrow Q_0$ giving the vertices at the head and tail of each arrow. We assume that $Q$ is connected, acyclic and rooted at a unique source. A non-trivial path in $Q$ is a sequence of arrows $p = a_1 \ldots a_k$ such that $\textbf{h} (a_i) = \textbf{t} (a_{i+1})$ for $1 \leq i \leq k-1$, in which case $\textbf{t}(p) := \textbf{t}(a_1)$, $\textbf{h}(p) := \textbf{h}(a_k)$ and $\supp (p) = \{a_1, \ldots, a_k \}$. Each vertex $i \in Q_0$ gives a trivial path $e_i$ with $\textbf{h}(e_i) = \textbf{t}(e_i) = i$. By taking the paths as a generating set and defining multiplication to be concatenation of paths when possible and zero otherwise, we obtain the path algebra $\CC Q$. If $J$ is a two-sided ideal of relations in $\CC Q$ then we obtain the quotient algebra $\CC Q/J$, and we will use the notation $(Q,J)$ when we are considering the quiver with relations.   

A \emph{representation} $W$ of a quiver $Q$ assigns a $\CC$-vector space $W_i$ to each vertex $i \in Q_0$ and a $\CC$-linear map $w_a \colon W_{\textbf{t}(a)} \rightarrow W_{\textbf{h}(a)}$ to each arrow $a \in Q_1$. We will assume that the dimension of every vector space $W_i$ is equal to 1 and so the \emph{dimension vector} is $\textbf{v} := (1, \ldots , 1)^t \in \ZZ^{Q_0}$. A morphism $\phi$ between two representations $W$ and $V$ is a collection of $\CC$-linear maps $\phi_i \colon W_i \rightarrow V_i$ such that for any arrow $a \in Q_1$, the following square commutes:
\begin{equation}
 \begin{CD}   
     W_{\textbf{t}(a)} @>{w_a}>> W_{\textbf{h}(a)} \\
     @V{\phi_{\textbf{t}(a)}}VV   @VV{\phi_{\textbf{h}(a)}}V  \\
     V_{\textbf{t}(a)} @>>{v_a}> V_{\textbf{h}(a)} \\
  \end{CD}
\end{equation}
For the quiver with relations $(Q,J)$, we can consider representations of $Q$ that respect the relations in $J$. More precisely, a representation of $(Q,J)$ is a representation $W$ of $Q$ such that for any two paths $p_1 = a_1 a_2 \ldots a_t, \  p_2 = b_1 b_2 \ldots b_s$ in $Q$ with $p_1 - p_2 \in J$, we have \[ w_{a_t} \circ \cdots \circ w_{a_2} \circ w_{a_1} = w_{b_s} \circ \cdots \circ w_{b_2} \circ w_{b_1}.\]  

Define $\Wt(Q) := \{ \theta \in (\ZZ^{Q_0})^\vee \mid \theta(\textbf{v}) = 0 \}$ to be the \emph{weight space} for $Q$. Each $\theta \in \Wt(Q)$ determines a stability parameter, where a representation $W$ is $\theta$-(semi)stable if for every non-zero proper subrepresentation $W' \subset W$ we have $\theta (W') := \sum_{\{i \mid W'_i \neq 0\}} \theta_i > (\geq) \ 0$. A parameter $\theta$ is \emph{generic} if every $\theta$-semistable representation is $\theta$-stable. In particular, the \emph{special parameter} $\vartheta := (-r,1,1,\ldots,1)$ is generic, where $r = \vert Q_0 \vert-1$. A generic stability parameter $\theta$ can then be used to construct the \emph{fine moduli space of $\theta$-stable representations} $\Mcal_\theta (Q)$, as introduced by King \cite{King1}. The space $\Mcal_\theta (Q)$ is a projective variety as $Q$ is acyclic \cite[Proposition 4.3]{King1} and Hille \cite[Section 1.3]{Hill} has shown that $\Mcal_\theta (Q)$ is a smooth toric variety due to our choice of dimension vector \textbf{v}. To construct $\Mcal_\theta (Q)$ explicitly, let the characteristic functions $\chi_i \colon Q_0 \rightarrow \ZZ$ for $i \in Q_0$ and $\chi_a \colon Q_1 \rightarrow \ZZ$ for $a \in Q_1$ form bases for $\ZZ^{Q_0}$ and $\ZZ^{Q_1}$ respectively. The incidence map $\inc \colon \ZZ^{Q_1} \rightarrow \ZZ^{Q_0}$ defined by $\inc(\chi_a) = \chi_{\textbf{h}(a)} - \chi_{\textbf{t}(a)}$ determines the exact sequence
\begin{equation}\label{eq:IncidenceExactSequence}
 \begin{CD}   
    0@>>> \tilde{M} @>>> \ZZ^{Q_1} @>{\text{inc}}>> \text{Wt}(Q) @>>> 0.
\end{CD}
\end{equation}
For a fixed generic stability parameter $\theta \in \Wt(Q)$, let $\CC[y_a \mid a \in Q_1]_\theta = \CC[\NN^{Q_1} \cap \text{inc}^{-1}(\theta)]$ denote the $\theta$-graded piece. Then $\Mcal_\theta (Q)$ is the GIT quotient
\begin{equation}
\Mcal_\theta(Q) = \CC^{Q_1} \git_\theta T = \Proj \left( \bigoplus_{j \geq 0} \CC \left[  y_a \mid a \in Q_1 \right]_{j\theta} \right) 
\end{equation} 
where the action of $T := \Hom_\ZZ (\text{Wt}(Q), \CC^*)$ is induced from the action of $(\CC^*)^{Q_0} \cong \prod_{i \in Q_0} \GL(W_i)$ on $\CC^{Q_1}$ determined by the incidence map. By choosing a group isomorphism between $T$ and $\{ (g_1, \ldots , g_k) \in (\CC^*)^{Q_0} \mid g_0 = 1 \}$ we obtain a $T$-equivariant vector bundle $\bigoplus_{i \in Q_0} \Ocal_{\CC^{Q_1}}$ on $\CC^{Q_1}$ which descends to the universal family $\bigoplus_{i \in Q_0} F_i$ on $\Mcal_\theta (Q)$ \cite[Proposition 5.3]{King1}. The summands $F_i$ are called the \emph{tautological line bundles} on $\Mcal_\theta (Q)$ and $F_0$ is the trivial line bundle, where $0 \in Q_0$ labels the source of $Q$, as $T$ acts trivially on the summand given by $i=0$ in $\bigoplus_{i \in Q_0} \Ocal_{\CC^{Q_1}}$. The dimension of $\Mcal_\theta (Q)$ is $\vert Q_1 \vert - \vert Q_0 \vert + 1$ and $\Pic(\Mcal_\theta (Q)) \cong \Wt(Q)$ \cite[Theorem 2.3]{Hill}. 

If we are considering a quiver with ideal of relations $J$, we denote the fine moduli space of $\theta$-stable representations of $Q$ that respect the relations in $J$ by $\Mcal_\theta (Q,J)$. By sending a path $p = a_1 \ldots a_k$ to the monomial $y_{a_1}\cdots y_{a_k} \in \CC  \left[  y_a \mid a \in Q_1 \right]$ and extending linearly, we obtain a $\CC$-linear map from $\CC Q$ to $\CC  \left[  y_a \mid a \in Q_1 \right]$. We let $I_J$ be the ideal in $\CC  \left[  y_a \mid a \in Q_1 \right]$ generated by the image of $J$ under this map, in which case $\Mcal_\theta (Q,J)$ is given by the GIT quotient 
\begin{equation}
\Mcal_\theta(Q,J) = \mathbb{V}(I_J) \git_\theta T = \Proj \left( \bigoplus_{j \geq 0} \left( \CC  \left[  y_a \mid a \in Q_1 \right]/I_J \right)_{j\theta} \right). 
\end{equation} 

Let $\Lcal = \{ L_0 , \ldots , L_r\}$ be a collection of non-isomorphic effective line bundles on a projective toric variety $X$ with $L_0 := \Ocal_X$. As $X$ is projective, if $\Hom(L_i , L_j) \neq 0$ then $\Hom(L_j , L_i) = 0$ and so we can assume $\Lcal$ is ordered such that $i < j$ when $\Hom(L_i , L_j) \neq 0$. The endomorphism algebra $\End (\bigoplus_i L_i)$ can be conveniently described by its \emph{quiver of sections} $Q$, whose vertices $Q_0 = \{ 0 , \ldots , r \}$ are the line bundles in $\Lcal$ and the number of arrows from vertex $i$ to $j$ for $i < j$ is given by the dimension of the cokernel of the map
\begin{equation}
\bigoplus_{i<k<j} \Hom(L_i , L_k) \otimes \Hom(L_k , L_j) \longrightarrow \Hom (L_i , L_j).
\end{equation}
A torus-invariant section $s \in \Hom(L_i , L_j)$ is \emph{irreducible} if it lies in this cokernel. Each section in a basis of the irreducible sections determines a divisor of zeroes, and these divisors label the arrows between vertex $i$ and $j$; we therefore denote $\div (a)$ for the divisor that labels the arrow $a \in Q_1$, and $\div(p) := \sum_{a \in \supp(p)} \div(a)$ for a path $p$. The corresponding labelling monomial is $x^{\div(p)} := \prod_{a \in \supp(p)} x^{\div(a)} \in \CC [x_\rho \mid \rho \in \Sigma(1)]$. Note that the quiver is acyclic and as the collection is effective, the quiver is connected and rooted at $0$. The arrow labels determine the two-sided ideal of relations $J$, generated by the set
\begin{equation}
\{ p_i - p_j \mid p_i, p_j \text{ paths in } Q, \  \textbf{t}(p_i) = \textbf{t}(p_j), \textbf{h}(p_i) = \textbf{h}(p_j), \ \div(p_i) = \div(p_j)\}. 
\end{equation}

\begin{lemma}\cite[Proposition 3.3]{CrSm}\label{lem:PathAlgEqualsEndomorphismAlg}
Let $Q$ be the quiver of sections for the collection $\Lcal$ above, with ideal of relations $J$. Then $\CC Q/J \cong \End(\bigoplus_i L_i )$. 
\end{lemma}
\noindent Each line bundle $L_i$ is equal to $\Ocal_X(D'_i)$ for some Cartier divisor $D'_i$ and we can construct $Q$ explicitly by computing the vertices of the polyhedron conv$(\NN^d \cap \deg^{-1} (D'_i - D'_j))$ for each $i \neq j \in Q_0$. The vertices correspond to the torus-invariant generators of $\Hom(L_i , L_j)$, from which we pick the irreducible sections. 

\begin{example}\label{ex:E1FourfoldQuiver}
Let $X$ be the smooth toric Fano fourfold $E_1$ in Example \ref{ex:E1Fourfold} and fix $m \gg 0$. Choose $\{D_4,D_5,D_6\}$ to be the basis of $\Pic(X)$; the exact sequence \eqref{exseq} for $X$ is \[ 
\begin{CD}   
    0@>>> M @>{\left[ \begin{smallmatrix}
1 & 0 & 0 & 0 \\
0 & 1 & 0 & 0 \\
0 & 0 & 1 & 0 \\
0 & 0 & 0 & 1 \\
-1& 0 & 0 & 0 \\
3 &-1 &-1 &-1 \\
2 &-1 &-1 &-1
\end{smallmatrix} \right]}>> \ZZ^7 @>{
 \left[ \begin{smallmatrix}
1 & 0 & 0 & 0 &1&0&0 \\
-3 & 1 & 1 & 1 &0&1&0 \\
-2 & 1 & 1 & 1 &0&0&1
\end{smallmatrix} \right]}>> \Pic(X)@>>> 0.
\end{CD}  \] 
Every line bundle $L_i$ in the collection \[ \Lcal = \{ \Ocal_X(iD_5+iD_6),\ \Ocal_X(D_4 +iD_5+iD_6), \ \Ocal_X(D_4 +jD_5 + (j+1)D_6) \mid 0 \leq i  \leq 3, \ 0 \leq j \leq 2 \}  \] is nef and $L_i^{-1} \in \mkD_m$, so $\Lcal$ is a strong exceptional collection by Lemma \ref{nefsec}. The quiver of sections $Q$ for this collection is given in Fig. \ref{fig:4fold10}.
$ \\ $

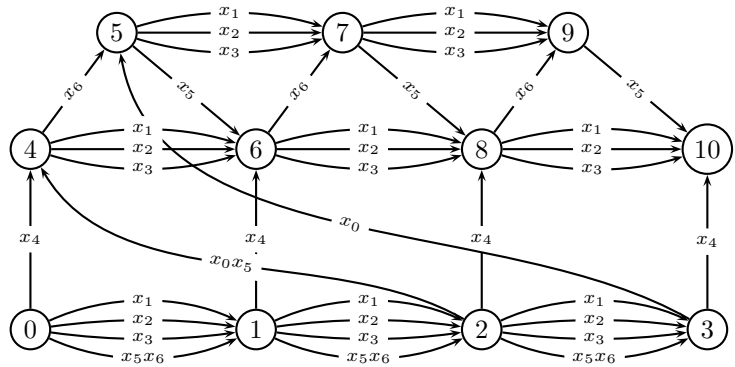
\begin{figure}[!ht]
\centering
  \psset{unit=1cm}
   \begin{pspicture}(0,-0.5)(9,4)
        \cnodeput(0,0){A}{\small{0}}
        \cnodeput(3,0){B}{\small{1}} 
        \cnodeput(6,0){C}{\small{2}}
        \cnodeput(9,0){D}{\small{3}} 
        \cnodeput(0,2.4){E}{\small{4}}
        \cnodeput(1.15,3.95){F}{\small{5}} 
        \cnodeput(3,2.4){G}{\small{6}} 
        \cnodeput(4.15,3.95){H}{\small{7}} 
        \cnodeput(6,2.4){I}{\small{8}}
        \cnodeput(7.15,3.95){J}{\small{9}}  
        \cnodeput(9,2.4){K}{\small{10}}   
    \psset{nodesep=0pt}
\nccurve[angleA=25,angleB=155]{->}{A}{B}\lput*{:U}{\tiny{$x_1$}}
\nccurve[angleA=8,angleB=172]{->}{A}{B}\lput*{:U}{\tiny{$x_2$}}
\nccurve[angleA=352,angleB=188]{->}{A}{B}\lput*{:U}{\tiny{$x_3$}}
\nccurve[angleA=335,angleB=205]{->}{A}{B}\lput*{:U}{\tiny{$x_5x_6$}}
\ncline{->}{A}{E}\lput*{:270}{\tiny{$x_4$}}
    
\nccurve[angleA=25,angleB=155]{->}{B}{C}\lput*{:U}{\tiny{$x_1$}}
\nccurve[angleA=8,angleB=172]{->}{B}{C}\lput*{:U}{\tiny{$x_2$}}
\nccurve[angleA=352,angleB=188]{->}{B}{C}\lput*{:U}{\tiny{$x_3$}}
\nccurve[angleA=335,angleB=205]{->}{B}{C}\lput*{:U}{\tiny{$x_5x_6$}}   
\ncline{->}{B}{G}\lput*{:270}{\tiny{$x_4$}} 
   
\nccurve[angleA=25,angleB=155]{->}{C}{D}\lput*{:U}{\tiny{$x_1$}}
\nccurve[angleA=8,angleB=172]{->}{C}{D}\lput*{:U}{\tiny{$x_2$}}
\nccurve[angleA=352,angleB=188]{->}{C}{D}\lput*{:U}{\tiny{$x_3$}}
\nccurve[angleA=335,angleB=205]{->}{C}{D}\lput*{:U}{\tiny{$x_5x_6$}}   
\ncline{->}{C}{I}\lput*{:270}{\tiny{$x_4$}}  
\nccurve[angleA=150,angleB=300]{->}{C}{E}\lput*{:180}{\tiny{$x_0x_5$}}

\ncline{->}{D}{K}\lput*{:270}{\tiny{$x_4$}}  
\nccurve[angleA=150,angleB=280]{->}{D}{F}\lput*{:180}{\tiny{$x_0$}}

\nccurve[angleA=17,angleB=163]{->}{E}{G}\lput*{:U}{\tiny{$x_1$}}
\ncline{->}{E}{G}\lput*{:U}{\tiny{$x_2$}}
\nccurve[angleA=343,angleB=197]{->}{E}{G}\lput*{:U}{\tiny{$x_3$}}
\ncline{->}{E}{F}\lput*{:U}{\tiny{$x_6$}}

\nccurve[angleA=17,angleB=163]{->}{F}{H}\lput*{:U}{\tiny{$x_1$}}
\ncline{->}{F}{H}\lput*{:U}{\tiny{$x_2$}}
\nccurve[angleA=343,angleB=197]{->}{F}{H}\lput*{:U}{\tiny{$x_3$}}
\ncline{->}{F}{G}\lput*{:U}{\tiny{$x_5$}}

\nccurve[angleA=17,angleB=163]{->}{G}{I}\lput*{:U}{\tiny{$x_1$}}
\ncline{->}{G}{I}\lput*{:U}{\tiny{$x_2$}}
\nccurve[angleA=343,angleB=197]{->}{G}{I}\lput*{:U}{\tiny{$x_3$}}
\ncline{->}{G}{H}\lput*{:U}{\tiny{$x_6$}}

\nccurve[angleA=17,angleB=163]{->}{H}{J}\lput*{:U}{\tiny{$x_1$}}
\ncline{->}{H}{J}\lput*{:U}{\tiny{$x_2$}}
\nccurve[angleA=343,angleB=197]{->}{H}{J}\lput*{:U}{\tiny{$x_3$}}
\ncline{->}{H}{I}\lput*{:U}{\tiny{$x_5$}}

\nccurve[angleA=17,angleB=163]{->}{I}{K}\lput*{:U}{\tiny{$x_1$}}
\ncline{->}{I}{K}\lput*{:U}{\tiny{$x_2$}}
\nccurve[angleA=343,angleB=197]{->}{I}{K}\lput*{:U}{\tiny{$x_3$}}
\ncline{->}{I}{J}\lput*{:U}{\tiny{$x_6$}}

\ncline{->}{J}{K}\lput*{:U}{\tiny{$x_5$}}     
\end{pspicture}
\caption{A quiver of sections on the smooth toric Fano fourfold $E_1$}
  \label{fig:4fold10}
  \end{figure}

\end{example} 

\subsection{Quiver Moduli and the Structure Sheaf of the Diagonal} Let $\iota \colon \Delta \hookrightarrow X \times X$ be the diagonal embedding and for two line bundles $L_1$ and $L_2$ on $X$ define \[ L_1 \boxtimes L_2 := p_1^*(L_1) \otimes p_2^*(L_2) \]
where $p_1$ and $p_2$ are the projections from $X \times X$ onto the first and second component respectively. We define the map $d_1$ of vector bundles on $X \times X$ as follows. Let $d_1$ have domain and codomain:
\begin{equation}\label{eq:d1Map}
d_1 \colon \bigoplus_{a \in Q_1} L_{\textbf{t}(a)} \boxtimes  L^{-1}_{\textbf{h}(a)}  \longrightarrow \bigoplus_{i \in Q_0} L_i \boxtimes L^{-1}_i. 
\end{equation}
The summands of the vector bundles are line bundles and so are given by twists of the $\Pic(X\times X)$-graded module $S_{X \times X}$. We write $S_{X\times X} = \CC [x_1, \ldots, x_{d}, w_1, \ldots , w_{d}]$ where $d = \vert \Sigma_X(1)\vert$ to distinguish sections $x_i$ on the first copy of $X$ in $X \times X$ from sections $w_i$ on the second copy. For line bundles $L_i$ and $L_j$ on $X$, denote $S_{X \times X}(L_i,L_j)$ to be the free $S_{X \times X}$-module generated by $\mathbf{e}_{L_i,L_j}$ corresponding to the line bundle $L_i \boxtimes L_j$ on $X \times X$ by \eqref{eq:FunctorGradedSModulesQCohSheaves}. Then our map $d_1$ sends
\begin{align*} S_{X\times X}(L_{\textbf{t}(a)},L^{-1}_{\textbf{h}(a)})  &\rightarrow S_{X\times X}(L_{\textbf{h}(a)},L^{-1}_{\textbf{h}(a)}) \oplus S_{X\times X}(L_{\textbf{t}(a)},L^{-1}_{\textbf{t}(a)})\\
\\
\mathbf{e}_{L_{\textbf{t}(a)},L^{-1}_{\textbf{h}(a)}} &\mapsto x^{\div(a)}\mathbf{e}_{L_{\textbf{h}(a)},L^{-1}_{\textbf{h}(a)}} -w^{\div(a)}\mathbf{e}_{L_{\textbf{t}(a)},L^{-1}_{\textbf{t}(a)}}.
\end{align*}

\noindent The following proposition provides a condition as to when the cokernel of $d_1$ is $\Ocal_\Delta$. Note that our choice of $\theta$ in the proposition will depend on our collection $\Lcal$, as explained in the following subsection.

\begin{proposition}\label{prop:EmbeddingGivesCokernel}
Suppose that there exists a generic stability parameter $\theta$ and a closed immersion $\phi\colon X \hookrightarrow \Mcal_\theta (Q,J)$ such that $L_i \cong \phi^*(F_i)$ for $\ 0 \leq i \leq r$. Then the cokernel of $d_1$ in \eqref{eq:d1Map} is $\Ocal_\Delta$.
\end{proposition}

\begin{proof}
Assume that $\theta$ and $\phi$ satisfy the conditions in the proposition. For the opposite quiver with relations $(Q^{op},J^{op})$, the stability parameter $-\theta$ is generic and $\Mcal_\theta (Q,J) \cong \Mcal_{-\theta} (Q^{op},J^{op})$. In addition, we have a closed immersion of $X$ into $\Mcal_{-\theta} (Q^{op},J^{op})$ such that the tautological line bundles on $\Mcal_{-\theta} (Q^{op},J^{op})$ restrict to the line bundles $L_i^{-1}$ on $X$. A $\theta$-stable representation $W = (W_i, \psi_i)$ of $(Q,J)$ determines a $( -\theta )$-stable representation $W^* = (W^*_i,\psi^*_i)$ of $(Q^{op},J^{op})$ and so for a point $(x_1,x_2) \in X \times X \hookrightarrow \Mcal_\theta (Q,J) \times \Mcal_{-\theta} (Q^{op},J^{op})$, $x_1$ parametrises a isomorphism class of a $\theta$-stable representation $V := ( V_i,\phi_i )$, whilst $x_2$ parametrises the isomorphism class of a $(-\theta)$-stable representation $W^*$. Therefore, the map $d_1$ of vector bundles on $X \times X$ from \eqref{eq:d1Map} restricted to a point $(x_1,x_2) \in X \times X$ is given by
\[D \colon \bigoplus_{a \in Q_1} V_{\mathbf{t}(a)} \otimes  W^*_{\mathbf{h}(a)} \longrightarrow \bigoplus_{i \in Q_0} V_i \otimes W^*_i.\]

The map $D$ is dual to the map:
\begin{equation}\label{eq:MapDualToD1}
D^* \colon \bigoplus_{i \in Q_0} \Hom_\CC (V_i , W_i ) \longrightarrow \bigoplus_{a \in Q_1} \Hom_\CC (V_{\mathbf{t}(a)} , W_{\mathbf{h}(a)} )
\end{equation} 
given by $( \beta_i ) \mapsto (\beta_{\mathbf{h}(a)} \phi_a - \psi_a \beta_{\mathbf{t}(a)})$. The kernel $\ker(D^*)$ of this map is precisely the morphisms from $V$ to $W$. As $V$ and $W$ are $\theta$-stable, we have $\theta(V) = \theta(W) = 0$. If $f$ is a morphism in $\ker(D^*)$ then the image $\im (f)$ of $f$ is a quotient of $V$, hence $\theta (\im(f)) \leq 0$. However, $\im (f)$ also injects into $W$ implying that $\theta (\im(f)) \geq 0$, so $\theta (\im(f)) = 0$. Therefore, $f$ is either an isomorphism or the zero morphism. It follows that when $W = V$, the kernel of $D^*$ is canonically a copy of $\CC$. Using this observation, we see that away from the diagonal of $X \times X$ the cokernel of $d_1$ is rank zero as the representations $V$ and $W$ are not isomorphic, whilst at each point on the diagonal the cokernel restricts to a canonical copy of $\CC$. Therefore the cokernel of $d_1$ is $\Ocal_\Delta$.

\end{proof} 

\subsection{Nef And Non-Nef Collections} Our choice of the generic stability parameter used in Proposition \ref{prop:EmbeddingGivesCokernel} will depend on whether our chosen line bundles are nef or not. Firstly, assume that $\Lcal$ is a collection of nef line bundles on $X$ and recall the special stability parameter $\vartheta = (-r,1,1,\ldots,1)$. Craw and Smith \cite{CrSm} associate to $Q$ a projective toric variety $\vert \Lcal \vert \cong \Mcal_\vartheta (Q)$ called the \emph{multigraded linear series} of $\Lcal$. They define the morphism $\phi_\Lcal \colon X \rightarrow \vert \Lcal \vert$ which factors into
\begin{equation}
X \stackrel{\phi_\Lcal}{\longrightarrow} \Mcal_\vartheta (Q,J) \hookrightarrow \vert \Lcal \vert
\end{equation}
and \cite[Corollary 4.10]{CrSm} present criteria as to when $\phi_\Lcal$ is a closed embedding. For a line bundle $L$ on $X_\Sigma$, there is a natural bijection between $\deg^{-1}(L) \cap \NN^{\Sigma(1)}$ and the generators of $\Gamma(X_\Sigma,L)$. Define $P_L$ to be the polytope in $\RR^{\Sigma(1)}$ that is the convex hull of the lattice points $\deg^{-1}(L) \cap \NN^{\Sigma(1)}$.

\begin{proposition}\label{prop:SumPolytopeNefEmbedding}
Let $\Lcal$ be a collection of nef line bundles. If $L := \bigotimes_{L_i \in \Lcal} L_i$ is very ample and the Minkowski sum of the polytopes $\{P_{L_i} \mid L_i \in \Lcal \}$ is equal to $P_L$, then the morphism $\phi_\Lcal \colon X \rightarrow \vert \Lcal \vert$ is a closed embedding. In this case, we can recover the line bundles in $\Lcal$ as the restriction of the tautological bundle on $\vert \Lcal \vert$ to $X$.  
\end{proposition}

\begin{proof}
This is immediate from \cite[Corollary 4,10,Theorem 4.15]{CrSm}.
\end{proof}
\noindent Note that as the varieties that we are considering are smooth and toric, then any ample line bundle is very ample. 

If the collection $\Lcal$ contains a line bundle that is not nef then we cannot using the multigraded linear series construction to show that $X_\Sigma \subset \Mcal_\vartheta(Q,J)$. In the examples of toric Fano varieties where this is the case, we show that $X$ is a closed subvariety of $\Mcal_\theta(Q,J)$ for a generic stability parameter $\theta$ different to $\vartheta$, such that the tautological bundle on $\Mcal_\theta (Q,J)$ restricts to $\bigoplus_{i \in Q_0} L_i$ on $X$. To achieve this, we recall the construction of the toric variety $Y_\theta \subset \Mcal_\theta(Q,J)$ from \cite{CrSm}, (see also \cite{CrMT} and \cite{CrQV}).

Define the map
\[ \pi := (\inc,\div) \colon \ZZ^{Q_1} \rightarrow \Wt(Q) \oplus \ZZ^{\Sigma(1)} \]
with image $\ZZ(Q) := \pi (\ZZ^{Q_1})$ and subsemigroup $\NN(Q) := \pi(\NN^{Q_1})$. The projections $\pi_1 \colon \ZZ(Q) \rightarrow \Wt(Q)$ and $\pi_2 \colon \ZZ(Q) \rightarrow \ZZ^{\Sigma(1)}$ fit in to the commutative diagram
\begin{center}
\begin{tikzpicture}[scale=1.8]
\node (A) at (-1,2) {$\ZZ^{Q_1}$};
\node (B) at (0,1) {$\ZZ(Q)$};
\node (C) at (1,1) {$\Wt(Q)$};
\node (D) at (0,0) {$\ZZ^{\Sigma(1)}$};
\node (E) at (1,0) {$\Pic(X)$};
\path[->,font=\tiny]
(A) edge node[right]{$\pi$} (B)
(A) edge node[above right]{$\inc$} (C)
(A) edge node[below left]{$\div$} (D)
(B) edge node[below]{$\pi_1$} (C)
(B) edge node[right]{$\pi_2$} (D)
(C) edge node[right]{$\pic$} (E)
(D) edge node[below]{$\deg$} (E);
\end{tikzpicture}
\end{center}
where $\pic(\chi_i) := L_i, \ i \in Q_0$ is a group homomorphism. Let $\CC[\NN(Q)]$ and $\CC[\NN^{Q_1}]$ be the semigroup algebra defined by $\NN(Q)$ and $\NN^{Q_1}$ respectively. The surjective map of semigroup algebras $\pi_* \colon \CC[\NN^{Q_1}] \rightarrow \CC [\NN(Q)]$ induced by $\pi$ has kernel $I_Q$ that defines an affine toric subvariety $\mathbb{V}(I_Q) \subset \CC^{Q_1}$. We obtain a $T$-action on $\mathbb{V}(I_Q)$ via restriction of the $T$-action on $\CC^{Q_1}$. For a generic weight $\theta \in \Wt(Q)$, we have the categorical quotient
\[ Y_\theta := \mathbb{V}(I_Q) /\!/\!_\theta T = \Proj \left( \bigoplus_{j \geq 0 } \CC[\NN(Q)]_{j\theta} \right) \]
where $\CC[\NN(Q)]_\theta$ is the $\theta$-graded piece. The variety $Y_\theta$ is toric and is a closed subvariety of $\Mcal_\theta (Q,J)$.   

\begin{proposition}\label{prop:NonNefEmbedding}
Fix a generic $\theta \in \Wt(Q)$ such that $L := \pic(\theta)$ is an ample line bundle on $X$. If 
\begin{equation}\label{eq:NonNefSurjectivePolytope}
\deg^{-1}(L) \cap \NN^{\Sigma(1)} \subset \pi_2 \left(\pi_1^{-1}(\theta) \cap \NN(Q)\right)
\end{equation} then $X$ is a closed subvariety of $Y_\theta$. Furthermore, if $\theta$ and $\vartheta$ are in the same open GIT-chamber for the $T$-action on $\mathbb{V}(I_Q)$, then the tautological bundles on $\Mcal_\theta(Q)$ restricts to the line bundles $L_i$ on $X$.  
\end{proposition}

\begin{proof}
The morphism $(\pi_2)_*$ is equivariant under the action of $T$ and $\Hom_\ZZ (\Pic(X),\CC^\times)$ on $\mathbb{V}(I_Q)$ and $\CC^{\Sigma(1)}$ respectively as the diagram of lattice maps
\begin{equation}\label{YthetaLattices} \begin{CD}
\ZZ(Q) @>{\pi_1}>> \Wt(Q) \\
@V{\pi_2}VV @VV{\pic}V     \\
\ZZ^{\Sigma(1)} @>>{\deg}> \Pic(X) \\
\end{CD} \end{equation}
commutes, hence we obtain a rational map from $X$ to $Y_\theta$. As $L$ is ample, we have $X = \Proj\left(\bigoplus_{j \geq 0} \CC [x_\rho \mid \rho \in \Sigma_X(1) ]_{jL} \right)$ and so $X$ is a closed subvariety of $Y_\theta$ when the homomorphism of graded rings \[ (\pi_2)_* \colon \bigoplus_{j \geq 0} \CC [\NN (Q)]_{j\theta} \rightarrow \bigoplus_{j \geq 0} \CC [x_\rho \mid \rho \in \Sigma_X(1) ]_{jL}\]
induced from $\pi_2$ is surjective. The bundle $L$ is very ample as $X$ is smooth and toric, so $\bigoplus_{j \geq 0} \CC [x_\rho \mid \rho \in \Sigma_X(1) ]_{jL}$ is generated in the first graded piece and thus it is enough to check surjectivity on this piece, which follows from \eqref{eq:NonNefSurjectivePolytope}. 

Now assume that $\theta$ and $\vartheta$ are in the same open GIT-chamber for the $T$-action on $\mathbb{V}(I_Q)$. Then $Y_\theta \cong Y_\vartheta$, so $X$ is a closed subvariety of $Y_\vartheta$ and by the arguments in the proof of \cite[Theorem 4.15]{CrSm}, the tautological bundles on $\Mcal_\vartheta (Q)$ restrict to the line bundles $L_i$ on $X$. As the chosen weight $\theta'$ varies, the restriction of the tautological line bundles will change if and only if the $\theta'$-stable representations parametrised by points of $\mathbb{V}(I_Q)$ change; as $\theta$ and $\vartheta$ are in the same open GIT-chamber this is not the case, so the tautological bundles on $\Mcal_\theta(Q)$ restrict to the line bundles $L_i$ on $X$.
\end{proof}

\section{Generation of $\Dx$: Resolution of $\Ocal_\Delta$  (Method 2)}

\noindent Section \ref{sect:d1Map} introduced a map of vector bundles $d_1$ and gave methods to determine if the cokernel of $d_1$ is $\Ocal_\Delta$. This section uses $d_1$ to determine if $\Lcal$ generates $\Dx$.  

\subsection{Resolution of $\Ocal_\Delta$}

\noindent Let $X$ be a smooth projective toric variety and $\Lcal = \{L_0, \ldots, L_r\}$ be a collection of line bundles on $X$. For $\mathcal{E} \in \Dcal^b(X \times X)$, denote $\Phi^{\mathcal{E}}(-) := \mathbf{R}(p_1)_*(\mathcal{E} \stackrel{\mathbf{L}}{\otimes} p_2^* (-)) \colon \Dx \rightarrow \Dx$ to be the Fourier-Mukai transform with kernel $\mathcal{E}$.  
\begin{proposition}\label{prop:GenerationResolutionDiagonal}
If there exists an exact sequence of sheaves on $X \times X$ of the form: \[ 0 \rightarrow \Ecal_k \rightarrow \cdots \rightarrow \Ecal_1 \stackrel{d_1}{\rightarrow} \Ecal_0 \rightarrow \Ocal_\Delta \rightarrow 0 \]
where 
\begin{align*}
\Ecal_0 & = \bigoplus_{i \in Q_0} L_i \boxtimes L^{-1}_i, \\
\Ecal_1 & = \bigoplus_{a \in Q_1} L_{\textbf{t}(a)} \boxtimes L^{-1}_{\textbf{h}(a)}
\end{align*}
and\[\Ecal_t = \bigoplus_{L_i, L_j \in \Lcal} L_i^{r_{i,t}} \boxtimes L_j^{-s_{j,t}}, \text{ for }  2 \leq t \leq k \text{ and some fixed } r_{i,t}, s_{j,t} \in \ZZ_{\geq 0},\]
then $\Lcal$ classically generates $\Dx$.
\end{proposition}
\begin{proof}
Assume that we have a resolution of $\Ocal_\Delta$ as given in the proposition. It follows from the projection formula that $\Phi^{\Ocal_\Delta}$ is naturally isomorphic to the identity functor on $\Dx$. Therefore for any object $\Fcal \in \Dx$, the object $\Phi^{\Ocal_\Delta}(\Fcal) \cong \Fcal$ is classically generated by $\{ \Phi^{\Ecal_0}(\Fcal), \ldots, \Phi^{\Ecal_k}(\Fcal) \}$. As $\mathbf{R}(p_1)_* \circ p^*_2 (-) \cong \mathbf{R} \Gamma (-) \otimes \Ocal_X$ \cite[page 86]{Huyb} 
we have that \[\Phi^{\Ecal_t}(\Fcal) \cong \bigoplus_{L_i, L_j \in \Lcal}\mathbf{R} \Gamma (X, \Fcal \otimes L_j^{-s_{j,t}})\otimes L_i^{r_{i,t}}\] is an object in $\langle \bigoplus_{L_i \in \Lcal} L_i^{r_{i,t}} \rangle$ for all $0 \leq t \leq k$. As $\bigoplus_{L_i \in \Lcal} L_i^{r_{i,t}} \in \langle \Lcal \rangle$ for all $0 \leq t \leq k$, $\Lcal$ classically generates $\Dx$. 
\end{proof}

In order to find a resolution of the diagonal sheaf as in Proposition \ref{prop:GenerationResolutionDiagonal}, we first recall the approach taken by King \cite{King}. For the locally free sheaf $\Tcal = \bigoplus_{L_i \in \Lcal} L_i$ on $X$ such that $\Hom^i_X (\Tcal,\Tcal) = 0$ for $i \neq 0$, define $A := \End(\Tcal)$ and $\Tcal^\vee := \Hom_{\Ocal_X} (\Tcal,\Ocal_X)$.
Note that \[\pi_1^* (\Tcal) = \pi_1^*(\bigoplus_{L_i \in \Lcal} L_i) = \bigoplus_{L_i \in \Lcal}\pi_1^*( L_i)\] and \[\pi_2^*(\Tcal^\vee) = \pi_2^*(\bigoplus_{L_i \in \Lcal} L_i^{-1}) =  \bigoplus_{L_i \in \Lcal}\pi_2^*( L_i^{-1}).\] By Lemma \ref{lem:PathAlgEqualsEndomorphismAlg}, $A$ is isomorphic to $\CC Q/J$ for some quiver with relations $(Q,J)$. The following gives the final part of a minimal projective $A,A$-bimodule resolution of $A$ \cite{King}.

\begin{lemma}\label{lem:QuiverResn}
Let $A = \CC Q/J$ and $\{e_i \ \vert \ i \in Q_0 \}$ be the indecomposable orthogonal idempotents. The following complex of $A,A$-bimodules gives the final part of the minimal projective resolution of A.
\begin{equation}
\bigoplus_{a \in Q_1} Ae_{\textbf{t}(a)} \otimes [a] \otimes e_{\textbf{h}(a)} A \longrightarrow \bigoplus_{i \in Q_0} Ae_i \otimes [i] \otimes e_i A
\end{equation}    
where $[\rho], \ [a]$ and $[i]$ are formal symbols. The map in the sequence is determined by
\begin{center}
$\begin{array}{lll}
e_{\textbf{t}(a)} \otimes [a] \otimes e_{\textbf{h}(a)} &  \mapsto & a \otimes [\textbf{h}(a)] \otimes e_{\textbf{h}(a)} - e_{\textbf{t}(a)} \otimes [\textbf{t}(a)] \otimes a
\end{array}$
\end{center}
and the map onto $A$ is $e_i \otimes [i] \otimes e_i \mapsto e_i$.
\end{lemma}

\noindent Given a projective $A,A$-bimodule resolution $P^\bullet$ of $A$, define $\Tcal \stackrel{\mathbf{L}}{\boxtimes}_A \Tcal^\vee$ to be the object 
\begin{equation}\label{eq:TboxTdualasResolution}
p_1^*(\Tcal) \otimes_A P^\bullet \otimes_A p_2^*(\Tcal^\vee)
\end{equation} in $\Dcal^b(X \times X)$. Using Lemma \ref{lem:QuiverResn}, the final map in this chain complex is the map $d_1$ from \eqref{eq:d1Map}.
 
\begin{lemma}\label{lem:KingsResOfDiag}\cite[Theorem 1.2]{King} If the cokernel of the map $d_1$ in the chain complex $\Tcal \stackrel{\mathbf{L}}{\boxtimes}_A \Tcal^\vee$ is $\Ocal_\Delta$, then  $\Tcal$ is a classical generator of $\Dx$. 
\end{lemma}

Although the final part of the minimal projective $A,A$-bimodule resolution is given by Lemma \ref{lem:QuiverResn}, the full resolution is not known in general and so one cannot compute $\Tcal \stackrel{\mathbf{L}}{\boxtimes}_A \Tcal^\vee$. What we do instead is guess what the resolution of $A$ is and then consider the sheafified version of the resolution as a chain complex $S^\bullet$
\begin{equation}\label{eq:ExactSeqClSMod}
0 \rightarrow S_k \rightarrow \cdots \rightarrow S_2 \rightarrow S_1 \stackrel{d_1}{\rightarrow} S_0
\end{equation} of $\Pic(X \times X)$-graded $S_{X \times X}$-modules, where $S_0$ is the $S_{X \times X}$ module corresponding to \[\bigoplus_{i \in Q_0} L_i \boxtimes L^{-1}_i\]and $S_1$ is the $S_{X \times X}$ module corresponding to \[\bigoplus_{a \in Q_1} L_{\textbf{t}(a)} \boxtimes L^{-1}_{\textbf{h}(a)}.\]If $S^\bullet$ is exact up to saturation by the irrelevant ideal $B_{X \times X}$ then it determines an exact sequence of sheaves on $X \times X$ by \eqref{eq:FunctorGradedSModulesQCohSheaves}.

We guess the construction of $S^\bullet$ by using the concept of the \emph{toric cell complex} introduced by Craw--Quintero-V\'{e}lez \cite{CrQV}. This is a combinatorial geometric structure that encodes the minimal projective bimodule resolution for certain classes of algebras; in particular, Calabi-Yau algebras in dimension 3 obtained from consistent dimer models. Given a collection $\Ecal = \{E_0, \ldots, E_r\}$ of rank one reflexive sheaves on a Gorenstein affine toric variety $Y$, the associated \emph{toric algebra} is $\End(\bigoplus_{i = 0}^r E_i)$. Craw--Quintero-V\'{e}lez state the following conjecture for consistent toric algebras:

\begin{conjecture}\cite[Conjecture 6.4]{CrQV}
Assume that the toric algebra associated to $\Ecal$ is consistent. If the global dimension of the algebra equals the dimension of $Y$, then the toric cell complex exists and is constructed as in \cite{CrQV}, from which the minimal projective bimodule resolution of the toric algebra can be recovered.
\end{conjecture}
\noindent Although the endomorphism algebra of a tilting bundle $\Tcal$ on a toric Fano variety $X$ is not Calabi-Yau, the pullback $\pi^*(\Tcal)$ on the total space $\tot(\omega_X)$ of the canonical bundle is, so we guess the resolution of this bundle on $\tot(\omega_X)$ and then restrict it to $X$.
  
In what follows, we define a combinatorial method to guess the resolution of the diagonal sheaf by $\Lcal$ based on the construction in \cite{CrQV}. Although the calculations are lengthy and tedious, many of the steps can be achieved using a computer algorithm, the results of which are contained in \cite{Prna1}.

\subsection{The Toric Cell Complex}\label{sect:ToricCellComplex}

For a smooth $n$-dimensional Fano toric variety $X$, set $Y := \tot(\omega_X)$ to be the total space of the canonical bundle on $X$. A collection of line bundles $\Lcal$ on $X$ defines a collection of line bundles $\Lcal_Y$ on $Y$ by pulling back along $\tot(\omega_X) \rightarrow X$. The Picard lattice $\Pic(Y)$ is isomorphic to the quotient of $\Pic(X)$ by the subgroup generated by $\omega_X$. Let $Q'$ be the quiver of sections associated to $\Lcal_Y$ and $B = \End(\bigoplus_{L \in \Lcal_Y} L)$. The quiver $Q'$ is cyclic and naturally embeds into $\Pic(Y)_\RR$. As $Y$ is a toric variety, it has a fan $\Sigma'$ and we have the exact sequence 
\[ \begin{CD} 
0 @>>> M' @>>> \ZZ^{\Sigma'(1)} @>{\deg}>> \Pic(Y) @>>> 0. \end{CD} \]

\begin{definition}
Let $Q'$ be the quiver above. Define $\tilde{Q'_0}$ to be the set $\bigoplus_{i \in Q'_0} \deg^{-1} (i) \subset \ZZ^{\Sigma(1)}$ and for every arrow $a \in Q'_1$ from $i$ to $j$ and each vertex $u \in \deg^{-1} (i)$, define the arrow $\tilde{a}$ in the set $\tilde{Q'_1}$ to be the arrow from $u$ to $u + \div(a) \in \deg^{-1}(j)$. The \emph{covering quiver} $\tilde{Q'}$ is the quiver in $\ZZ^{\Sigma(1)}$ with vertex set $\tilde{Q'_0}$ and arrow set $\tilde{Q'_1}$.   
\end{definition}

The projection $f\colon \RR^{\Sigma'(1)} \rightarrow M'_\RR \cong \RR^{n+1}$ restricts to $f\vert_{\ZZ^{\Sigma'(1)}} \colon \ZZ^{\Sigma'(1)} \rightarrow \RR^{n+1}$ and so we have the diagram
\[ \begin{CD} 
0 @>>> M' @>>> \ZZ^{\Sigma'(1)} @>{\deg}>> \Pic(Y) @>>> 0
\\
@.   @|            @VV{f\vert_{\ZZ^{\Sigma'(1)}}}V   @.     @.          
\\
0 @>>> M' @>>> \RR^{n+1}  @>>> \mathbb{T}^{n+1}   @>>> 0 \\
\end{CD} \]
where $\mathbb{T}^{n+1} := \RR^{n+1} / M'$ is a real $(n+1)$-torus.   

If $\Lcal$ is a full strong exceptional collection, Craw--Quintero-V\'{e}lez \cite[Conjecture 6.4]{CrQV} conjecture that the image of the arrows $a \in \tilde{Q'}$ in $\mathbb{T}^{n+1}$ under the map $f$ decomposes $\mathbb{T}^{n+1}$ into a toric cell complex, comprising of \emph{$k$-cells} for $0 \leq k \leq n+1$. The minimal $B,B$-bimodule projective resolution of $B$ that is expected to be encoded by the toric cell complex has maps determined by differentiating $k$-cells with respect to $(k-1)$-cells, for $1 \leq k \leq n+1$. To any cell $\eta$ in the toric cell complex, there is a well-defined divisor $\div(\eta)$ and monomial $x^{\div(\eta)} \in S_{Y}$ associated to it. We attempt to construct the exact sequence \eqref{eq:ExactSeqClSMod} by considering what the image in $S_{Y \times Y}$ is of the maps determined by cell differentiation.

An \emph{anticanonical cycle} in $Q'$ is a path $p$ such that $x^{\div(p)} = \prod_{\rho \in \Sigma'(1)} x_\rho$. Define $W$ to be the sum of all anticanonical cycles in $Q'$. For two paths $p$ and $q$ in $Q'$, the partial left derivative of $p$ with respect to $q$ is
\[ \partial_q p := 
\begin{cases}
r & \text{if } p = rq, \\
0 & \text{otherwise}
\end{cases}\] which can be extended by $\CC$-linearity to determine partial derivatives in $\CC Q'$. Let 
\[\mathcal{P} := \left\lbrace
q \in Q' \phantom{a} 
\begin{array}{|c}
\partial_q W \text{ is the sum of precisely two paths} \\ 
\text{that share neither initial or final arrow}
\end{array}
\right\rbrace \]
and
\[\mathcal{J} := \{ (p^+,p^-) \mid p^\pm \in \CC Q', \exists q \in \mathcal{P} \text{ such that } \partial_qW = p^+ + p^- \}. \] 
Assume now that the dimension of $X$ is $4$. We define the following sets:
\begin{center}
$\begin{array}{lcr}
\Gamma'_0:= \tilde{Q'_0},
& \Gamma'_1:= \tilde{Q'_1},
& \Gamma'_2:= \mathcal{J}.
\end{array}$
\end{center}
For 
\begin{itemize}
\item  $(p^+,p^-)  \in \Gamma'_2$, define $D_{p^+p^-} := \{ p \text{ a path in } \tilde{Q} \mid p \text{ is a summand in } \partial_{p^+} W \text{ or } \partial_{p^-} W\}$,
\item $a  \in \Gamma'_1$, define $D_{a} := \{ p \text{ a path in } \tilde{Q} \mid p \text{ is a summand in } \partial_{a} W \}$,
\item $i  \in \Gamma'_0$, define $D_{i} := \{ p \text{ a path in } \tilde{Q} \mid p \text{ is a summand in } \partial_{e_i} W \}$.
\end{itemize}   
Then let
\begin{center}
$\begin{array}{lcr}
\Gamma'_3:= \{ D_{p^+p^-} \mid (p^+,p^-) \in \Gamma'_2\}, 
& \Gamma'_4:= \{ D_{a} \mid a \in \Gamma'_1\},
& \Gamma'_5:= \{ D_i \mid i \in \Gamma'_0\}.  
\end{array}$
\end{center}
\begin{remark}\label{rem:DualToricCells}
A set of paths $P \in \Gamma'_k$ is expected to be the $1$-skeleton contained in a $k$-cell in the toric cell complex for $\tilde{Q}'$, if the toric cell complex exists. The construction of $\Gamma'_3, \ \Gamma'_4$ and $\Gamma'_5$ follow from the conjecture on duality between $k$-cells and $(n-k)$-cells in \cite[Conjecture 6.5]{CrQV}. For brevity we will therefore refer to $P$ as a \emph{$k$-cell}.
\end{remark}

Let $P \in \Gamma'_k$ for $1 \leq k \leq 5$ and $p \in P$ be a path. We define the \emph{head}, \emph{tail} and \emph{label} of $P$ as $\mathbf{h}(P) := \mathbf{h}(p) \in \Gamma'_0$, $\mathbf{t}(P) := \mathbf{t}(p) \in \Gamma'_0$ and $\div(P) := \div(p)$, and note that the definitions do not depend on our choice of $p$. For $P' \in \Gamma'_{k-i}, P \in \Gamma'_k,\ 0 \leq i < k \leq 5$ we write $P' \subset P$ if for every path $p \in P'$, there is a path $q \in P$ such that $p\subset q$. If $P' \in \Gamma'_{k-1}, P \in \Gamma'_k$ and $P' \subset P$, then a path $q \in P$ containing a path $p \in P'$ defines a monomial $\overleftarrow{\partial}_{p}q = x^{\div(p')} \in \CC[x_0, \ldots, x_d] \cong S_Y$ given by the label of the subpath $p' \subset p$ from $\mathbf{t}(P)$ to $\mathbf{t}(P')$, and a monomial $\overrightarrow{\partial}_{p}q = w^{\div(p'')} \in \CC[w_0, \ldots, w_d] \cong S_Y$ given by the label of the subpath $p'' \subset p$ from $\mathbf{h}(P')$ to $\mathbf{h}(P)$. Let $\mathcal{R}_{P',P}$ be the set of equivalence classes \[ \{[(p,q)] \mid (p,q) \in P' \times P,\ p \subset q \} \] where \[ (p,q) \sim (p',q') \Leftrightarrow \overleftarrow{\partial}_{p}q = \overleftarrow{\partial}_{p'}q', \] and note that \[ \overleftarrow{\partial}_{p}q = \overleftarrow{\partial}_{p'}q' \Leftrightarrow \overrightarrow{\partial}_{p}q = \overrightarrow{\partial}_{p'}q'. \] For $P' \in \Gamma'_{k-1}$ and $P \in \Gamma'_k$, define
\[\partial_{P'}P := \sum_{[(p,q)] \in \mathcal{R}_{P',P}} (\overleftarrow{\partial}_{p}q, -\overrightarrow{\partial}_{p}q) \in S_Y  \times S_Y  \cong S_{Y \times Y} \] if $P' \subset P$ and $0$ otherwise. The definition of $\partial_{P'}P$ does not depend on the choice of representatives for the equivalence classes in $\mathcal{R}_{P',P}$. We now have the maps
\[
d'_k := (\partial_{P'}P)_{\left\lbrace \substack{P' \in \Gamma'_{k-1} \\ P \in \Gamma'_k} \right\rbrace} \colon  (S_{Y \times Y})^{\vert \Gamma'_k \vert}  \longrightarrow  (S_{Y \times Y})^{\vert \Gamma'_{k-1} \vert}, \ \ 0 < k \leq 5,\]\[
 \mathbf{e}_P   \mapsto \bigoplus_{P' \in \Gamma'_{k-1}} (\partial_{P'}P)\mathbf{e}_{P'}. 
\]

\begin{remark}\label{rem:cellularMapsMultipleAppearances}
The derivatives $\partial_{P'}P$ are defined differently to how they are defined in \cite{CrQV} as a cell $P'$ may `appear' more than once in $P$ (see Example \ref{ex:E1FanoFourfoldGeneration}). Consequently, the property \cite[(4.3)]{CrQV} does not hold and we do not immediately obtain an incidence function $\varepsilon$ (see \cite[(4.4)]{CrQV}) that determines signs in the differentiations of cells.
\end{remark}

The construction above can be restricted to the Fano $X$ as follows. For any cone $\sigma \subset \Sigma_X \subset N_\RR$, define 
\[\sigma' := \Cone ((0,1),(u_\rho,1) \mid u_\rho \in \sigma(1)) \subset N_\RR \times \RR.\]
The fan $\Sigma' \subset N_\RR \times \RR$ that has cones given by $\sigma'$ for all $\sigma \subset \Sigma_X$ is the fan for $Y$ \cite[Proposition 7.3.1]{CoLiSc}. The toric divisors for $X$ are in one-to-one correspondence with the divisors for $Y$ minus the divisor determined by the ray with ray generator $(0,1)$, which we label $\rho_{\tot}$. Define the subsets $\Gamma_k := \{P \in \Gamma'_k \mid \text{for all } p \in P, x^{\rho_{\tot}} \nmid x^{\div(p)} \} \subset \Gamma'_k, \ 0 \leq k \leq 5$. Then the maps $d'_k$ restrict to
\begin{equation}\label{eq:CellularMapsOnFano}
d_k := (\partial_{P'}P)_{\left\lbrace \substack{P' \in \Gamma_{k-1} \\ P \in \Gamma_k} \right\rbrace} \colon (S_{X \times X})^{\vert \Gamma_k \vert} \longrightarrow (S_{X \times X})^{\vert \Gamma_{k-1} \vert}, \ \ 0 < k \leq 4,
\end{equation}\[
 \mathbf{e}_P   \mapsto \bigoplus_{P' \in \Gamma_{k-1}} (\partial_{P'}P)\mathbf{e}_{P'}. 
\] 
The $(S_{X \times X})$-modules $(S_{X \times X})^{\vert \Gamma_0 \vert}$ and $(S_{X \times X})^{\vert \Gamma_1 \vert}$ are graded as follows: for $i \in \Gamma_0$ and $a \in \Gamma_1$, let $S_{X \times X}^i \subset (S_{X \times X})^{\vert \Gamma_0 \vert}$ be given by $S_{X \times X}(L_i,L_i^{-1})$ and $S_{X \times X}^a \subset (S_{X \times X})^{\vert \Gamma_1 \vert}$ be given by $S_{X \times X}(L_{\mathbf{t}(a)},L_{\mathbf{h}(a)}^{-1})$. Then $(S_{X \times X})^{\vert \Gamma_0 \vert}$ and $(S_{X \times X})^{\vert \Gamma_1 \vert}$ correspond to the bundles $\bigoplus_{i \in Q_0} L_i \boxtimes L^{-1}_i$ and $\bigoplus_{a \in Q_1} L_{\textbf{t}(a)} \boxtimes L^{-1}_{\textbf{h}(a)}$ respectively, so the map $d_1$ in \eqref{eq:CellularMapsOnFano} is the map given in \eqref{eq:d1Map}. Similarly for $2 \leq k \leq 4$, the modules $(S_{X \times X})^{\vert \Gamma_k \vert}$ are graded so that they correspond to $\bigoplus_{L_i, L_j \in \Lcal} L_i^{r_{i,k}} \boxtimes L_j^{-s_{j,k}}$, for some fixed $r_{i,k}, s_{j,k} \in \ZZ_{\geq 0}$. We attempt to add signs to entries in the maps $d_k$ for $2 \leq k \leq 4$ so that we get a $\Pic(X \times X)$-graded chain complex of $S_{X \times X}$-modules
\begin{equation}
0 \longrightarrow (S_{X \times X})^{\vert \Gamma_4 \vert} \stackrel{d_4}{\longrightarrow} (S_{X \times X})^{\vert \Gamma_3 \vert} \stackrel{d_3}{\longrightarrow} (S_{X \times X})^{\vert \Gamma_2 \vert} \stackrel{d_2}{\longrightarrow} (S_{X \times X})^{\vert \Gamma_1 \vert} \stackrel{d_1}{\longrightarrow} (S_{X \times X})^{\vert \Gamma_0 \vert}
\end{equation} 

\noindent It then needs to be checked that the chain complex is exact up to saturation by the irrelevant ideal $B_{X\times X}$ in order to show that the chain complex determines an exact sequence of sheaves on $X \times X$.

\begin{remark}
A similar construction can be obtained for a smooth toric Fano threefold $X$. In this case, $\tot(\omega_X)$ is of dimension $4$, and so the $4$-cells in $\Gamma'_4$ are given as the dual cells to the $0$-cells in $\Gamma'_0$, the $3$-cells in $\Gamma'_3$ are computed from the $1$-cells in $\Gamma'_1$ and the set of $2$-cells is self-dual (see Remark \ref{rem:DualToricCells}).  
\end{remark} 

Using the method above and for the database of full strong exceptional collections of line bundles in \cite{Prna1}, the exact sequence of sheaves $S^\bullet$ from \eqref{eq:ExactSeqClSMod} has been computed for all smooth toric Fano threefolds and $88$ of the $124$ smooth toric Fano fourfolds. These exact sequences are contained in a database in \cite{Prna1}, and the fact that they can be computed leads to the following conjecture:

\begin{conjecture}
Let $X$ be a smooth toric Fano threefold or one of the $88$ smooth toric Fano fourfolds such that the given full strong exceptional collection $\Lcal$ in the database \cite{Prna1} has a corresponding exact sequence of sheaves $S^\bullet \in \Dcal^b(X \times X)$. Let $B$ denote the rolled up helix algebra of $A$. Then the toric cell complex of $B$ exists and is supported on a real five-dimensional torus. Moreover,
\begin{itemize}
\item the cellular resolution exists in the sense of \cite{CrQV}, thereby producing the minimal projective bimodule resolution of $B$;
\item the object $S^\bullet$ is quasi-isomorphic to $\Tcal \stackrel{\mathbf{L}}{\boxtimes}_A \Tcal^\vee \in \Dcal^b(X \times X)$, for $\Tcal := \bigoplus_{L \in \Lcal} L$. 
\end{itemize}
\end{conjecture}

\subsection{Method 2}\label{subsect:Method 2} 

Given a strong exceptional collection $\Lcal$ of line bundles on $X$ with associated quiver $Q$, we can thus proceed as follows to show that $\Lcal$ generates $\Dx$:
\begin{itemize}
\item[Step 1:] Using the method described in Section \ref{sect:ToricCellComplex}, construct the chain complex of $\Pic(X \times X)$-graded $S_{X \times X}$-modules \eqref{eq:ExactSeqClSMod} such that $S_t$ determines the sheaf $\Ecal_t$, where \[\Ecal_t = \bigoplus_{L_i, L_j \in \Lcal} L_i^{r_{i,t}} \boxtimes L_j^{-s_{j,t}}, \text{ for }  2 \leq t \leq 4 \text{ and some fixed } r_{i,t}, s_{j,t} \in \ZZ_{\geq 0}.\] Check that this chain complex is exact up to saturation by $B_{X \times X}$.

\item[Step 2:] If $\Lcal$ is a collection of nef line bundles then
\begin{itemize}
\item check that the line bundle $L = \bigotimes_{L_i \in \Lcal} L_i$ is ample;
\item show that the Minkowski sum of the polytopes $\{ P_{L_i} \mid L_i \in \Lcal\}$ is equal to $P_L$.
\end{itemize}
Then Propositions \ref{prop:SumPolytopeNefEmbedding} and \ref{prop:EmbeddingGivesCokernel} imply that the exact sequence of sheaves computed in \emph{Step 1} is a resolution of $\Ocal_\Delta$, so $\langle \Lcal \rangle = \Dx$ by Proposition \ref{prop:GenerationResolutionDiagonal}.
\\
If $\Lcal$ contains non-nef line bundles then
\begin{itemize}
\item choose a weight $\theta \in \text{Wt}(Q)$ such that $\pic(\theta)$ is ample and construct $Y_\theta$;
\item check that $\theta$ is generic. By the first step of the proof of \cite[Lemma 4.2]{BCQV}, it is enough to show that the representations corresponding to each torus-invariant point of $Y_\theta$ are $\theta$-stable. Confirm that $\theta$ and $\vartheta$ are in the same open GIT-chamber for $Y_\theta$;
\item show $\deg^{-1}(L) \cap \NN^{\Sigma(1)} \subset \pi_2 \left(\pi_1^{-1}(\theta) \cap \NN(Q)\right).$
\end{itemize} Then Propositions \ref{prop:NonNefEmbedding} and \ref{prop:EmbeddingGivesCokernel} imply that the exact sequence of sheaves computed in \emph{Step 1} is a resolution of $\Ocal_\Delta$, so $\langle \Lcal \rangle = \Dx$ by Proposition \ref{prop:GenerationResolutionDiagonal}.
\end{itemize}

\noindent An example of this construction for a collection of nef line bundles on the birationally maximal smooth toric Fano fourfold $E_1$ is given in Example \ref{ex:E1FanoFourfoldGeneration}, whilst a collection that contains non-nef line bundles on $J_1$ is shown to be full using this method in Example \ref{ex:J1FanoFourfoldGeneration}.

\begin{example}\label{ex:E1FanoFourfoldGeneration}
Using the variety $X$ and the collection of line bundles $\Lcal$ in Example \ref{ex:E1FourfoldQuiver}, let $Y = \tot(\omega_X)$ and $\Lcal_Y$ be the corresponding collection of line bundles on $Y$. The quiver of sections $Q'$ for $\Lcal_Y$ with and without arrow labels is given in Figure \ref{fig:CY4fold10}, where vertices with the same labels are identified.

\begin{figure}[]
\centering
\subfigure[Quiver of sections]{
  \psset{unit=0.7cm}
   \begin{pspicture}(0,-0.5)(9,4)
        \cnodeput(0,0){A}{\tiny{0}}
        \cnodeput(3,0){B}{\tiny{1}} 
        \cnodeput(6,0){C}{\tiny{2}}
        \cnodeput(9,0){D}{\tiny{3}} 
        \cnodeput(0,2.4){E}{\tiny{4}}
        \cnodeput(1.15,3.95){F}{\tiny{5}} 
        \cnodeput(3,2.4){G}{\tiny{6}} 
        \cnodeput(4.15,3.95){H}{\tiny{7}} 
        \cnodeput(6,2.4){I}{\tiny{8}}
        \cnodeput(7.15,3.95){J}{\tiny{9}}  
        \cnodeput(9,2.4){K}{\tiny{10}}
        \cnodeput(9,8){L}{\tiny{0}}
        \cnodeput(7.15,5.6){M}{\tiny{1}}
        \cnodeput(11,4.5){N}{\tiny{2}}   
    \psset{nodesep=0pt}
\nccurve[angleA=25,angleB=155]{->}{A}{B}\lput*{:U}{\tiny{$x_1$}}
\nccurve[angleA=0,angleB=180]{->}{A}{B}\lput*{:U}{\tiny{$x_2$}}
\nccurve[angleA=335,angleB=205]{->}{A}{B}\lput*{:U}{\tiny{$x_3$}}
\nccurve[angleA=310,angleB=230]{->}{A}{B}\lput*{:U}{\tiny{$x_5x_6$}}
\ncline{->}{A}{E}\lput*{:270}{\tiny{$x_4$}}
    
\nccurve[angleA=25,angleB=155]{->}{B}{C}\lput*{:U}{\tiny{$x_1$}}
\nccurve[angleA=0,angleB=180]{->}{B}{C}\lput*{:U}{\tiny{$x_2$}}
\nccurve[angleA=335,angleB=205]{->}{B}{C}\lput*{:U}{\tiny{$x_3$}}
\nccurve[angleA=310,angleB=230]{->}{B}{C}\lput*{:U}{\tiny{$x_5x_6$}}   
\ncline{->}{B}{G}\lput*{:270}{\tiny{$x_4$}} 
   
\nccurve[angleA=25,angleB=155]{->}{C}{D}\lput*{:U}{\tiny{$x_1$}}
\nccurve[angleA=0,angleB=180]{->}{C}{D}\lput*{:U}{\tiny{$x_2$}}
\nccurve[angleA=335,angleB=205]{->}{C}{D}\lput*{:U}{\tiny{$x_3$}}
\nccurve[angleA=310,angleB=230]{->}{C}{D}\lput*{:U}{\tiny{$x_5x_6$}}   
\ncline{->}{C}{I}\lput*{:270}{\tiny{$x_4$}}  
\nccurve[angleA=148,angleB=290]{->}{C}{E}\lput*[opacity = .1]{:180}{\tiny{$x_0x_5$}}

\ncline{->}{D}{K}\lput*{:270}{\tiny{$x_4$}}  
\nccurve[angleA=150,angleB=280]{->}{D}{F}\lput*{:180}{\tiny{$x_0$}}

\nccurve[angleA=23,angleB=157]{->}{E}{G}\lput*{:U}{\tiny{$x_1$}}
\ncline{->}{E}{G}\lput*{:U}{\tiny{$x_2$}}
\nccurve[angleA=337,angleB=203]{->}{E}{G}\lput*{:U}{\tiny{$x_3$}}
\ncline{->}{E}{F}\lput*{:U}{\tiny{$x_6$}}

\nccurve[angleA=23,angleB=157]{->}{F}{H}\lput*{:U}{\tiny{$x_1$}}
\ncline{->}{F}{H}\lput*{:U}{\tiny{$x_2$}}
\nccurve[angleA=337,angleB=203]{->}{F}{H}\lput*{:U}{\tiny{$x_3$}}
\ncline{->}{F}{G}\lput*{:U}{\tiny{$x_5$}}

\nccurve[angleA=23,angleB=157]{->}{G}{I}\lput*{:U}{\tiny{$x_1$}}
\ncline{->}{G}{I}\lput*{:U}{\tiny{$x_2$}}
\nccurve[angleA=337,angleB=203]{->}{G}{I}\lput*{:U}{\tiny{$x_3$}}
\ncline{->}{G}{H}\lput*{:U}{\tiny{$x_6$}}

\nccurve[angleA=23,angleB=157]{->}{H}{J}\lput*{:U}{\tiny{$x_1$}}
\ncline{->}{H}{J}\lput*{:U}{\tiny{$x_2$}}
\nccurve[angleA=337,angleB=203]{->}{H}{J}\lput*{:U}{\tiny{$x_3$}}
\ncline{->}{H}{I}\lput*{:U}{\tiny{$x_5$}}
\ncline{->}{H}{L}\lput*{:U}{\tiny{$x_4x_7$}}

\nccurve[angleA=23,angleB=157]{->}{I}{K}\lput*{:U}{\tiny{$x_1$}}
\ncline{->}{I}{K}\lput*{:U}{\tiny{$x_2$}}
\nccurve[angleA=337,angleB=203]{->}{I}{K}\lput*{:U}{\tiny{$x_3$}}
\ncline{->}{I}{J}\lput*{:U}{\tiny{$x_6$}}

\ncline{->}{J}{K}\lput*{:U}{\tiny{$x_5$}}

\nccurve[angleA=111,angleB=249]{->}{K}{L}\lput*{:180}{\tiny{$x_0x_1x_7$}}
\nccurve[angleA=97,angleB=263]{->}{K}{L}\lput*{:180}{\tiny{$x_0x_2x_7$}}     
\nccurve[angleA=83,angleB=277]{->}{K}{L}\lput*{:180}{\tiny{$x_0x_3x_7$}}
\nccurve[angleA=69,angleB=291]{->}{K}{L}\lput*{:180}{\tiny{$x_0x_5x_6x_7$}}

\ncline{->}{J}{M}\lput*{:270}{\tiny{$x_4x_7$}}
\ncline{->}{K}{N}\lput*{:U}{\tiny{$x_4x_6x_7$}}

\end{pspicture}
} \qquad \qquad
\subfigure[Listing the arrows]{
  \psset{unit=0.7cm}
   \begin{pspicture}(0,-0.5)(9,4)
        \cnodeput(0,0){A}{\tiny{0}}
        \cnodeput(3,0){B}{\tiny{1}} 
        \cnodeput(6,0){C}{\tiny{2}}
        \cnodeput(9,0){D}{\tiny{3}} 
        \cnodeput(0,2.4){E}{\tiny{4}}
        \cnodeput(1.15,3.95){F}{\tiny{5}} 
        \cnodeput(3,2.4){G}{\tiny{6}} 
        \cnodeput(4.15,3.95){H}{\tiny{7}} 
        \cnodeput(6,2.4){I}{\tiny{8}}
        \cnodeput(7.15,3.95){J}{\tiny{9}}  
        \cnodeput(9,2.4){K}{\tiny{10}}
        \cnodeput(9,8){L}{\tiny{0}}
        \cnodeput(7.15,5.6){M}{\tiny{1}}
        \cnodeput(11,4.5){N}{\tiny{2}}   
    \psset{nodesep=0pt}
\nccurve[angleA=25,angleB=155]{->}{A}{B}\lput*{:U}{\tiny{$a_1$}}
\nccurve[angleA=0,angleB=180]{->}{A}{B}\lput*{:U}{\tiny{$a_2$}}
\nccurve[angleA=335,angleB=205]{->}{A}{B}\lput*{:U}{\tiny{$a_3$}}
\nccurve[angleA=310,angleB=230]{->}{A}{B}\lput*{:U}{\tiny{$a_4$}}
\ncline{->}{A}{E}\lput*{:270}{\tiny{$a_5$}}
    
\nccurve[angleA=25,angleB=155]{->}{B}{C}\lput*{:U}{\tiny{$a_6$}}
\nccurve[angleA=0,angleB=180]{->}{B}{C}\lput*{:U}{\tiny{$a_7$}}
\nccurve[angleA=335,angleB=205]{->}{B}{C}\lput*{:U}{\tiny{$a_8$}}
\nccurve[angleA=310,angleB=230]{->}{B}{C}\lput*{:U}{\tiny{$a_9$}}   
\ncline{->}{B}{G}\lput*{:270}{\tiny{$a_{10}$}} 
   
\nccurve[angleA=25,angleB=155]{->}{C}{D}\lput*{:U}{\tiny{$a_{11}$}}
\nccurve[angleA=0,angleB=180]{->}{C}{D}\lput*{:U}{\tiny{$a_{12}$}}
\nccurve[angleA=335,angleB=205]{->}{C}{D}\lput*{:U}{\tiny{$a_{13}$}}
\nccurve[angleA=310,angleB=230]{->}{C}{D}\lput*{:U}{\tiny{$a_{14}$}}   
\ncline{->}{C}{I}\lput*{:270}{\tiny{$a_{16}$}}  
\nccurve[angleA=148,angleB=290]{->}{C}{E}\lput*[opacity = .1]{:180}{\tiny{$a_{15}$}}

\ncline{->}{D}{K}\lput*{:270}{\tiny{$a_{18}$}}  
\nccurve[angleA=150,angleB=280]{->}{D}{F}\lput*{:180}{\tiny{$a_{17}$}}

\nccurve[angleA=23,angleB=157]{->}{E}{G}\lput*{:U}{\tiny{$a_{20}$}}
\ncline{->}{E}{G}\lput*{:U}{\tiny{$a_{21}$}}
\nccurve[angleA=337,angleB=203]{->}{E}{G}\lput*{:U}{\tiny{$a_{22}$}}
\ncline{->}{E}{F}\lput*{:U}{\tiny{$a_{19}$}}

\nccurve[angleA=23,angleB=157]{->}{F}{H}\lput*{:U}{\tiny{$a_{24}$}}
\ncline{->}{F}{H}\lput*{:U}{\tiny{$a_{25}$}}
\nccurve[angleA=337,angleB=203]{->}{F}{H}\lput*{:U}{\tiny{$a_{26}$}}
\ncline{->}{F}{G}\lput*{:U}{\tiny{$a_{23}$}}

\nccurve[angleA=23,angleB=157]{->}{G}{I}\lput*{:U}{\tiny{$a_{28}$}}
\ncline{->}{G}{I}\lput*{:U}{\tiny{$a_{29}$}}
\nccurve[angleA=337,angleB=203]{->}{G}{I}\lput*{:U}{\tiny{$a_{30}$}}
\ncline{->}{G}{H}\lput*{:U}{\tiny{$a_{27}$}}

\nccurve[angleA=23,angleB=157]{->}{H}{J}\lput*{:U}{\tiny{$a_{32}$}}
\ncline{->}{H}{J}\lput*{:U}{\tiny{$a_{33}$}}
\nccurve[angleA=337,angleB=203]{->}{H}{J}\lput*{:U}{\tiny{$a_{34}$}}
\ncline{->}{H}{I}\lput*{:U}{\tiny{$a_{31}$}}

\ncline{->}{H}{L}\lput*{:U}{\tiny{$a_{40}$}}

\nccurve[angleA=23,angleB=157]{->}{I}{K}\lput*{:U}{\tiny{$a_{36}$}}
\ncline{->}{I}{K}\lput*{:U}{\tiny{$a_{37}$}}
\nccurve[angleA=337,angleB=203]{->}{I}{K}\lput*{:U}{\tiny{$a_{38}$}}
\ncline{->}{I}{J}\lput*{:U}{\tiny{$a_{35}$}}

\ncline{->}{J}{K}\lput*{:U}{\tiny{$a_{39}$}}

\nccurve[angleA=111,angleB=249]{->}{K}{L}\lput*{:180}{\tiny{$a_{44}$}}
\nccurve[angleA=97,angleB=263]{->}{K}{L}\lput*{:180}{\tiny{$a_{43}$}}     
\nccurve[angleA=83,angleB=277]{->}{K}{L}\lput*{:180}{\tiny{$a_{42}$}}
\nccurve[angleA=69,angleB=291]{->}{K}{L}\lput*{:180}{\tiny{$a_{41}$}}

\ncline{->}{J}{M}\lput*{:270}{\tiny{$a_{45}$}}
\ncline{->}{K}{N}\lput*{:U}{\tiny{$a_{46}$}}
\end{pspicture}
}
\caption{A quiver of sections on $\tot(\omega_X)$}
  \label{fig:CY4fold10}
  \end{figure}
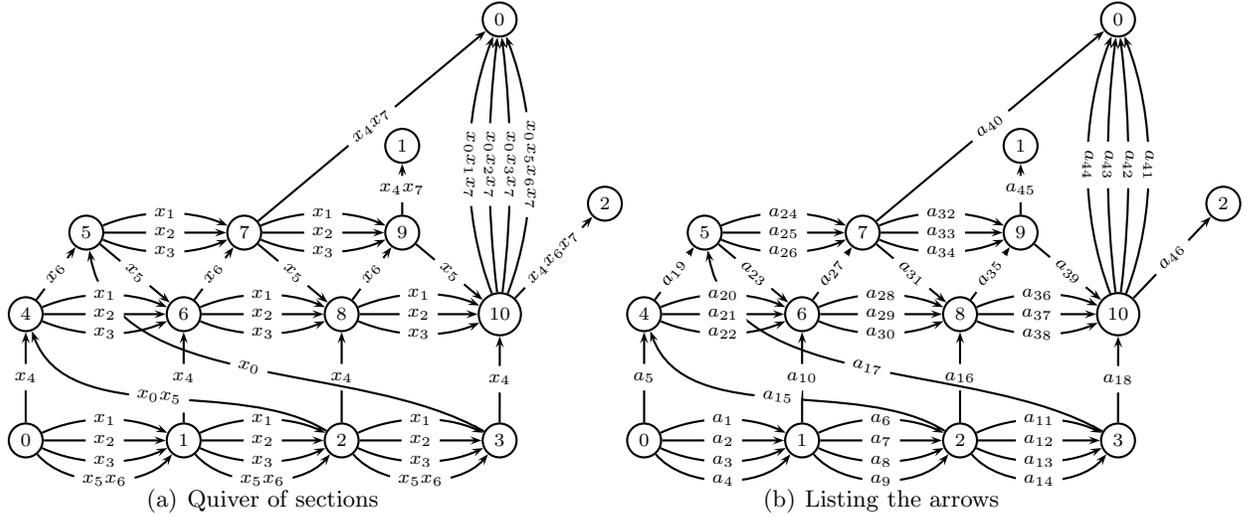

\noindent The set $J$ is
\[ \mathcal{J} = \left\lbrace \begin{array}{c}
(a_1a_7,a_2a_6),  (a_1a_8, a_3a_6),  (a_1a_9,a_4a_6),  (a_1a_{10},a_5a_{20}),  (a_{40}a_1,a_{32}a_{45}), \\ 
(a_{41}a_1,a_{44}a_4),  (a_{42}a_1,a_{44}a_3),  (a_{43}a_1,a_{44}a_2),  (a_{2}a_8,a_3a_7),  (a_2a_9,a_4a_7), \\
 \vdots  \\ 
(a_{36}a_{42},a_{38}a_{44}),(a_{36}a_{43}, a_{37}a_{44}),(a_{37}a_{42}, a_{38}a_{43})
\end{array} \right\rbrace \]
As $\vert Q'_0 \vert = 11, \vert Q'_1 \vert = 46$ and $\vert \mathcal{J} \vert = 83$, we have $\vert \Gamma'_0 \vert = \vert \Gamma'_5 \vert = 11, \vert \Gamma'_1 \vert = \vert \Gamma'_4 \vert = 46$ and $\vert \Gamma'_2 \vert = \vert \Gamma'_3 \vert = 83$. Note that $P' = a_{17} \in \Gamma'_1$ appears twice in $P = (a_{11}a_{17}a_{25},a_{12}a_{17}a_{24}) \in \Gamma'_2$ (see Remark \ref{rem:cellularMapsMultipleAppearances}). In this case, $\partial_{P'}P = -(x_1w_2 + x_2w_1)$.

The monomial for the the extra divisor in $Y$ is $x^{\rho_{\tot}} = x_7$, so the sets $\Gamma_k$ are composed of sets of paths that do not contain the any of the arrows in $\{ a_{40}, a_{41}, \ldots , a_{46} \}$. Via this restriction, we obtain the chain complex of $\Pic(X \times X)$-graded $(S_{X \times X})$-modules
\begin{equation}\label{eq:E1ExactSequenceOfSheaves}
 0 \rightarrow (S_{X \times X})^{7} \stackrel{d_4}{\rightarrow} (S_{X \times X})^{31} \stackrel{d_3}{\rightarrow} (S_{X \times X})^{52} \stackrel{d_2}{\rightarrow} (S_{X \times X})^{39} \stackrel{d_1}{\rightarrow} (S_{X \times X})^{11}.
\end{equation}
This complex is exact up to saturation by $B_{X \times X}$ \cite{Prna1} and so is an exact sequence of sheaves on $X \times X$.   

We now check that the cokernel of this sequence is $\Ocal_\Delta$. For this collection, the line bundle $L = \bigotimes_{L_i \in \Lcal} L_i = \Ocal_X (7D_4 + 15D_5 + 18D_6)$ is ample. Each column of the matrices below is a vertex of the convex polytope in $\RR^7$ for the corresponding line bundle in $\Lcal$:
\[ \Ocal_X(iD_5 + iD_6) : \left[ \begin{smallmatrix} 
0 & 0 & 0 & 0 \\
i & 0 & 0 & 0 \\
0 & i & 0 & 0 \\
0 & 0 & i & 0 \\
0 & 0 & 0 & 0 \\
0 & 0 & 0 & i \\
0 & 0 & 0 & i \end{smallmatrix}\right], \ 
\Ocal_X(D_4+iD_5+iD_6): \left[ \begin{smallmatrix}
0&0&0&1&1&1&0&1\\
i&0&0&i+2&0&0&0&0\\
0&i&0&0&i+2&0&0&0\\
0&0&i&0&0&i+2&0&0\\
1&1&1&0&0&0&1&0\\
0&0&0&1&1&1&i&i+3\\
0&0&0&0&0&0&i&i+2 \end{smallmatrix}\right],\]\[
\Ocal_X(D_4+jD_5+(j+1)D_6): \left[ \begin{smallmatrix}
1&1&1&0&0&0&0&1\\
j+3&0&0&j&0&0&0&0\\
0&j+3&0&0&j&0&0&0\\
0&0&j+3&0&0&j&0&0\\
0&0&0&1&1&1&1&0\\
0&0&0&0&0&0&j&j+3\\
0&0&0&1&1&1&j+1&j+3 \end{smallmatrix}\right], \
i = 0,1,2,3, \ j = 0,1,2.\]
The vertices for the polytope corresponding to $L$ is 
\[ \Ocal_X(7D_4 +15D_5 + 18D_6): \left[ \begin{smallmatrix}
3&3&3&7&7&7&0&0&0&0&7\\
24&0&0&32&0&0&15&0&0&0&0\\
0&24&0&0&32&0&0&15&0&0&0\\
0&0&24&0&0&32&0&0&15&0&0\\
4&4&4&0&0&0&7&7&7&7&0\\
0&0&0&4&4&4&0&0&0&15&36\\
0&0&0&0&0&0&3&3&3&18&32 \end{smallmatrix}\right].\]
The Minkowski sum of the polytopes $\{P_{L_i} \mid L_i \in \Lcal\}$ is equal to the polytope corresponding to $L$, so \eqref{eq:E1ExactSequenceOfSheaves} is a resolution of $\Ocal_\Delta$ by Propositions \ref{prop:SumPolytopeNefEmbedding} and \ref{prop:EmbeddingGivesCokernel}. Therefore, $\Lcal$ is full by Proposition \ref{prop:GenerationResolutionDiagonal}.
\end{example}

\begin{example}\label{ex:J1FanoFourfoldGeneration}
Let $X$ be the birationally maximal smooth toric Fano fourfold $J_1$. The primitive generators for the rays of $X$ are
\[u_0 = \begin{tiny}
\begin{bmatrix}
1 \\ 0 \\ 0 \\ 0
\end{bmatrix} \end{tiny} ,
u_1 = \begin{tiny}\begin{bmatrix}
0 \\ 1 \\ 0 \\ 0
\end{bmatrix} \end{tiny},
u_2 = \begin{tiny}\begin{bmatrix}
-1 \\ -1 \\ -1 \\ 0
\end{bmatrix} \end{tiny},
u_3 = \begin{tiny}\begin{bmatrix}
0 \\ 0 \\ 1 \\ 0
\end{bmatrix} \end{tiny},
u_4 = \begin{tiny}\begin{bmatrix}
0 \\ 0 \\ 0 \\ 1
\end{bmatrix} \end{tiny},
u_5 = \begin{tiny}\begin{bmatrix}
0 \\ 0 \\ -1 \\ -1
\end{bmatrix} \end{tiny},
u_6 = \begin{tiny}\begin{bmatrix}
-1 \\ 0 \\ 0 \\ 0
\end{bmatrix} \end{tiny},
u_7 = \begin{tiny}\begin{bmatrix}
-1 \\ 0 \\ 1 \\ 0
\end{bmatrix} \end{tiny}
\]

\noindent The collection of line bundles on $X$ 
\[ \Lcal = \left\lbrace
\begin{array}{c|}
\Ocal_X((2+i)D_2 +(2+j-i)D_5 +2D_6 +(k+1)D_7), \\
\Ocal_X((1+i)D_2+ (k+i)D_5 +(1+j-i)D_6+(1+j-i)D_7), \\
\Ocal_X(kD_7),\ \Ocal_X(D_2+kD_5+D_6+D_7),\\
\Ocal_X(3D_2 +D_5 +2D_6 +2D_7)
\end{array}
\begin{array}{c}
\phantom{a} 1 \leq i \leq j \leq 2 \\
0 \leq k \leq 1
\end{array}
\right\rbrace \]
is strong exceptional and contains the non-nef line bundle $\Ocal_X(D_7)$.
We obtain the chain complex of $\Pic(X \times X)-$graded $(S_{X \times X})-$modules
\begin{equation}\label{eq:J1ExactSequenceOfSheaves}
0\rightarrow (S_{X \times X})^{12} \stackrel{d_4}{\rightarrow} (S_{X \times X})^{38} \stackrel{d_3}{\rightarrow} (S_{X \times X})^{59} \stackrel{d_2}{\rightarrow} (S_{X \times X})^{50} \stackrel{d_1}{\rightarrow} (S_{X \times X})^{17}
\end{equation}
from this collection, which is exact up to saturation by $B_{X \times X}$ \cite{Prna1}. The table in Figure \ref{fig:4fold81arrows} lists the arrows in the quiver of sections $Q$ corresponding to $\Lcal$.

\begin{figure}
\begin{small}
\begin{tabular}{|c|c|c| c |c|c|c| c |c|c|c| c |c|c|c|}
\cline{1-3}
a &  $\mathbf{t}(a),\mathbf{h}(a)$  & $\div(a)$ & 
\multicolumn{1}{c}{}  \\ \cline{1-3} \cline{5-7} \cline{9-11} \cline{13-15}
1 & 0,1 & $x_7$ & • & 14 & \phantom{aa}4,5\phantom{aa} & \phantom{a}$x_4$\phantom{a} & • & 27 & \phantom{aa}6,9\phantom{aa} & \phantom{a}$x_2$\phantom{a} & • & 40 & \phantom{aa}10,13\phantom{aa} & \phantom{a}$x_4$\phantom{a} \\ 
\cline{1-3} \cline{5-7} \cline{9-11} \cline{13-15} 
2 & 0,2 & $x_0$ & • & 15 & 4,5 & $x_5$ & • & 28 & 6,10 & $x_3$ & • & 41 & 10,13 & $x_5$ \\ 
\cline{1-3} \cline{5-7} \cline{9-11} \cline{13-15}
3 & 1,2 & $x_1x_6$ & • & 16 & 4,6 & $x_6x_7$ & • & 29 & 7,10 & $x_6$ & • & 42 & 10,14 & $x_1$ \\ 
\cline{1-3} \cline{5-7} \cline{9-11} \cline{13-15}
4 & 1,2 & $x_2x_6$ & • & 17 & 4,7 & $x_3x_7$ & • & 30 & 7,11 & $x_4$ & • & 43 & 10,14 & $x_2$ \\ 
\cline{1-3} \cline{5-7} \cline{9-11} \cline{13-15}
5 & 1,3 & $x_3x_6x_7$ & • & 18 & 4,9 & $x_0$ & • & 31 & 7,11 & $x_5$ & • & 44 & 11,13 & $x_6$ \\ 
\cline{1-3} \cline{5-7} \cline{9-11} \cline{13-15} 
6 & 1,4 & $x_0x_3$ & • & 19 & 5,7 & $x_1$ & • & 32 & 7,16 & $x_0$ & • & 45 & 12,15 & $x_4$ \\ 
\cline{1-3} \cline{5-7} \cline{9-11} \cline{13-15} 
7 & 2,3 & $x_4$ & • & 20 & 5,7 & $x_2$ & • & 33 & 8,12 & $x_1$ & • & 46 & 12,15 & $x_5$ \\ 
\cline{1-3} \cline{5-7} \cline{9-11} \cline{13-15} 
8 & 2,3 & $x_5$ & • & 21 & 5,8 & $x_6x_7$ & • & 34 & 8,12 & $x_2$ & • & 47 & 12,16 & $x_1$ \\ 
\cline{1-3} \cline{5-7} \cline{9-11} \cline{13-15}
9 & 2,4 & $x_3x_7$ & • & 22 & 5,11 & $x_3x_7$ & • & 35 & 8,13 & $x_3$ & • & 48 & 12,16 & $x_2$ \\ 
\cline{1-3} \cline{5-7} \cline{9-11} \cline{13-15} 
10 & 3,4 & $x_1$ & • & 23 & 5,12 & $x_0$ & • & 36 & 9,12 & $x_4$ & • & 49 & 13,15 & $x_7$ \\ 
\cline{1-3} \cline{5-7} \cline{9-11} \cline{13-15}
11 & 3,4 & $x_2$ & • & 24 & 6,8 & $x_4$ & • & 37 & 9,12 & $x_5$ & • & 50 & 14,16 & $x_7$ \\ 
\cline{1-3} \cline{5-7} \cline{9-11} \cline{13-15}
12 & 3,5 & $x_3x_7$ & • & 25 & 6,8 & $x_5$ & • & 38 & 9,14 & $x_3$ & • & • & • & • \\ 
\cline{1-3} \cline{5-7} \cline{9-11} \cline{13-15}
13 & 3,6 & $x_0$ & • & 26 & 6,9 & $x_1$ & • & 39 & 10,12 & $x_7$ & • & • & • & • \\ 
\cline{1-3} \cline{5-7} \cline{9-11} \cline{13-15}
\end{tabular} 
\end{small}
\caption{The arrows in a quiver of sections for the smooth toric Fano fourfold $J_1$}
\label{fig:4fold81arrows}
\end{figure}

\begin{figure}[!ht]
\subfigure[The quiver of sections]{
\centering
  \psset{unit=1.15cm}
\begin{pspicture}(0,-0.5)(9,4)
\rput(-1,-0.4){\rnode{0}{\pscirclebox[framesep=1pt]{\tiny{0}}}}
\rput(0.5,1){\rnode{1}{\pscirclebox[framesep=1pt]{\tiny{1}}}}
\rput(1.5,-0.4){\rnode{2}{\pscirclebox[framesep=1pt]{\tiny{2}}}}
\rput(2.7,0){\rnode{3}{\pscirclebox[framesep=1pt]{\tiny{3}}}}
\rput(2.7,2){\rnode{4}{\pscirclebox[framesep=1pt]{\tiny{4}}}}
\rput(3.9,2.4){\rnode{5}{\pscirclebox[framesep=1pt]{\tiny{5}}}}
\rput(8.2,0){\rnode{6}{\pscirclebox[framesep=1pt]{\tiny{6}}}}
\rput(3.9,4.4){\rnode{7}{\pscirclebox[framesep=1pt]{\tiny{7}}}}
\rput(9.4,0.4){\rnode{8}{\pscirclebox[framesep=1pt]{\tiny{8}}}}
\rput(8.2,2){\rnode{9}{\pscirclebox[framesep=1pt]{\tiny{9}}}}
\rput(6.5,3.8){\rnode{10}{\pscirclebox[framesep=0.3pt]{\tiny{10}}}}
\rput(5.1,4.8){\rnode{11}{\pscirclebox[framesep=0.3pt]{\tiny{11}}}}
\rput(9.4,2.4){\rnode{12}{\pscirclebox[framesep=0.3pt]{\tiny{12}}}}
\rput(7.7,4.2){\rnode{13}{\pscirclebox[framesep=0.3pt]{\tiny{13}}}}
\rput(6.5,5.8){\rnode{14}{\pscirclebox[framesep=0.3pt]{\tiny{14}}}}
\rput(10.6,2.8){\rnode{15}{\pscirclebox[framesep=0.3pt]{\tiny{15}}}}
\rput(9.4,4.4){\rnode{16}{\pscirclebox[framesep=0.3pt]{\tiny{16}}}}  
\psset{nodesep=0pt}

\ncline{->}{0}{1}\lput*{:U}{\tiny{$a_1$}}
\ncline{->}{1}{3}\lput*{:U}(0.4){\tiny{$a_5$}}

\ncline{->}{4}{6}\lput*{:U}{\tiny{$a_{16}$}}
\ncline{->}{5}{8}\lput*{:U}{\tiny{$a_{21}$}}

\nccurve[angleA=50,angleB=230]{->}{2}{4}\lput*{:U}{\tiny{$a_9$}}
\nccurve[angleA=50,angleB=230]{->}{3}{5}\lput*{:U}{\tiny{$a_{12}$}}
\nccurve[angleA=50,angleB=230]{->}{4}{7}\lput*{:U}{\tiny{$a_{17}$}}
\nccurve[angleA=50,angleB=230]{->}{5}{11}\lput*{:U}{\tiny{$a_{22}$}}

\nccurve[angleA=10,angleB=220]{->}{1}{4}\lput*{:U}{\tiny{$a_6$}}
\nccurve[angleA=300,angleB=150]{->}{10}{12}\lput*{:U}(0.8){\tiny{$a_{39}$}}
\nccurve[angleA=300,angleB=150]{->}{13}{15}\lput*{:U}{\tiny{$a_{49}$}}
\nccurve[angleA=300,angleB=150]{->}{14}{16}\lput*{:U}{\tiny{$a_{50}$}}
\nccurve[angleA=130,angleB=290]{->}{6}{10}\lput*{:180}{\tiny{$a_{28}$}}
\nccurve[angleA=130,angleB=290]{->}{8}{13}\lput*{:180}(0.2){\tiny{$a_{35}$}}
\nccurve[angleA=130,angleB=290]{->}{9}{14}\lput*{:180}(0.15){\tiny{$a_{38}$}}
\nccurve[angleA=330,angleB=180]{->}{7}{10}\lput*{:U}{\tiny{$a_{29}$}}
\nccurve[angleA=330,angleB=180]{->}{11}{13}\lput*{:U}(0.15){\tiny{$a_{44}$}}

\ncline{->}{0}{2}\lput*{:U}{\tiny{$a_2$}}
\ncline[offset=-0.03]{->}{1}{2}\bput*[0.08]{:U,border=-1.5pt}(0.45){\tiny{$a_{3,4}$}}
\ncline[offset=0.03]{->}{1}{2}
\ncline[offset=-0.03]{->}{2}{3}\bput*[0.08]{:U,border=-1.5pt}(0.45){\tiny{$a_{7,8}$}}
\ncline[offset=0.03]{->}{2}{3}
\ncline[offset=-0.03]{->}{3}{4}\bput*[0.1]{:U,border=-1.5pt}{\tiny{$a_{10,11}$}}
\ncline[offset=0.03]{->}{3}{4}
\ncline{->}{3}{6}\lput*{:U}{\tiny{$a_{13}$}}
\ncline[offset=-0.03]{->}{4}{5}
\ncline[offset=0.03]{->}{4}{5}\aput*[0.08]{:U,border=-1.5pt}(0.4){\tiny{$a_{14,15}$}}
\ncline{->}{4}{9}\lput*{:U}{\tiny{$a_{18}$}}
\ncline[offset=-0.03]{->}{6}{9}\bput*[0.1]{:U,border=-1.5pt}{\tiny{$a_{26,27}$}}
\ncline[offset=0.03]{->}{6}{9}
\ncline[offset=-0.03]{->}{6}{8}\bput*[0.08]{:U,border=-1.5pt}(0.45){\tiny{$a_{24,25}$}}
\ncline[offset=0.03]{->}{6}{8}
\ncline[offset=-0.03]{->}{5}{7}\bput*[0.1]{:U,border=-1.5pt}{\tiny{$a_{19,20}$}}
\ncline[offset=0.03]{->}{5}{7}
\ncline{->}{5}{12}\lput*{:U}{\tiny{$a_{23}$}}
\ncline[offset=-0.03]{->}{7}{11}
\ncline[offset=0.03]{->}{7}{11}\aput*[0.08]{:U,border=-1.5pt}(0.4){\tiny{$a_{30,31}$}}
\ncline{->}{7}{16}\lput*{:U}(0.85){\tiny{$a_{32}$}}
\ncline[offset=-0.03]{->}{9}{12}\bput*[0.08]{:U,border=-1.5pt}(0.45){\tiny{$a_{36,37}$}}
\ncline[offset=0.03]{->}{9}{12}
\ncline[offset=-0.03]{->}{8}{12}\bput*[0.1]{:U,border=-1.5pt}{\tiny{$a_{33,34}$}}
\ncline[offset=0.03]{->}{8}{12}
\ncline[offset=-0.03]{->}{10}{13}\bput*[0.08]{:U,border=-1.5pt}(0.45){\tiny{$a_{40,41}$}}
\ncline[offset=0.03]{->}{10}{13}
\ncline[offset=-0.03]{->}{10}{14}
\ncline[offset=0.03]{->}{10}{14}\aput*[0.1]{:U,border=-1.5pt}(0.7){\tiny{$a_{42,43}$}}
\ncline[offset=-0.03]{->}{12}{16}\bput*[0.1]{:U,border=-1.5pt}{\tiny{$a_{47,48}$}}
\ncline[offset=0.03]{->}{12}{16}
\ncline[offset=-0.03]{->}{12}{15}\bput*[0.08]{:U,border=-1.5pt}(0.45){\tiny{$a_{45,46}$}}
\ncline[offset=0.03]{->}{12}{15}

\end{pspicture}}

\subfigure[A representation of the quiver corresponding to a torus-invariant point in $Y_\theta$]{\centering
  \psset{unit=1cm}
   \begin{pspicture}(0,0)(9,6)
\rput(-0.5,0){\rnode{0}{\pscirclebox[framesep=1pt]{\tiny{0}}}}
\rput(0.5,1){\rnode{1}{\pscirclebox[framesep=1pt]{\tiny{1}}}}
\rput(1.5,0){\rnode{2}{\pscirclebox[framesep=1pt]{\tiny{2}}}}
\rput(2.7,0){\rnode{3}{\pscirclebox[framesep=1pt]{\tiny{3}}}}
\rput(2.7,2){\rnode{4}{\pscirclebox[framesep=1pt]{\tiny{4}}}}
\rput(3.9,2.4){\rnode{5}{\pscirclebox[framesep=1pt]{\tiny{5}}}}
\rput(8.2,0){\rnode{6}{\pscirclebox[framesep=1pt]{\tiny{6}}}}
\rput(3.9,4.4){\rnode{7}{\pscirclebox[framesep=1pt]{\tiny{7}}}}
\rput(9.4,0.4){\rnode{8}{\pscirclebox[framesep=1pt]{\tiny{8}}}}
\rput(8.2,2){\rnode{9}{\pscirclebox[framesep=1pt]{\tiny{9}}}}
\rput(6.5,3.8){\rnode{10}{\pscirclebox[framesep=0.3pt]{\tiny{10}}}}
\rput(5.1,4.8){\rnode{11}{\pscirclebox[framesep=0.3pt]{\tiny{11}}}}
\rput(9.4,2.4){\rnode{12}{\pscirclebox[framesep=0.3pt]{\tiny{12}}}}
\rput(7.7,4.2){\rnode{13}{\pscirclebox[framesep=0.3pt]{\tiny{13}}}}
\rput(6.5,5.8){\rnode{14}{\pscirclebox[framesep=0.3pt]{\tiny{14}}}}
\rput(10.6,2.8){\rnode{15}{\pscirclebox[framesep=0.3pt]{\tiny{15}}}}
\rput(9.4,4.4){\rnode{16}{\pscirclebox[framesep=0.3pt]{\tiny{16}}}}  
\psset{nodesep=0pt}

\ncline{->}{0}{1}
\ncline{->}{1}{2}
\ncline{->}{2}{3}
\ncline{->}{3}{4}
\ncline{->}{4}{5}
\ncline{->}{4}{6}
\ncline{->}{5}{7}
\ncline{->}{5}{8}
\ncline{->}{6}{8}
\ncline{->}{6}{9}
\nccurve[angleA=330,angleB=180]{->}{7}{10}
\ncline{->}{7}{11}
\ncline{->}{8}{12}
\ncline{->}{9}{12}
\nccurve[angleA=300,angleB=150]{->}{10}{12}
\ncline{->}{10}{13}
\ncline{->}{10}{14}
\nccurve[angleA=330,angleB=180]{->}{11}{13}
\ncline{->}{12}{15}
\ncline{->}{12}{16}
\nccurve[angleA=300,angleB=150]{->}{13}{15}
\nccurve[angleA=300,angleB=150]{->}{14}{16}

\end{pspicture}}
\caption{A quiver of sections on the smooth toric Fano fourfold $J_1$}
  \label{fig:4fold81}
  \end{figure}

\noindent As $\vert Q_0 \vert = 17$ and $\vert \Sigma(1) \vert = 8$, we let $\{\mathbf{e}_i \mid i \in Q_0\} \cup \{\mathbf{e}_\rho \mid \rho \in \Sigma(1)\}$ be the standard basis of $\ZZ^{17+8}$ and define the lattice points $c_a := \mathbf{e}_{\mathbf{h}(a)} - \mathbf{e}_{\mathbf{t}(a)} + \mathbf{e}_{\div(a)}$ for each arrow $a \in Q_1$. The map $\pi$ is then given by the matrix $C \colon \ZZ^{50} \rightarrow \ZZ^{17+8}$ where the columns of $C$ are given by $c_a, a \in Q_1$, and the semigroup $\NN(Q)$ is given by the lattice points generated by positive linear combinations of the $c_a$. Our choice of bases for $\Pic(X)$ and $\Wt(Q)$ imply that the lattice maps $\deg$ and $\pic$ are given by the matrices:
\[
\deg \colon \left[\begin{smallmatrix}
1& 1& 1& 1&  0& 0& 0& 0 \\
0& 0& 0& 1&  1& 1& 0& 0 \\
1& 0& 0& 0&  0& 0& 1& 0 \\
1& 0& 0& -1& 0& 0& 0& 1 
\end{smallmatrix} \right], \phantom{a}
\pic \colon \left[ \begin{smallmatrix}
0& 0& 1& 1& 2& 2& 2& 3& 2& 3& 3& 3& 3& 3& 4& 3& 4 \\
0& 0& 0& 1& 1& 2& 1& 2& 2& 1& 2& 3& 2& 3& 2& 3& 2 \\
0& 0& 1& 1& 1& 1& 2& 1& 2& 2& 2& 1& 2& 2& 2& 2& 2 \\
0& 1& 1& 1& 1& 1& 2& 1& 2& 2& 1& 1& 2& 1& 1& 2& 2 
\end{smallmatrix} \right] \]

\noindent Fix $\theta$ to be the weight that assigns $-6$ to the vertex $0$ in the quiver, $1$ to the vertices $\{11,12,\ldots,16\}$ and $0$ to every other vertex. We note that $\pic(\theta)$ is the ample line bundle $L = \Ocal_X(20D_2+15D_5+11D_6+9D_7)$. For this choice of $\theta$, not only does $\pi_2 \left(\NN(Q) \cap (\pi_1)^{-1}(\theta)\right)$ surject onto $\NN^{\Sigma(1)} \cap \deg^{-1}(L)$, but $Y_\theta$ is isomorphic to $X$. As $\theta_i \geq 0$ for $i > 0$, $\theta$ is in the same closed GIT-chamber for $Y_\theta$ as $\vartheta$, and so they are in the same open chamber if $\theta$ is generic. To check that $\theta$ is generic, it is enough to check that for each torus-invariant point on $Y_\theta$, the corresponding representation is $\theta$-stable. The list below gives the $17$ maximal cones in the fan for $Y_\theta$ -- recall that each maximal cone corresponds to a torus-invariant point and the point is in the intersection of the divisors labelled by the rays of the cone. 
\noindent \begin{center}
\begin{tabular}{ccccc}
$\{\rho_0 ,\rho_1 ,\rho_3 ,\rho_4 \}$ & 
$\{\rho_0 ,\rho_1 ,\rho_3 ,\rho_5 \}$ &
$\{\rho_0 ,\rho_1 ,\rho_4 ,\rho_5 \}$ &
$\{\rho_0 ,\rho_2 ,\rho_3 ,\rho_4 \}$ &
$\{\rho_0 ,\rho_2 ,\rho_3 ,\rho_5 \}$ \\
$\{\rho_0 ,\rho_2 ,\rho_4 ,\rho_5 \}$ &
$\{\rho_1 ,\rho_2 ,\rho_4 ,\rho_5 \}$ &
$\{\rho_1 ,\rho_2 ,\rho_4 ,\rho_6 \}$ &
$\{\rho_1 ,\rho_2 ,\rho_5 ,\rho_6 \}$ &
$\{\rho_1 ,\rho_3 ,\rho_4 ,\rho_7 \}$ \\
$\{\rho_1 ,\rho_3 ,\rho_5, \rho_7 \}$ &
$\{\rho_1 ,\rho_4 ,\rho_6 ,\rho_7 \}$ &
$\{\rho_1 ,\rho_5 ,\rho_6 ,\rho_7 \}$ &
$\{\rho_2 ,\rho_3 ,\rho_4 ,\rho_7 \}$ &
$\{\rho_2 ,\rho_3 ,\rho_5 ,\rho_7 \}$ \\
\multicolumn{5}{c}{
$\{\rho_2 ,\rho_4 ,\rho_6 ,\rho_7 \}$ \phantom{a}
$\{\rho_2 ,\rho_5 ,\rho_6 ,\rho_7 \}$}\\
\end{tabular}
\end{center}
For the representation $(V,\phi)$ corresponding to the torus-invariant point with rays $\{\rho_{i_1}, \rho_{i_2},$ $\rho_{i_3}, \rho_{i_4}\}$, the map $\phi_a$ is $0$ if for any $x_i \in \{x_{i_1},x_{i_2},x_{i_3},x_{i_4}\}$, $x_i$ divides $\div(a)$, whilst $\phi_a = 1$ otherwise. For example, consider the maximal cone $\{\rho_0 ,\rho_1 ,\rho_3 ,\rho_4 \}$. The corresponding representation $V = (V,\phi)$ has $\phi_a = 0$ for 
\[ a \in \left\lbrace \begin{array}{c}
a_2,\ a_3,\ a_5,\ a_6,\ a_7,\ a_9,\ a_{10},\ a_{12},\ a_{13},\ a_{14},\ a_{17},\ a_{18},\ a_{19},\ a_{22},\ a_{23},\ a_{24},\ a_{26},\\
a_{28},\ a_{30},\ a_{32},\ a_{33},\ a_{35},\ a_{36},\ a_{38},\ a_{40},\ a_{42},\ a_{45},\ a_{47}
\end{array} \right\rbrace \]
and is displayed in Figure \ref{fig:4fold81}. Specifying a subrepresentation $(V',\phi')$ of $V$ is equivalent to setting $\phi'_a = \phi_a$ for all $a \in Q_1$ and choosing a subset $I \subset Q_0$ so that $V'_i = \CC$ for $i \in I$, and $V'_i = 0$ otherwise. In our example, for any subrepresentation $V'$ with $V'_0 = \CC$, we have $V' = V$ as there is a non-zero map from $V'_0$ to every other $V'_i$. It is also clear from Figure \ref{fig:4fold81} that for any $i \in Q_0$, there is a non zero map from $V_i$ to $V_{j}$ for some $j \in \{11,12, \ldots,16\}$. As a result, the corresponding nonzero proper subrepresentation $V'$ of $V$ must have $V'_{j} = \CC$ and so $\theta(V') > 0$ by the choice of $\theta$. By considering the subrepresentations of the representation corresponding to each of the $17$ torus-invariant points on $Y_\theta$, we see that $\theta$ is generic -- the calculations for this example can be found in the file \cite{Prna3}. Therefore, \eqref{eq:J1ExactSequenceOfSheaves} is a resolution of $\Ocal_\Delta$ by Propositions \ref{prop:NonNefEmbedding} and \ref{prop:EmbeddingGivesCokernel}, so the collection $\Lcal$ on $J_1$ is full by Proposition \ref{prop:GenerationResolutionDiagonal}. 
\end{example}

\section{Full Strong Exceptional Collections on Toric Varieties}

\noindent In this section we present the main theorems of this paper.

For a divisorial contraction $(f,\phi) \colon (X, \Sigma_{X}) \rightarrow (Y, \Sigma_{Y})$ the Frobenius morphism can be used to give examples of when the pushforward of a line bundle $L$ via $f$ and the image of $L$ under the map $\gamma$ from (\ref{gamma}) are equal. Note that for a toric variety $X$, the canonical bundle is $\omega_X = -\sum_{\rho \in \Sigma_X(1)} D_\rho$. 

\begin{lemma}\label{fgam}
Fix an integer $m > 0$ and let $(f,\phi) \colon (X, \Sigma_{X}) \rightarrow (Y, \Sigma_{Y})$ be a torus-equivariant extremal birational contraction between smooth $n$-dimensional projective toric varieties. Let $\sigma \subset \Sigma_{X}$ be a maximal cone such that $\phi(\sigma)$ is a cone in $\Sigma_{Y}$, and $\textbf{w} = \textbf{0}$ or $(-1, \ldots, -1)^t \in \ZZ^{\Sigma_X(1)}$. Then for any $\textbf{v} \in P^n_m$,
\begin{equation}
f_* \Ocal_{X} (D^X_{\textbf{v},\textbf{w},\sigma} ) = \Ocal_{Y} (D^Y_{\textbf{v},\textbf{w},\phi(\sigma)})
\end{equation}
and
\begin{equation}
\mkD (\Ocal_Y )_m = \lbrace f_* L_X \ \vert \ L_X \in \mkD ( \Ocal_X )_m \rbrace,
\end{equation}
\begin{equation}
\mkD (\omega_Y )_m = \lbrace f_* L_X \ \vert \ L_X \in \mkD (\omega_X )_m \rbrace.
\end{equation} 
In particular, the maps $f_*$ and $\gamma$ coincide for $\Ocal_{X} (D^X_{\textbf{v},\textbf{w},\sigma} )$.  
\end{lemma}
\begin{proof}
The result \cite[Lemma 6.1]{Ueha} gives the case $\textbf{w} = \textbf{0}$. Noting that $f_* (\omega_{X}) \cong \omega_{Y}$ \cite[Theorem 9.3.12]{CoLiSc}, the proof can also be applied to $\textbf{w} = (-1, \ldots , -1 )^t$.  The algorithm to compute $F_m (L )$ demonstrates the equality between $f_*$ and $\gamma$ for the line bundles considered.   
\end{proof}

\begin{proposition}\label{prop:picMapFull}
With the same assumptions as in Lemma \ref{fgam}, choose a collection of line bundles $\Lcal \subset \mkD_m \cup \mkD(\omega_X)_m$. If $\Lcal$ generates $\Dx$ then the non-isomorphic line bundles in the image of $\gamma(\Lcal)$ generate $\Dcal^b(Y)$.
\end{proposition}

\begin{proof}
Note that $\RR f_* \Ocal_{X} = \Ocal_{Y}$ and $\RR f_* \omega_{X} = \omega_{Y}$ \cite[Theorem 9.3.12]{CoLiSc}, so by the equality $F_m^{Y} \circ f = f \circ F_m^{X}$ we have $\RR f_* F_m^{X} \Ocal_{X} = F_m^{Y} \Ocal_{Y}$ and $\RR f_* F_m^{X} \omega_{X} = F_m^{Y} \omega_{Y}$. The result then follows by Lemmas \ref{pushgen} and \ref{fgam}.
\end{proof}

\begin{remark}
A collection of line bundles $\Lcal$ on $X$ is full strong exceptional if and only if the dual collection $\Lcal^{-1} := \{ L^{-1} \mid L \in \Lcal\}$ is full strong exceptional. In the following theorem when we choose $\Lcal \subset \mkD_m \cup \mkD(\omega_X)_m$ and use \emph{\textbf{method 2}} to show that $\Lcal$ is full, we actually compute the $S_{X \times X}$-module chain complex using $\Lcal^{-1}$, as $\Lcal^{-1}$ will be an effective collection whilst $\Lcal$ will not be effective.
\end{remark}

\begin{theorem}\label{thm:FSECfourfolds}
Let $X$ be a smooth toric Fano fourfold. There exists a full strong exceptional collection comprising of line bundles for $X$. A database of these collections can be found in \cite{Prna1}. 
\end{theorem}

\begin{proof}
Firstly, if $X$ is the product of smooth toric Fano varieties of a lower dimension, then it has a full strong exceptional collection of line bundles by Lemma \ref{lem:FSECproducts} and \cite{King,Ueha}. These varieties, along with $\PP^4$, account for $29$ of the $124$ fourfolds. Now consider the $26$ birationally-maximal smooth toric Fano fourfolds that are not products. Given a collection of line bundles on each of these varieties, we can check if it is strong exceptional by using the method outlined in Section \ref{subsect:EMS-StrExc}, which has been implemented in the computer package \emph{QuiversToricVarieties} in \emph{Macaulay2} \cite{Prna1,M2}. Once we have a strong exceptional collection, we can use \emph{\textbf{method 1}} or \emph{\textbf{method 2}} to show that it is full; preference has been given to using \emph{\textbf{method 2}}, but when the chain complex of $S_{X \times X}$-modules cannot be computed we use \emph{\textbf{method 1}}.

The collections of line bundles on the following maximal varieties are shown to be full using \emph{\textbf{method 2}}:
\[ E_1,\ I_5,\ I_8,\ I_{15},\ M_5,\ J_1,\ J_2,\ K_1,\ K_2,\ R_1,\ R_3,\ U_2,\ U_7,\ U_8,\ V^4,\ \tilde{V}^4  \]
For all of these varieties except $J_1$ (see Example \ref{ex:J1FanoFourfoldGeneration}) and $V^4$ (see Example \ref{ex:4fold117}), the collections chosen are collections of nef line bundles. The collections on the rest of the maximal fourfolds are shown to be full using \emph{\textbf{method 1}}. The collections for all of the $26$ maximal fourfolds are chosen from $\mkD_m \cup \mkD(\omega_X)_m$ for some $m > 0$, except for $R_1, \ R_2, \ R_3, \ \tilde{V}^4,$ and $V^4$. Note however that for $R_3$, the given collection determines a new full strong exceptional collection chosen from $\mkD_m \cup \mkD(\omega_X)_m$, as shown in Example \ref{ex:FSECthreadsR3}.

For a full strong exceptional collection $\Lcal_X$ chosen from $\mkD_m \cup \mkD(\omega_X)_m$ on a maximal fourfold and a chain of divisorial contractions $X \rightarrow X_1 \rightarrow \cdots \rightarrow X_t$, Proposition \ref{prop:picMapFull} implies that the distinct line bundles $\Lcal_{X_t}$ in the image of $\gamma(\Lcal)$ generates $\Dcal^b(X_t)$, where $\gamma \colon \Pic(X) \rightarrow \Pic(X_t)$. Therefore, it is enough to show that $\Lcal_{X_t}$ is strong exceptional to obtain a full strong exceptional collection on $X_t$. To check that $\Lcal_{X_t}$ is strong exceptional, we can check that $\Lcal$ avoids the pullback of the non-vanishing cohomology cones for $X_t$ via $\gamma$ as detailed in Section \ref{subsect:CohConesBlowUps} -- again, this calculation has been implemented into \emph{QuiversToricVarieties}. By performing these calculations, we obtain full strong exceptional collections of line bundles for $64$ of the Fano fourfolds.

The remaining fourfolds $D_6, \ H_{10}, \ M_1, \ M_2$ and $M_3$ are non-maximal and do not obtain a full strong exceptional collection from a maximal fourfold. There is a divisorial contraction $H_{10} \rightarrow D_6$ and we can choose a strong exceptional collection $\Lcal_{H_{10}} \subset \mkD_m \cup \mkD(\omega_{H_{10}})_m$ such that we get a strong exceptional collection $\Lcal_{D_6}$ on $D_6$ via the map $\gamma$. We use \emph{\textbf{method 2}} to show that $\Lcal_{H_{10}}$ is full, and hence we also obtain a full strong exceptional collection on $D_6$. For varieties $M_1$ and $M_3$ we use \emph{\textbf{method 2}} to show generation; the collection on $M_1$ contains non-nef line bundles (see Example \ref{ex:4fold76}) whilst all of the line bundles chosen on $M_3$ are nef. For variety $M_2$, our chosen strong exceptional collection is shown to be full using \emph{\textbf{method 1}}.      

The calculations in \emph{Macaulay2} and \emph{Sage} \cite{sage} that are required for this proof can be found in the file \cite{Prna3}.
\end{proof}

\begin{remark}
The construction above provides a new proof that there exist full strong exceptional collections on $n$-dimensional toric Fano smooth varieties for $n \leq 3$. In particular, resolutions of $\Ocal_\Delta$ using $S_{X \times X}-$modules have been constructed for each birationally maximal Fano threefold $X$. 
\end{remark}

\begin{example}\label{ex:FSECthreadsR3}
The $106^{th}$ smooth toric Fano fourfold $X := R_3$ has ray generators 
\[
u_0 = \begin{tiny}
\begin{bmatrix}
1 \\ 0 \\ 0 \\0
\end{bmatrix} \end{tiny}, 
u_1 = \begin{tiny} \begin{bmatrix}
0 \\ 1 \\ 0 \\ 0
\end{bmatrix} \end{tiny},
u_2 = \begin{tiny} \begin{bmatrix}
0 \\ -1 \\ 1 \\ -1
\end{bmatrix} \end{tiny},
u_3 = \begin{tiny} \begin{bmatrix}
0 \\ 0 \\ 1 \\ 0
\end{bmatrix} \end{tiny},
u_4 = \begin{tiny} \begin{bmatrix}
0 \\ 0 \\ 0 \\ 1
\end{bmatrix} \end{tiny},
u_5 = \begin{tiny} \begin{bmatrix}
0 \\ 0 \\ -1 \\ 0
\end{bmatrix} \end{tiny},
u_6 = \begin{tiny} \begin{bmatrix}
1 \\ 0 \\ 0 \\ -1
\end{bmatrix} \end{tiny},
u_7 = \begin{tiny} \begin{bmatrix}
-1 \\ 0 \\ 0 \\ 1
\end{bmatrix} \end{tiny},\]
\[u_8 = \begin{tiny} \begin{bmatrix}
-1 \\ 0 \\ 0 \\ 0
\end{bmatrix} \end{tiny}\] for its fan $\Sigma_X$. We take the corresponding divisors $D_2, D_5, D_6, D_7, D_8$ to be a basis for $\Pic(X)$. The collection of nef line bundles on $X$
\[ \begin{small} \Lcal = \left\lbrace
\begin{array}{c|}
 \Ocal_X ( iD_2 +jD_5 +D_6 + D_7 +D_8),\ \Ocal_X ((j-1)D_5 + (i-1)D_7 +(i-1)D_8), \\ \Ocal_X ((i-1)D_2 +2D_5 +j D_7 +j D_8),\ 
\Ocal_X (D_2 + D_5 +i D_7 +jD_8), \\ \Ocal_X (D_2 +2 D_5 +(i-1) D_6 + D_7+j D_8), \ \Ocal_X (2D_2 +i D_5 +j D_7 +2D_8), \\
\Ocal_X (2D_2 +(j-i+1) D_5 +(1-j+i) D_6 + D_7 +jD_8) \\
\end{array} \ 1 \leq i \leq j \leq 2 \right\rbrace \end{small}\] 
is strong exceptional and is shown to be full using \emph{\textbf{method 2}}, but $\Lcal^{-1} \nsubseteq \mkD_m \cup \mkD(\omega_X)_m$. However, following Bridgeland and Stern \cite{BrSt}, we can use $\Lcal$ to define a \emph{helix} $\mathbb{H}_\Lcal$ for $X$:

\begin{definition}
A sequence of coherent sheaves $\mathbb{H} = (E_i)_{i \in \ZZ}$ on $X$ is a \emph{helix} if
\begin{itemize}
\item for each $i \in \ZZ$ the \emph{thread} $(E_{i+1}, \ldots , E_{i+k} )$ is a full exceptional collection,
\item for each $i \in \ZZ$, we have $E_{i-k} = E_i \otimes \omega_X$. 
\end{itemize}
\end{definition}
\noindent If $\{E_0, \ldots , E_{k-1}\}$ is a full exceptional collection then by \cite[Remark 3.2]{BrSt} each thread of $\mathbb{H}$ is a full exceptional collection. Therefore, as $\Lcal$ is full then the sequence $\mathbb{H}_\Lcal$ constructed from $\Lcal$ is indeed a helix. We can choose a thread in $\mathbb{H}_\Lcal$ and twist it by the line bundle $\Ocal_X(-D_5-D_7-D_8)$ to obtain the following full strong exceptional collection:
\[ \begin{small} \Lcal' = \left\lbrace
\begin{array}{c|}
\Ocal_X ( jD_5 + iD_7 + iD_8),\ \Ocal_X (D_2 + iD_7 +jD_8), \\ \Ocal_X (D_2 +jD_5 +D_6+i D_8),\ 
\Ocal_X (D_2 + D_5 +i D_7 +jD_8), \\ \Ocal_X (2D_2 +i D_5 + jD_7+ D_8), \ \Ocal_X (2D_2 +i D_5 + D_6 +jD_8), \\
\Ocal_X (2D_2 +(i+1) D_5 +j D_6 +(i-j+1)D_8) \\
\end{array} \ 0 \leq i \leq j \leq 1 \right\rbrace \end{small}\]
\noindent Now $(\Lcal')^{-1}$ is a non-nef collection that is contained in $\mkD_m \cup \mkD(\omega_X)_m$ for some $m > 0$, which we can use to obtain a full strong exceptional collection on the divisorial contraction $M_4$ via the method outlined in the proof of Theorem \ref{thm:FSECfourfolds}.   
\end{example}

The full strong exceptional collections on each smooth toric Fano variety $X$ determine tilting bundles on the total space of $\omega_X$.

\begin{theorem}
Let $Y = \text{tot}(w_X)$ be the total space of the canonical bundle on an $n$-dimensional smooth toric Fano variety $X$, for $n \leq 4$. Then $Y$ has a tilting bundle that decomposes as a direct sum of line bundles.
\end{theorem}

\begin{proof}
Let $\pi : Y \rightarrow X$ be the bundle map and $\{ L_0, \ldots , L_r \}$ be a full strong exceptional collection of line bundles on $X$. The pullback $\pi^* (E)$ of the bundle $E := \bigoplus_{i = 0}^r L_i$ generates $\Dcal^b(Y)$ (see the proof for \cite[Theorem 3.6]{BrSt}) so it enough to show that $\Hom^i(\pi^* (E) , \pi^* (E)) = 0$ for $i > 0$. By adjunction and the projection formula this is equivalent to showing that 
\begin{equation}
\bigoplus_{i \neq j} \Hom^k (L_i , L_j \otimes \pi_* (\Ocal_Y)) \cong \bigoplus_{i \neq j} H^k (X ,L_i^{-1} \otimes L_j \otimes \pi_* (\Ocal_Y)) = 0 \text{ for } k \neq 0.
\end{equation}
Now $$\pi_* (\Ocal_Y) = \bigoplus_{p \geq 0} \omega_X^{-p}$$ and as $\omega_X^{-1}$ is ample, there is some positive integer $T$ such that for all $t \geq T$, $L_i^{-1} \otimes L_j \otimes \omega_X^{-t}$ is nef for all $0 \leq i, j \leq r$, in which case $H^k(X,L_i^{-1} \otimes L_j \otimes \omega_X^{-t}) = 0$ for $k > 0$ by Demazure vanishing. Hence $\pi^*(E)$ is a tilting bundle if
\begin{equation}
H^k (X ,L_i^{-1} \otimes L_j \otimes \omega_X^{-t}) = 0 \text{ for } k \neq 0, \ 0 \leq i \neq j \leq r,\ 0 \leq t < T.
\end{equation}
The non-vanishing cohomology cones in $\Pic(X)$ can then be used to show that the line bundle $L_i^{-1} \otimes L_j \otimes \omega_X^{-t}$ has vanishing higher cohomology for $0 \leq i \neq j \leq r,\ 0 \leq t < T$, as implemented in \emph{QuiversToricVarieties}.

Now let $X =: X_0 \rightarrow X_1 \rightarrow \cdots X_d$ be a chain of divisorial contractions between smooth toric Fano varieties and assume that the collection $\Lcal$ determines a full strong exceptional collection on each $X_k, \ 0 \leq k \leq d$ via the divisorial contractions, as detailed in Theorem \ref{thm:FSECfourfolds}. For each variety $X_k$ with collection $\Lcal_{X_k}$, we have an integer $T_{k} \geq 0$ such that $L_i^{-1} \otimes L_j \otimes \omega_{X_k}^{-T_k}$ is nef for all $L_i,L_j \in \Lcal_{X_k}$. Define $T = \max(T_0, \ldots , T_d)$. Then we can check simultaneously that each $\Lcal_{X_k}$ determines a tilting bundle on $\tot(\omega_{X_k})$ by considering whether the line bundles $L_i^{-1} \otimes L_j \otimes \omega_X^{-t}$ avoid the preimage in $\Pic(X)$ of all non-vanishing cohomology cones for $X_1, \ldots , X_d$ via the Picard lattice maps, for all $0 \leq i \neq j \leq r, \ 0 \leq t < T$. Again, this calculation can be performed in \emph{QuiversToricVarieties}.   
\end{proof}

\begin{remark}
Using the full strong exceptional collections of line bundles $\Lcal_X$ on a smooth toric Fano variety $X$ given by King \cite{King},  Uehara \cite{Ueha} and Theorem \ref{thm:FSECfourfolds}, we find that 
\begin{enumerate}
\item the minimal $T$ such that $L_i^{-1} \otimes L_j \otimes \omega_X^{-T}$ in nef for all smooth toric Fano surfaces and all $L_i,L_j \in \Lcal_X$ is $T=1$;
\item the minimal $T$ such that $L_i^{-1} \otimes L_j \otimes \omega_X^{-T}$ in nef for all smooth toric Fano threefolds and all $L_i,L_j \in \Lcal_X$ is $T=2$;
\item the minimal $T$ such that $L_i^{-1} \otimes L_j \otimes \omega_X^{-T}$ in nef for all smooth toric Fano fourfolds and all $L_i,L_j \in \Lcal_X$ is $T=3$.
\end{enumerate}
\end{remark}

\section{Appendix A}
\noindent The tables below contain details on the construction of the full strong exceptional collections of line bundles $\Lcal$ on smooth $n$-dimensional toric Fano varieties, for $2 \leq n \leq 4$. If $X = \PP^n$ then the full strong exceptional collection is provided by Be\u\i linson \cite{Beil}, whilst if $X$ is a product of smooth toric Fano varieties then a full strong exceptional collection is given for $X$ by Lemma \ref{lem:FSECproducts}. 

A \emph{maximal} variety is a variety that is birationally maximal, as explained in Section \ref{subsect:ToricGeometry}. The sets $\mkD_m$ and $\mkD(\omega_X)_m$ are the Frobenius pushforwards of $\Ocal_X$ and $\omega_X$ respectively as defined in Section \ref{frob}, for some integer $m > 0$. \textbf{\emph{Method 1}} and \textbf{\emph{method 2}} are described in Section \ref{subsect:Method 1} and Section \ref{subsect:Method 2} respectively, and are used where stated to show that the collection $\Lcal$ is full. The description ``collection from ($j$)" for an $n$-dimensional variety $X$ means that there is a chain of torus-invariant divisorial contractions $X_0 \rightarrow X_1 \rightarrow \cdots \rightarrow X_0 := X$ between smooth $n$-dimensional toric Fano varieties, where $X_0$ is the $j^{th}$ $n$-dimensional variety. The full strong exceptional collection on $X$ is then given by the non-isomorphic line bundles in the image of the full strong exceptional collection for $X_0$ under the induced Picard lattice map $\gamma_{(0 \rightarrow t)} \colon \Pic(X_0) \rightarrow \Pic(X_t)$ (see Section \ref{subsect:CohConesBlowUps} and Theorem \ref{thm:FSECfourfolds} for details).        
\subsection{Toric del Pezzo Surfaces}

$$ $$
\begin{longtable}{|c||l||c|}

\hline
\endhead
\hline
\endfoot

\hspace{1cm} &  Variety  & \hspace{1.0cm} Details of the Full Strong Exceptional Collection \hspace{1.0cm} \\ \hline
\hline
(0) &  $\PP^2$ & Be\u\i linson's collection \\ \hline
(1) & $\PP^{1} \times \PP^{1}$ & product of smooth toric Fano varieties \\ \hline
(2) & $S_1$, $Bl_1 (\PP^2)$ & collection from (4) \\ \hline
(3) & $S_2$, $Bl_2 (\PP^2)$ & collection from (4)\\ \hline
(4) & $S_3$, $Bl_3 (\PP^2)$ & Maximal variety. $\Lcal = \mkD_m$. Generation: \textbf{\emph{Method 2}} \\ \hline
\end{longtable}
$$ $$
\subsection{Smooth Toric Fano Threefolds}
$$ $$
\begin{longtable}{|c||l||c|}

\hline
\endhead
\hline
\endfoot

\hspace{1cm} & Variety  & \hspace{1.0cm} Details of the Full Strong Exceptional Collection \hspace{1.0cm} \\ \hline
\hline 
(0) & $\PP^3$ & Be\u\i linson's collection \\ 
\hline 
(1) & $\Bcal_1$ & collection from (10) \\ 
\hline 
(2) & $\Bcal_2$ & collection from (17) \\ 
\hline 
(3) & $\Bcal_3$ & collection from (17) \\ 
\hline 
(4) & $\Bcal_4$, $\PP^2 \times \PP^1$ & product of smooth toric Fano varieties \\ 
\hline 
(5) & $\Ccal_1$ & collection from (17) \\ 
\hline 
(6) & $\Ccal_2$ & collection from (17) \\ 
\hline 
(7) & $\Ccal_3$, $\PP^1 \times \PP^1 \times \PP^1$ & product of smooth toric Fano varieties \\ 
\hline 
(8) & $\Ccal_4$, $S_1 \times \PP^1$ & product of smooth toric Fano varieties \\ 
\hline 
(9) & $\Ccal_5$ & collection from (17) \\ 
\hline 
(10) & $\Dcal_1$ & Maximal variety. $\Lcal \subset \mkD_m$. Generation: \textbf{\emph{Method 2}} \\ 
\hline 
(11) & $\Dcal_2$ & collection from (17) \\ 
\hline 
(12) & $\Ecal_1$ & collection from (17) \\ 
\hline 
(13) & $\Ecal_2$ & collection from (17) \\ 
\hline 
(14) & $\Ecal_3$, $S_2 \times \PP^1$ & product of smooth toric Fano varieties \\ 
\hline
(15) & $\Ecal_4$ & collection from (17) \\ 
\hline 
(16) & $\Fcal_1$, $S_3 \times \PP^1$ & Maximal, product of smooth toric Fano varieties \\ 
\hline 
(17) & $\Fcal_2$ & Maximal variety. $\Lcal = \mkD_m$. Generation: \textbf{\emph{Method 2}} \\ 
\hline 
\end{longtable}

$$ $$
\subsection{Smooth Toric Fano Fourfolds}

$$ $$

\begin{longtable}{|c||l||c|}

\hline
\endhead
\hline
\endfoot

\hspace{1cm} & Variety & \hspace{0.5cm} Details of the Full Strong Exceptional Collection \hspace{0.5cm} \\ \hline
\hline 
(0) & $\PP^4$ & Be\u\i linson's collection \\ \hline
(1) & $B_1$ & collection from (10) \\ \hline
(2) & $B_2$ & collection from (65) \\ \hline
(3) & $B_3$ & collection from (114) \\ \hline
(4) & $B_4$, $\PP^1 \times \PP^3$ & product of smooth toric Fano varieties \\ \hline
(5) & $B_5$ & collection from (114) \\ \hline
(6) & $C_1$ & collection from (100) \\ \hline
(7) & $C_2$ & collection from (114) \\ \hline
(8) & $C_3$ & collection from (80) \\ \hline
(9) & $C_4$, $\PP^2 \times \PP^2$ & product of smooth toric Fano varieties \\ \hline
(10) & $E_1$ & Maximal variety. $\Lcal \subset \mkD_m$. Generation: \textbf{\emph{Method 2}} \\ \hline
(11) & $E_2$ & collection from (65) \\ \hline
(12) & $E_3$ & collection from (109) \\ \hline
(13) & $D_1$ & collection from (100) \\ \hline
(14) & $D_2$ & collection from (100) \\ \hline
(15) & $D_3$ & collection from (101) \\ \hline
(16) & $D_4$ & collection from (65) \\ \hline
(17) & $D_5$, $\PP^1 \times \Bcal_1$ & product of smooth toric Fano varieties \\ \hline
(18) & $D_6$ & collection from (109) \\ \hline
(19) & $D_7$ & collection from (115) \\ \hline
(20) & $D_8$ & collection from (109) \\ \hline
(21) & $D_9$ & collection from (107) \\ \hline
(22) & $D_{10}$ & collection from (114) \\ \hline
(23) & $D_{11}$ & collection from (114) \\ \hline
(24) & $D_{12}$, $\PP^1 \times \Bcal_2$ & product of smooth toric Fano varieties \\ \hline
(25) & $D_{13}$, $\PP^1 \times \PP^1 \times \PP^2$ & product of smooth toric Fano varieties \\ \hline
(26) & $D_{14}$, $\PP^1 \times \Bcal_3$ & product of smooth toric Fano varieties \\ \hline
(27) & $D_{15}$, $S_1 \times \PP^2$ & product of smooth toric Fano varieties \\ \hline
(28) & $D_{16}$ & collection from (47) \\ \hline
(29) & $D_{17}$ & collection from (114) \\ \hline
(30) & $D_{18}$ & collection from (100) \\ \hline
(31) & $D_{19}$ & collection from (114) \\ \hline
(32) & $G_1$ & collection from (81) \\ \hline
(33) & $G_2$ & collection from (75) \\ \hline
(34) & $G_3$ & collection from (82) \\ \hline
(35) & $G_4$ & collection from (80) \\ \hline
(36) & $G_5$ & collection from (82) \\ \hline
(37) & $G_6$ & collection from (114) \\ \hline
(38) & $H_1$ & collection from (100) \\ \hline
(39) & $H_2$ & collection from (101) \\ \hline
(40) & $H_3$ & collection from (100) \\ \hline
(41) & $H_4$ & collection from (109) \\ \hline
(42) & $H_5$ & collection from (109) \\ \hline
(43) & $H_6$ & collection from (101) \\ \hline
(44) & $H_7$ & collection from (100) \\ \hline
(45) & $H_8$, $S_2 \times \PP^2$ & product of smooth toric Fano varieties \\ \hline
(46) & $H_9$ & collection from (109) \\ \hline
(47) & $H_{10}$ & Non-maximal. $\Lcal \subset \mkD_m \cup \mkD (\omega_X)_m$. Generation: \textbf{\emph{Method 2}} \\ \hline
(48) & $L_1$ & collection from (108) \\ \hline
(49) & $L_2$ & collection from (108) \\ \hline
(50) & $L_3$ & collection from (109) \\ \hline
(51) & $L_4$ & collection from (109) \\ \hline
(52) & $L_5$, $\PP^1 \times \Ccal_1$ & product of smooth toric Fano varieties \\ \hline
(53) & $L_6$, $\PP^1 \times \Ccal_2$ & product of smooth toric Fano varieties \\ \hline
(54) & $L_7$, $S_1 \times S_1$ & product of smooth toric Fano varieties \\ \hline
(55) & $L_8$, $\PP^1 \times \PP^1 \times \PP^1 \times \PP^1$ & product of smooth toric Fano varieties \\ \hline
(56) & $L_9$, $S_1 \times \PP^1 \times \PP^1$ & product of smooth toric Fano varieties \\ \hline
(57) & $L_{10}$ & collection from (114) \\ \hline
(58) & $L_{11}$, $\PP^1 \times \Ccal_5$ & product of smooth toric Fano varieties \\ \hline
(59) & $L_{12}$ & collection from (114) \\ \hline
(60) & $L_{13}$ & collection from (115) \\ \hline
(61) & $I_1$ & Maximal variety. $\Lcal \subset \mkD_m$. Generation: \textbf{\emph{Method 1}} \\ \hline
(62) & $I_2$ & Maximal variety. $\Lcal \subset \mkD_m$. Generation: \textbf{\emph{Method 1}} \\ \hline
(63) & $I_3$ & Maximal variety. $\Lcal \subset \mkD_m$. Generation: \textbf{\emph{Method 1}} \\ \hline
(64) & $I_4$ & Maximal variety. $\Lcal \subset \mkD_m$. Generation: \textbf{\emph{Method 1}} \\ \hline
(65) & $I_5$ & Maximal variety. $\Lcal \subset \mkD_m$. Generation: \textbf{\emph{Method 2}} \\ \hline
(66) & $I_6$ & collection from (114) \\ \hline
(67) & $I_7$, $\PP^1 \times \Dcal_1$ & Maximal, product of smooth toric Fano varieties \\ \hline
(68) & $I_8$ & Maximal variety. $\Lcal \subset \mkD_m \cup \mkD (\omega_X)_m$. Generation: \textbf{\emph{Method 2}} \\ \hline
(69) & $I_9$ & collection from (115) \\ \hline
(70) & $I_{10}$ & collection from (110) \\ \hline
(71) & $I_{11}$ & collection from (107) \\ \hline
(72) & $I_{12}$ & collection from (114) \\ \hline
(73) & $I_{13}$, $\PP^1 \times \Dcal_2$ & product of smooth toric Fano varieties \\ \hline
(74) & $I_{14}$ & collection from (109) \\ \hline
(75) & $I_{15}$ & Maximal variety. $\Lcal \subset \mkD_m$. Generation: \textbf{\emph{Method 2}} \\ \hline
(76) & $M_1$ & Non-maximal. $\Lcal \nsubseteq \mkD_m \cup \mkD (\omega_X)_m$. Generation: \textbf{\emph{Method 2}} \\ \hline
(77) & $M_2$ & Non-maximal. $\Lcal \nsubseteq \mkD_m \cup \mkD (\omega_X)_m$. Generation: \textbf{\emph{Method 1}} \\ \hline
(78) & $M_3$ & Non-maximal. $\Lcal \subset \mkD_m \cup \mkD (\omega_X)_m$. Generation: \textbf{\emph{Method 2}} \\ \hline
(79) & $M_4$ & collection from (106) \\ \hline
(80) & $M_5$ & Maximal variety. $\Lcal \subset \mkD_m \cup \mkD (\omega_X)_m$. Generation: \textbf{\emph{Method 2}} \\ \hline
(81) & $J_1$ & Maximal variety. $\Lcal \subset \mkD_m \cup \mkD (\omega_X)_m$. Generation: \textbf{\emph{Method 2}} \\ \hline
(82) & $J_2$ & Maximal variety. $\Lcal \subset \mkD_m \cup \mkD (\omega_X)_m$. Generation: \textbf{\emph{Method 2}} \\ \hline
(83) & $Q_1$ & collection from (108) \\ \hline
(84) & $Q_2$ & collection from (109) \\ \hline
(85) & $Q_3$ & collection from (108) \\ \hline
(86) & $Q_4$ & collection from (110) \\ \hline
(87) & $Q_5$ & collection from (109) \\ \hline
(88) & $Q_6$, $\PP^1 \times \Ecal_1$ & product of smooth toric Fano varieties \\ \hline
(89) & $Q_7$ & collection from (114) \\ \hline
(90) & $Q_8$, $\PP^1 \times \Ecal_2$ & product of smooth toric Fano varieties \\ \hline
(91) & $Q_9$ & collection from (110) \\ \hline
(92) & $Q_{10}$, $S_1 \times S_2$ & product of smooth toric Fano varieties \\ \hline
(93) & $Q_{11}$, $\PP^1 \times \PP^1 \times S_2$ & product of smooth toric Fano varieties \\ \hline
(94) & $Q_{12}$ & collection from (114) \\ \hline
(95) & $Q_{13}$ & collection from (108) \\ \hline
(96) & $Q_{14}$ & collection from (109) \\ \hline
(97) & $Q_{15}$, $\PP^1 \times \Ecal_4$ & product of smooth toric Fano varieties \\ \hline
(98) & $Q_{16}$ & collection from (115) \\ \hline
(99) & $Q_{17}$ & collection from (114) \\ \hline
(100) & $K_1$ & Maximal variety. $\Lcal \subset \mkD_m$. Generation: \textbf{\emph{Method 2}} \\ \hline
(101) & $K_2$ & Maximal variety. $\Lcal \subset \mkD_m$. Generation: \textbf{\emph{Method 2}} \\ \hline
(102) & $K_3$ & collection from (109) \\ \hline
(103) & $K_4$, $\PP^2 \times S_3$ & product of smooth toric Fano varieties \\ \hline
(104) & $R_1$ & Maximal variety. $\Lcal \nsubseteq \mkD_m \cup \mkD (\omega_X)_m$. Generation: \textbf{\emph{Method 2}} \\ \hline
(105) & $R_2$ & Maximal variety. $\Lcal \nsubseteq \mkD_m \cup \mkD (\omega_X)_m$. Generation: \textbf{\emph{Method 1}} \\ \hline
(106) & $R_3$ & Maximal variety. See Remark (\ref{ex:FSECthreadsR3}). Generation: \textbf{\emph{Method 2}} \\ \hline
(107) & & Maximal variety. $\Lcal \subset \mkD_m$. Generation: \textbf{\emph{Method 1}} \\ \hline
(108) & $U_1$ & Maximal variety. $\Lcal = \mkD_m$. Generation: \textbf{\emph{Method 1}} \\ \hline
(109) & $U_2$ & Maximal variety. $\Lcal = \mkD_m$. Generation: \textbf{\emph{Method 2}} \\ \hline
(110) & $U_3$ & Maximal variety. $\Lcal = \mkD_m$. Generation: \textbf{\emph{Method 1}} \\ \hline
(111) & $U_4$, $S_1 \times S_3$ & product of smooth toric Fano varieties \\ \hline
(112) & $U_5$, $\PP^1 \times PP^1 \times S_3$ & product of smooth toric Fano varieties \\ \hline
(113) & $U_6$, $\PP^1 \times \Fcal_2$ & Maximal, product of smooth toric Fano varieties \\ \hline
(114) & $U_7$ & Maximal variety. $\Lcal = \mkD_m$. Generation: \textbf{\emph{Method 2}} \\ \hline
(115) & $U_8$ & Maximal variety. $\Lcal = \mkD_m$. Generation: \textbf{\emph{Method 2}} \\ \hline
(116) & ${\widetilde{V}}^{4}$ & Maximal variety. $\Lcal \nsubseteq \mkD_m \cup \mkD (\omega_X)_m$. Generation: \textbf{\emph{Method 2}} \\ \hline
(117) & $V^{4}$ & Maximal variety. $\Lcal \nsubseteq \mkD_m \cup \mkD (\omega_X)_m$. Generation: \textbf{\emph{Method 2}} \\ \hline
(118) & $S_2\times S_2$ & product of smooth toric Fano varieties \\ \hline
(119) & $S_2\times S_3$ & product of smooth toric Fano varieties \\ \hline
(120) & $S_3\times S_3$ & Maximal, product of smooth toric Fano varieties \\ \hline
(121) & $Z_1$ & collection from (123)\\ \hline
(122) & $Z_2$ & Maximal variety. $\Lcal \subset \mkD_m \cup \mkD (\omega_X)_m$. Generation: \textbf{\emph{Method 1}} \\ \hline
(123) & $W$ & Maximal variety. $\Lcal \subset \mkD_m \cup \mkD (\omega_X)_m$. Generation: \textbf{\emph{Method 1}}  \\ \hline

\end{longtable}

\thispagestyle{empty}
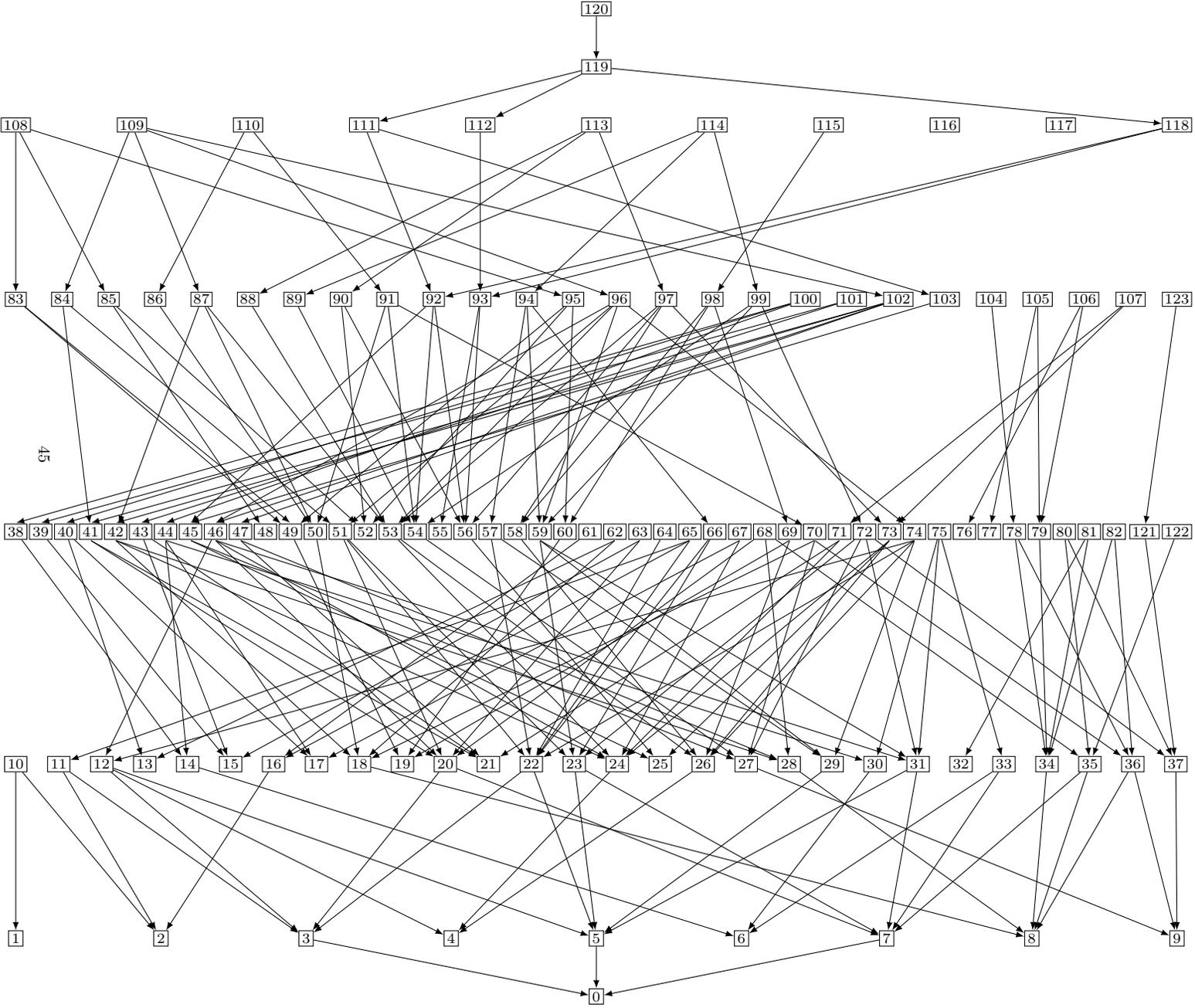
\begin{figure}
\begin{sideways}
\begin{tikzpicture}[every path/.style={>=latex},every node/.style={draw,font =\tiny,inner sep=1.5pt, minimum size=0.1cm}]
  \node             (0) at (0,0)  { $0$ };
   \pgfmathtruncatemacro{\S}{1};
   \pgfmathtruncatemacro{\T}{9};
   \tikzstyle{every node}=[draw,shape=circle];
   \foreach \n in {\S,...,\T}
  {  \pgfmathtruncatemacro{\nn}{\n-\S};
        \node [rectangle,draw,fill=white,inner sep=1.5pt,minimum size=1.8mm] (\n) at ( \nn*2.5-10,1)  {\tiny \n};
    };
   \pgfmathtruncatemacro{\S}{10};
   \pgfmathtruncatemacro{\T}{37};
   \foreach \n in {\S,...,\T}
  {  \pgfmathtruncatemacro{\nn}{\n-\S};
        \node [rectangle,draw,fill=white,inner sep=1.5pt,minimum size=1.8mm] (\n) at ( \nn*0.74 -10,4)  {\tiny \n};
    };
   \pgfmathtruncatemacro{\S}{38};
   \pgfmathtruncatemacro{\T}{82};
   \foreach \n in {\S,...,\T}
  {  \pgfmathtruncatemacro{\nn}{\n-\S};
        \node [rectangle,fill=white,inner sep=1.5pt,minimum size=1.8mm] (\n) at ( \nn*0.43 -10,8)  {\tiny \n};
    };
       \pgfmathtruncatemacro{\S}{83};
   \pgfmathtruncatemacro{\T}{107};
   \foreach \n in {\S,...,\T}
  {  \pgfmathtruncatemacro{\nn}{\n-\S};
        \node [rectangle,draw,fill=white,inner sep=1.2pt,minimum size=1.8mm] (\n) at ( \nn*0.8 -10,12)  {\tiny \n};
    };
       \pgfmathtruncatemacro{\S}{108};
   \pgfmathtruncatemacro{\T}{118};
   \foreach \n in {\S,...,\T}
  {  \pgfmathtruncatemacro{\nn}{\n-\S};
        \node [rectangle,draw,fill=white,inner sep=1.2pt,minimum size=1.8mm] (\n) at ( \nn*2 -10,15)  {\tiny \n};
    };
  \node[rectangle,draw,fill=white,inner sep=1.2pt,minimum size=1.8mm] (119) at (0,16) { \tiny $119$ };
  \node[rectangle,draw,fill=white,inner sep=1.2pt,minimum size=1.8mm] (120) at (0,17)  { \tiny $120$ };
  \node[rectangle,draw,fill=white,inner sep=1.2pt,minimum size=1.8mm] (121) at (9.44,8) { \tiny $121$ };
  \node[rectangle,draw,fill=white,inner sep=1.2pt,minimum size=1.8mm] (122) at (10,8) { \tiny $122$ };
  \node[rectangle,draw,fill=white,inner sep=1.2pt,minimum size=1.8mm] (123) at (10,12) { \tiny $123$ };

\pgfmathtruncatemacro{\S}{38};
   \pgfmathtruncatemacro{\T}{82};
   \foreach \n in {\S,...,\T}
  {  \pgfmathtruncatemacro{\nn}{\n-\S};
        \node [inner sep=0pt,minimum size=0mm] (a\n) at ( \nn*0.43 -10,8.15)  {};
    };
\pgfmathtruncatemacro{\S}{38};
   \pgfmathtruncatemacro{\T}{82};
   \foreach \n in {\S,...,\T}
  {  \pgfmathtruncatemacro{\nn}{\n-\S};
        \node [inner sep=0pt,minimum size=0mm] (b\n) at ( \nn*0.43 -10,7.85)  {};
    };    
    
\draw[->] (3) edge (0);
\draw[->] (5) edge (0);
\draw[->] (7) edge (0);
\draw[->] (10) edge (1);
\draw[->] (10) edge (2);
\draw[->] (11) edge (2);
\draw[->] (11) edge (3);
\draw[->] (12) edge (3);  
\draw[->] (12) edge (4);
\draw[->] (12) edge (5);
\draw[->] (14) edge (6);
\draw[->] (16) edge (2);
\draw[->] (18) edge (8);
\draw[->] (20) edge (3);
\draw[->] (20) edge (7);
\draw[->] (22) edge (3);  
\draw[->] (22) edge (5);
\draw[->] (23) edge (5);
\draw[->] (23) edge (7);
\draw[->] (24) edge (4);
\draw[->] (26) edge (4);
\draw[->] (27) edge (9);
\draw[->] (28) edge (8);
\draw[->] (29) edge (5);  
\draw[->] (30) edge (6);
\draw[->] (31) edge (5);
\draw[->] (31) edge (7);
\draw[->] (33) edge (6);
\draw[->] (33) edge (7);
\draw[->] (34) edge (8);
\draw[->] (35) edge (7);
\draw[->] (35) edge (8); 
\draw[->] (36) edge (8);
\draw[->] (36) edge (9);
\draw[->] (37) edge (9);
\draw[->] (38) edge (14);
\draw[->] (39) edge (15);
\draw[->] (40) edge (13);
\draw[->] (40) edge (17);
\draw[->] (b41) edge (20);  
\draw[->] (b41) edge (21);
\draw[->] (b42) edge (18);
\draw[->] (b42) edge (24);
\draw[->] (b42) edge (28);
\draw[->] (43) edge (15);
\draw[->] (43) edge (21);
\draw[->] (b44) edge (14);
\draw[->] (b44) edge (17);  
\draw[->] (b44) edge (30);
\draw[->] (b45) edge (25);
\draw[->] (b45) edge (27);
\draw[->] (46) edge (12);
\draw[->] (b46) edge (20);
\draw[->] (b46) edge (24);
\draw[->] (b46) edge (31);
\draw[->] (b47) edge (21);  
\draw[->] (b47) edge (28);
\draw[->] (49) edge (19);
\draw[->] (50) edge (18);
\draw[->] (51) edge (20);
\draw[->] (51) edge (22);
\draw[->] (51) edge (23);
\draw[->] (53) edge (24);
\draw[->] (53) edge (26); 
\draw[->] (54) edge (27);
\draw[->] (56) edge (25);
\draw[->] (57) edge (22);
\draw[->] (57) edge (29);
\draw[->] (58) edge (26);
\draw[->] (b59) edge (23);
\draw[->] (b59) edge (29);
\draw[->] (b59) edge (31);  
\draw[->] (60) edge (19);
\draw[->] (61) edge (16);
\draw[->] (b62) edge (13);
\draw[->] (b62) edge (18);
\draw[->] (b63) edge (15);
\draw[->] (b63) edge (20);
\draw[->] (64) edge (22);
\draw[->] (b65) edge (11);  
\draw[->] (b65) edge (16);
\draw[->] (b65) edge (22);
\draw[->] (66) edge (22);
\draw[->] (66) edge (23);
\draw[->] (b67) edge (17);
\draw[->] (b67) edge (24);
\draw[->] (68) edge (20);
\draw[->] (68) edge (28);  
\draw[->] (68) edge (35);
\draw[->] (b69) edge (19);
\draw[->] (b69) edge (26);
\draw[->] (b70) edge (18);
\draw[->] (b70) edge (27);
\draw[->] (b70) edge (36);
\draw[->] (b71) edge (21);
\draw[->] (b71) edge (24); 
\draw[->] (72) edge (23);
\draw[->] (72) edge (27);
\draw[->] (72) edge (31);
\draw[->] (72) edge (37);
\draw[->] (b73) edge (24);
\draw[->] (b73) edge (25);
\draw[->] (b73) edge (26);
\draw[->] (b74) edge (12);  
\draw[->] (b74) edge (22);
\draw[->] (b74) edge (26);
\draw[->] (b74) edge (29);
\draw[->] (75) edge (30);
\draw[->] (75) edge (31);
\draw[->] (75) edge (33);
\draw[->] (78) edge (34);
\draw[->] (78) edge (36);  
\draw[->] (79) edge (34);
\draw[->] (80) edge (35);
\draw[->] (80) edge (37);
\draw[->] (81) edge (32);
\draw[->] (81) edge (34);
\draw[->] (82) edge (34);
\draw[->] (82) edge (36);
\draw[->] (83) edge (49);  
\draw[->] (84) edge (41);
\draw[->] (84) edge (51);
\draw[->] (85) edge (48);
\draw[->] (85) edge (52);
\draw[->] (86) edge (50);
\draw[->] (87) edge (42);
\draw[->] (87) edge (50);
\draw[->] (87) edge (53);
\draw[->] (88) edge (53);
\draw[->] (89) edge (54);
\draw[->] (90) edge (52);
\draw[->] (90) edge (56);
\draw[->] (91) edge (50);
\draw[->] (91) edge (54);
\draw[->] (91) edge (70);
\draw[->] (92) edge (a45);  
\draw[->] (92) edge (54);
\draw[->] (92) edge (56);
\draw[->] (93) edge (55);
\draw[->] (93) edge (56);
\draw[->] (94) edge (57);
\draw[->] (94) edge (59);
\draw[->] (94) edge (66);
\draw[->] (95) edge (49);  
\draw[->] (95) edge (52);
\draw[->] (95) edge (60);
\draw[->] (96) edge (a46);
\draw[->] (96) edge (51);
\draw[->] (96) edge (53);
\draw[->] (96) edge (59);
\draw[->] (96) edge (74);
\draw[->] (97) edge (53);  
\draw[->] (97) edge (56);
\draw[->] (97) edge (58);
\draw[->] (97) edge (73);
\draw[->] (98) edge (58);
\draw[->] (98) edge (60);
\draw[->] (98) edge (69);
\draw[->] (99) edge (54);
\draw[->] (99) edge (59); 
\draw[->] (99) edge (72);
\draw[->] (100) edge (a38);
\draw[->] (100) edge (a40);
\draw[->] (100) edge (a44);
\draw[->] (101) edge (a39);
\draw[->] (101) edge (a43);
\draw[->] (102) edge (a41);
\draw[->] (102) edge (a42);  
\draw[->] (102) edge (a46);
\draw[->] (102) edge (a47);
\draw[->] (103) edge (a45);
\draw[->] (104) edge (78);
\draw[->] (105) edge (77);
\draw[->] (105) edge (79);
\draw[->] (106) edge (76);
\draw[->] (106) edge (79);
\draw[->] (107) edge (71);  
\draw[->] (107) edge (73);
\draw[->] (108) edge (83);
\draw[->] (108) edge (85);
\draw[->] (108) edge (95);
\draw[->] (109) edge (84);
\draw[->] (109) edge (87);
\draw[->] (109) edge (96);
\draw[->] (109) edge (102);  
\draw[->] (110) edge (86);
\draw[->] (110) edge (91);
\draw[->] (111) edge (92);
\draw[->] (111) edge (103);
\draw[->] (112) edge (93);
\draw[->] (113) edge (88);
\draw[->] (113) edge (90);
\draw[->] (113) edge (97);  
\draw[->] (114) edge (89);
\draw[->] (114) edge (94);
\draw[->] (114) edge (99);
\draw[->] (115) edge (98);
\draw[->] (118) edge (92);
\draw[->] (118) edge (93);
\draw[->] (119) edge (111);
\draw[->] (119) edge (112); 
\draw[->] (119) edge (118);
\draw[->] (120) edge (119);
\draw[->] (121) edge (37);
\draw[->] (122) edge (35);
\draw[->] (123) edge (121);

\draw[->] (83) edge (a49);    

\end{tikzpicture}

\end{sideways}
\label{fig:poset}\caption{The torus-invariant divisorial contractions between the smooth toric Fano fourfolds}
\end{figure}
\newpage
 
\section{Appendix B}

\noindent The two examples below prove that the non-nef collections of line bundles on the smooth toric Fano fourfolds $M_1$ and $V^4$ are full strong exceptional, using \emph{\textbf{method 2}}. 
\begin{example}\label{ex:4fold76}
Let $X=M_1$ be the $76^{th}$ smooth toric Fano fourfold. The primitive generators for the rays of $X$ are
\[u_0 = \begin{tiny}
\begin{bmatrix}
1 \\ 0 \\ 0 \\ 0
\end{bmatrix} \end{tiny} ,
u_1 = \begin{tiny}\begin{bmatrix}
0 \\ 1 \\ 0 \\ 0
\end{bmatrix} \end{tiny},
u_2 = \begin{tiny}\begin{bmatrix}
-1 \\ -1 \\ 1 \\ 1
\end{bmatrix} \end{tiny},
u_3 = \begin{tiny}\begin{bmatrix}
0 \\ 0 \\ 1 \\ 0
\end{bmatrix} \end{tiny},
u_4 = \begin{tiny}\begin{bmatrix}
0 \\ 0 \\ -1 \\ 0
\end{bmatrix} \end{tiny},
u_5 = \begin{tiny}\begin{bmatrix}
0 \\ 0 \\ 0 \\ 1
\end{bmatrix} \end{tiny},
u_6 = \begin{tiny}\begin{bmatrix}
0 \\ 0 \\ 0 \\ -1
\end{bmatrix} \end{tiny},
u_7 = \begin{tiny}\begin{bmatrix}
-1 \\ 0 \\ 0 \\ 0
\end{bmatrix} \end{tiny}
\]

\noindent The collection of line bundles on $X$ 
\[ \Lcal = \left\lbrace
\begin{array}{c|}
\Ocal_X(kD_2  +iD_6 +jD_7), \Ocal_X(jD_2+D_4+iD_6+D_7),\\
\Ocal_X((k-1)D_2+ (j-1)D_4 +(1+i-j)D_6), \\
\Ocal_X((k+1)D_4 +(k+1)D_6 +(k+1)D_7)
\end{array}
\begin{array}{c}
\phantom{a} 0 \leq i \leq j \leq 1 \\
0 \leq k \leq 1
\end{array}
\right\rbrace \]
is strong exceptional and contains the non-nef line bundles $\{ \Ocal_X(-D_2+D_6), \Ocal_X(-D_2 +D_4),$ $\Ocal_X(-D_2+D_4+D_6),\Ocal_X(D_7), \Ocal_X(D_6+D_7),\Ocal_X(D_4+D_7),\Ocal_X(D_2)\}$.
We obtain the chain complex of $\Pic(X \times X)-$graded $(S_{X \times X})-$modules
\[0\rightarrow (S_{X \times X})^{10} \stackrel{d_4}{\rightarrow} (S_{X \times X})^{43} \stackrel{d_3}{\rightarrow} (S_{X \times X})^{76} \stackrel{d_2}{\rightarrow} (S_{X \times X})^{60} \stackrel{d_1}{\rightarrow} (S_{X \times X})^{17} \] from this collection, which is exact up to saturation by $B_X$ \cite{Prna1,M2}. The table in Figure \ref{fig:4fold76arrows} lists the arrows in the quiver of sections $Q$ corresponding to $\Lcal$.

\begin{figure}
\begin{small}
\begin{tabular}{|c|c|c| c |c|c|c| c |c|c|c| c |c|c|c|}
\cline{1-3}
a &  $\mathbf{t}(a),\mathbf{h}(a)$  & $\div(a)$ & 
\multicolumn{1}{c}{}  \\ \cline{1-3} \cline{5-7} \cline{9-11} \cline{13-15} 
1&0,1&$x_5$& &16&\phantom{aa}2,10\phantom{aa} &\phantom{a}$x_0$\phantom{a}& &31&\phantom{aa}6,13\phantom{aa}&\phantom{a}$x_0$\phantom{a}& &46&\phantom{aa}11,15\phantom{aa}&\phantom{a}$x_0$\phantom{a} \\ 
\cline{1-3} \cline{5-7} \cline{9-11} \cline{13-15}
2&0,2&$x_3$& &17&3,7&$x_1$& &32&7,9&$x_5$& &47&12,14&$x_3$ \\ 
\cline{1-3} \cline{5-7} \cline{9-11} \cline{13-15}
3&0,3&$x_7$& &18&3,7&$x_2$& &33&7,10&$x_3$& &48&12,15&$x_4$ \\ 
\cline{1-3} \cline{5-7} \cline{9-11} \cline{13-15}
4&0,4&$x_1$& &19&3,9&$x_6$& &34&7,12&$x_6$& &49&13,14&$x_5$ \\ 
\cline{1-3} \cline{5-7} \cline{9-11} \cline{13-15}
5&0,4&$x_2$& &20&3,10&$x_4$& &35&7,13&$x_4$& &50&13,15&$x_6$ \\ 
\cline{1-3} \cline{5-7} \cline{9-11} \cline{13-15}
6&0,5&$x_6$& &21&4,5&$x_5$& &36&8,11&$x_1$& &51&14,15&$x_1$ \\ 
\cline{1-3} \cline{5-7} \cline{9-11} \cline{13-15}
7&0,6&$x_4$& &22&4,6&$x_3$& &37&8,11&$x_2$& &52&14,15&$x_2$ \\ 
\cline{1-3} \cline{5-7} \cline{9-11} \cline{13-15}
8&0,7&$x_0$& &23&4,7&$x_7$& &38&8,14&$x_0$& &53&14,16&$x_0x_4x_5$ \\ 
\cline{1-3} \cline{5-7} \cline{9-11} \cline{13-15}
9&1,5&$x_1$& &24&5,8&$x_3$& &39&9,12&$x_1$& &54&14,16&$x_0x_3x_6$ \\ 
\cline{1-3} \cline{5-7} \cline{9-11} \cline{13-15}
10&1,5&$x_2$& &25&5,9&$x_7$& &40&9,12&$x_2$& &55&14,16&$x_4x_6x_7$ \\ 
\cline{1-3} \cline{5-7} \cline{9-11} \cline{13-15}
11&1,8&$x_4$& &26&5,11&$x_4$& &41&9,14&$x_4$& &56&15,16&$x_0x_3x_5$ \\ 
\cline{1-3} \cline{5-7} \cline{9-11} \cline{13-15}
12&1,9&$x_0$& &27&5,12&$x_0$& &42&10,13&$x_1$& &57&15,16&$x_1x_3x_5x_7$ \\ 
\cline{1-3} \cline{5-7} \cline{9-11} \cline{13-15}
13&2,6&$x_1$& &28&6,8&$x_5$& &43&10,13&$x_2$& &58&15,16&$x_2x_3x_5x_7$ \\ 
\cline{1-3} \cline{5-7} \cline{9-11} \cline{13-15}
14&2,6&$x_2$& &29&6,10&$x_7$& &44&10,14&$x_6$& &59&15,16&$x_4x_5x_7$ \\ 
\cline{1-3} \cline{5-7} \cline{9-11} \cline{13-15}
15&2,8&$x_6$& &30&6,11&$x_6$& &45&11,14&$x_7$& &60&15,16&$x_3x_6x_7$ \\ 
\cline{1-3} \cline{5-7} \cline{9-11} \cline{13-15}
\end{tabular}
\end{small}
\caption{The arrows in a quiver of sections for the smooth toric Fano fourfold $M_1$}
\label{fig:4fold76arrows}
\end{figure}

\noindent As $\vert Q_0 \vert = 17$ and $\vert \Sigma(1) \vert = 8$, we let $\{\mathbf{e}_i \mid i \in Q_0\} \cup \{\mathbf{e}_\rho \mid \rho \in \Sigma(1)\}$ be the standard basis of $\ZZ^{17+8}$ and define the lattice points $c_a := \mathbf{e}_{\mathbf{h}(a)} - \mathbf{e}_{\mathbf{t}(a)} + \mathbf{e}_{\div(a)}$ for each arrow $a \in Q_1$. The map $\pi$ is then given by the matrix $C \colon \ZZ^{60} \rightarrow \ZZ^{17+8}$ where the columns of $C$ are given by $c_a, a \in Q_1$, and the semigroup $\NN(Q)$ is given by the lattice points generated by positive linear combinations of the $c_a$. Our choice of bases for $\Pic(X)$ and $\Wt(Q)$ imply that the lattice maps $\deg$ and $\pic$ are given by the matrices:
\[
\deg \colon \left[\begin{smallmatrix}
1& 1& 1& -1&  0& -1& 0& 0 \\
0& 0& 0& 1&  1& 0& 0& 0 \\
0& 0& 0& 0&  0& 1& 1& 0 \\
1& 0& 0& 0& 0& 0& 0& 1 
\end{smallmatrix} \right], \phantom{a}
\pic \colon \left[ \begin{smallmatrix}
0&-1&-1& 0& 1& 0& 0& 1&-1& 0& 0& 0& 1& 1& 0& 1& 0 \\
0& 0& 1& 0& 0& 0& 1& 0& 1& 0& 1& 1& 0& 1& 1& 1& 2 \\
0& 1& 0& 0& 0& 1& 0& 0& 1& 1& 0& 1& 1& 0& 1& 1& 2 \\
0& 0& 0& 1& 0& 0& 0& 1& 0& 1& 1& 0& 1& 1& 1& 1& 2 
\end{smallmatrix} \right] \]
\noindent Fix $\theta$ to be the weight that assigns $-2$ to the vertex $0$ in the quiver, $1$ to the vertices $15$ and $16$ and $0$ to every other vertex. We note that $\pic(\theta)$ is the ample line bundle $L = \Ocal_X(D_2+3D_4+3D_6+3D_7)$. For this choice of $\theta$, not only does $\pi_2 \left(\NN(Q) \cap (\pi_1)^{-1}(\theta)\right)$ surject onto $\NN^{\Sigma(1)} \cap \deg^{-1}(L)$, but $Y_\theta$ is isomorphic to $X$. As $\theta_i \geq 0$ for $i > 0$, $\theta$ is in the same closed GIT-chamber for $Y_\theta$ as $\vartheta$, and so they are in the same open chamber if $\theta$ is generic. To check that $\theta$ is generic, it is enough to check that for each torus-invariant point on $Y_\theta$, the corresponding representation is $\theta$-stable. There are $17$ maximal cones in the fan for $Y_\theta$ -- recall that each maximal cone corresponds to a torus-invariant point. 

For each quiver that describes a torus-invariant representation in $Y_\theta$, we need to specify a path from the vertex $0$ to vertices $15$ and $16$, and a path from every other vertex to the vertex $15$ or $16$. Examples of these paths are given in Figure \ref{fig:4fold76paths} and are computed in \cite{Prna3}. As a result, every torus-invariant $\theta$-semistable representation of $Q$ is $\theta$-stable, so $\theta$ is generic and the collection $\Lcal$ on $X$ is full by Propositions \ref{prop:NonNefEmbedding}, \ref{prop:EmbeddingGivesCokernel} and \ref{prop:GenerationResolutionDiagonal}.

\begin{figure}
\begin{small}
\begin{tabular}{|c|c|c|}
\hline 
Torus-Invariant & ($0 \rightarrow 15$, via $\{i_1, \ldots ,i_{j_1}\}$), & (vertex $i$, $i \rightarrow 15$ or $i \rightarrow 16$,  \\ 
Point & ($0 \rightarrow 16$, via $\{i_1, \ldots ,i_{j_2}\}$) & via vertices $\{i_1, \ldots ,i_{j_3}\}$)\\
\hline 
$\{\rho_0,\rho_1,\rho_3,\rho_5\}$ & $(a_3a_{18}a_{34}a_{48}, \{3,7,12\})$, & $
       (1, a_{10}a_{25}a_{40}a_{48}, \{5,9,12\}), (2,a_{14}a_{29}a_{43}a_{50}, \{6,10,13\}),$
        \\
       &$(a_{3}a_{19}a_{41}a_{55},\{3,9,14\})$ & $(4,a_{23}a_{34}a_{48}, \{7,12\}), (8,a_{37}a_{45}a_{52}, \{11,14\})$      \\
\hline 
$\{\rho_0,\rho_1,\rho_3,\rho_6\}$ & $(a_1a_{10}a_{25}a_{40}a_{48}a_{59},$   & $(2,a_{14}a_{28}a_{37}a_{45}a_{52},\{6,8,11,14\}), (3,a_{18}a_{32}a_{40}a_{48}, \{7,9,$
        \\
& $\{1,5,9,12,15\})$  & $12\}),(4,a_{21}a_{25}a_{40}a_{48}, \{5,9,12\}), (10,a_{43}a_{49}a_{52}, \{13,14\}) $ \\ 
\hline 
$\{\rho_0,\rho_1,\rho_4,\rho_5\}$ & $(0,a_{2}a_{14}a_{29}a_{43}a_{50}a_{60},$  & $(1,a_{10}a_{24}a_{37}a_{45}a_{52}, \{5,8,11,14\}),
(3,a_{18}a_{33}a_{43}a_{50}, \{7,10,$
        \\
       & $ \{2,6,10,13,15\})$ & $13\}),(4,a_{22}a_{29}a_{43}a_{50}, \{6,10,13\}), (9,a_{40}a_{47}a_{52}, \{12,14\})$ \\
\hline 
$\{\rho_0,\rho_1,\rho_4,\rho_6\}$ & $(0,a_1a_{10}a_{24}a_{37}a_{45}a_{52}a_{58},$  
 & $(2,a_{14}a_{29}a_{43}a_{49}a_{52}, \{6,10,13,14\}), (3,a_{18}a_{32}a_{40}a_{47}a_{52}, $  \\ 
 & $\{1,5,8,11,14,15\})$ & $\{7,9,12,14\}), (4,a_{21}a_{24}a_{37}a_{45}a_{52}), \{5,8,11,14\})$\\
\hline 
$\{\rho_0,\rho_2,\rho_3,\rho_5\}$ & $(0,a_3a_{17}a_{34}a_{48},\{3,7,12\}),$ 
        & $(1,a_9a_{25}a_{39}a_{48},\{5,9,12\}), (2,a_{13}a_{29}a_{42}a_{50},\{6,10,13\}),$
        \\ 
       & $(0,a_3a_{19}a_{41}a_{55},\{3,9,14\})$ & $(4,a_{23}a_{34}a_{48},\{7,12\}), (8,a_{36}a_{45}a_{51},\{11,14\})$ \\
\hline 
$\vdots$ & $\vdots$ & $\vdots$ \\ 
\hline 
$\{\rho_2,\rho_4,\rho_6,\rho_7\}$ & $(0,a_1a_9a_{24}a_{36}a_{46}a_{56},$  & $(2,a_{13}a_{28}a_{36}a_{46},\{6,8,11\}), (3,a_{17}a_{32}a_{39}a_{47}a_{51},\{7,9,12,$ \\
& $ \{1,5,8,11,15\})$ & $14\}),(4,a_{21}a_{24}a_{36}a_{46},\{5,8,11\}), (10,a_{42}a_{49}a_{51},\{13,14\})$ \\ 
\hline 
\end{tabular}
\end{small}
\caption{Paths in the quiver associated to each torus-invariant representation in $Y_\theta \cong M_1$}
\label{fig:4fold76paths}
\end{figure}
\end{example}

\begin{example}\label{ex:4fold117}
Let $X=V^4$ be the $117^{th}$ smooth toric Fano fourfold. The primitive generators are
\[u_0 = \begin{tiny}
\begin{bmatrix}
1 \\ 0 \\ 0 \\ 0
\end{bmatrix} \end{tiny} ,
u_1 = \begin{tiny}\begin{bmatrix}
-1 \\ 0 \\ 0 \\ 0
\end{bmatrix} \end{tiny},
u_2 = \begin{tiny}\begin{bmatrix}
0 \\ 1 \\ 0 \\ 0
\end{bmatrix} \end{tiny},
u_3 = \begin{tiny}\begin{bmatrix}
0 \\ -1 \\ 0 \\ 0
\end{bmatrix} \end{tiny},
u_4 = \begin{tiny}\begin{bmatrix}
0 \\ 0 \\ 1 \\ 0
\end{bmatrix} \end{tiny},
u_5 = \begin{tiny}\begin{bmatrix}
0 \\ 0 \\ -1 \\ 0
\end{bmatrix} \end{tiny},
u_6 = \begin{tiny}\begin{bmatrix}
0 \\ 0 \\ 0 \\ 1
\end{bmatrix} \end{tiny},
u_7 = \begin{tiny}\begin{bmatrix}
0 \\ 0 \\ 0 \\ -1
\end{bmatrix} \end{tiny},\] \[ 
u_8 = \begin{tiny}\begin{bmatrix}
1 \\ 1 \\ 1 \\ 1
\end{bmatrix} \end{tiny}
\]

\noindent The strong exceptional collection of line bundles $\Lcal$ is given by the columns of the matrix $\pic$ below, where we choose the divisors $\{D_1,D_3,D_4,D_6,D_8\}$ as a basis of $\Pic(X)$. This collection contains the non-nef line bundle $\Ocal_X(D_8)$.
We obtain the chain complex of $\Pic(X \times X)-$graded $(S_{X \times X})-$modules
\[0\rightarrow (S_{X \times X})^{18} \stackrel{d_4}{\rightarrow} (S_{X \times X})^{78} \stackrel{d_3}{\rightarrow} (S_{X \times X})^{124} \stackrel{d_2}{\rightarrow} (S_{X \times X})^{87} \stackrel{d_1}{\rightarrow} (S_{X \times X})^{23} \] from this collection, which is exact up to saturation by $B_X$ \cite{Prna1,M2}. The lattice maps $\deg$ and $\pic$ are given by the matrices:
\[
\deg \colon \left[\begin{smallmatrix}
1&  1& 0&  0&  0& 0& 0& 0& 0 \\
0&  0& 1&  1&  0& 0& 0& 0& 0 \\
0&  0& 0&  0&  1& 1& 0& 0& 0 \\
0&  0& 0&  0&  0& 0& 1& 1& 0 \\
-1& 0& -1& 0&  0& 1& 0& 1& 1 \\
\end{smallmatrix}  \right]
, \phantom{a}
\pic \colon \left[ \begin{smallmatrix}
0& 0& 1& 1& 1& 1& 1& 2& 1& 1& 1& 2& 2& 2& 1& 2& 2& 2& 1& 2& 2& 2& 2 \\
0& 0& 1& 1& 1& 1& 2& 1& 1& 2& 2& 1& 1& 2& 2& 1& 2& 2& 2& 1& 2& 2& 2 \\
0& 0& 1& 1& 1& 2& 1& 1& 2& 1& 2& 1& 2& 1& 2& 2& 1& 2& 2& 2& 1& 2& 2 \\
0& 0& 1& 1& 2& 1& 1& 1& 2& 2& 1& 2& 1& 1& 2& 2& 2& 1& 2& 2& 2& 1& 2 \\
0& 1& 0& 1& 1& 1& 0& 0& 2& 1& 1& 1& 1& 0& 1& 1& 0& 0& 2& 2& 1& 1& 0 \\
\end{smallmatrix} \right]\]
\noindent Fix $\theta$ to be the weight that assigns $-9$ to vertex $0$, $1$ to vertices $\{14,\ 15,\ldots , 22\}$ and $0$ to all other vertices. We note that $\pic(\theta)$ is the ample line bundle $L = \Ocal_X(16D_1 + 16D_3+16D_4+16D_6+8D_8)$. For this choice of $\theta$, not only does $\pi_2 \left(\NN(Q) \cap (\pi_1)^{-1}(\theta)\right)$ surject onto $\NN^{\Sigma(1)} \cap \deg^{-1}(L)$, but $Y_\theta$ is isomorphic to $X$. As $\theta_i \geq 0$ for $i > 0$, $\theta$ is in the same closed GIT-chamber for $Y_\theta$ as $\vartheta$ and so they are in the same open chamber if $\theta$ is generic. For each quiver that describes a torus-invariant representation in $Y_\theta$, we need to specify paths from the vertex $0$ to the vertices $\{14,\ 15,\ldots , 22\}$, and a path from every other vertex to one of the vertices in $\{14,\ 15,\ldots , 22\}$ to show that $\theta$ is generic. These paths, as well as the other necessary computations for this example are found in \cite{Prna3}. As a result, every torus-invariant $\theta$-semistable representation of $Q$ is $\theta$-stable, so $\theta$ is generic and the collection $\Lcal$ on $X$ is full by Propositions \ref{prop:NonNefEmbedding}, \ref{prop:EmbeddingGivesCokernel} and \ref{prop:GenerationResolutionDiagonal}.
\end{example}

\bibliography{TiltingBundles}
\bibliographystyle{alpha}
\end{document}